\documentclass[11pt]{amsart}
 \usepackage[totalwidth=440pt, totalheight=640pt]{geometry}
\usepackage{amsmath, amssymb, amsthm, amsfonts, mathrsfs, eucal, enumerate, natbib, layout}
\usepackage[normalem]{ulem}
\usepackage[all,tips,2cell]{xy}
\usepackage[usenames]{color}
\usepackage{xspace}
\usepackage{ifpdf}
\usepackage{ifthen}
\usepackage{hyperref}
\usepackage{tikz}
\usepackage{verbatim}
\usetikzlibrary{matrix,arrows}
\usepackage{mathtools}
\usepackage[usenames]{color}

\newtheorem{theorem}{Theorem}[section]
\newtheorem*{theorem*}{Theorem}
\newtheorem{corollary}[theorem]{Corollary}
\newtheorem*{corollary*}{Corollary}
\newtheorem{lemma}[theorem]{Lemma}

\newtheorem*{claim*}{Claim}
\newtheorem*{lemma*}{Lemma}
\newtheorem{proposition}[theorem]{Proposition}
\newtheorem*{proposition*}{Proposition}
\newtheorem{def-lem}[theorem]{Definition-Lemma}
\newtheorem{def-prop}[theorem]{Definition-Proposition}
\theoremstyle{definition}
\newtheorem{definition}[theorem]{Definition}
\newtheorem*{definition*}{Definition}

\newtheorem*{example*}{Example}
\numberwithin{equation}{section}

\newcommand{\id}{\mathrm{id}}
\newcommand{\ass}{\mathsf{a}}
\newcommand{\st}{\mathrm{st}}

\newcommand{\eext}{\mathrm{Ext}}
\newcommand{\Ext}{\mathcal{E}\mathrm{xt}}
\newcommand{\pr}{\mathrm{pr}}

\newcommand{\Hom}{\mathbb{H}\mathrm{om}}

\newcommand{\HomTORS}{\mathbb{H}\mathrm{om}_{\textsc{Tors}(\mathbb{G})}}

\newcommand{\hhom}{\mathrm{Hom}}
\newcommand{\HH}{\mathrm{H}}
\newcommand{\MC}{\mathrm{MC}}
\newcommand{\alman}[1]{\mathfrak{#1}}
\newcommand{\gotika}{\alman a}
\newcommand{\gotikc}{\alman c}
\newcommand{\gotikd}{\alman d}
\newcommand{\gotikl}{\alman l}
\newcommand{\gotiks}{\alman s}
\newcommand{\gotikp}{\alman p}
\newcommand{\gotikL}{\alman L}
\newcommand{\gotikr}{\alman r}
\newcommand{\gotikR}{\alman R}
\newcommand{\gotikh}{\alman h}

\newcommand{\invass}{\mathsf{a}^{-1}}
\newcommand{\comm}{\mathsf{c}}

\newcommand{\TORS}{\textsc{Tors}}

\newcommand{\bGerbe}{\mathbb{G}\mathrm{erbe}}
\newcommand{\PICARD}{\mathcal{P}\textsc{icard}(\mathbf{S})}

\newcommand{\FPICARD}{\mathcal{P}\textsc{icard}^{\flat \flat}(\mathbf{S})}

\newcommand{\Perm}{\mathrm{Perm}}

\DeclareMathOperator{\fcirc}{{\diamond}}
\DeclareMathOperator{\ra}{{\rightarrow}}
\DeclareMathOperator{\la}{{\leftarrow}}
\DeclareMathOperator{\lra}{{\longrightarrow}}
\DeclareMathOperator{\Ra}{{\Rightarrow}}
\DeclareMathOperator{\car}{{\circlearrowright}}
\DeclareMathOperator{\Rra}{{\Rrightarrow}}
\DeclareMathOperator{\Lla}{{\Lleftarrow}}
\DeclareMathOperator{\La}{{\Leftarrow}}

\DeclareMathOperator{\Da}{{\Downarrow}}
\DeclareMathOperator{\Ua}{{\Uparrow}}

\newcommand{\boyut}[1]{\mathscr{#1}}

\newcommand{\oneG}{\boyut G}

\newcommand{\ES}{\mathbf{S}}

\newcommand{\Aut}{\mathscr{A}\mspace{-4mu}\mathit{ut}}

\newcommand{\twocat}[1]{\mathcal{#1}^{[-2,0]}(\mathbf{S})}
\newcommand{\twoA}{\mathbb{A}}
\newcommand{\twoB}{\mathbb{B}}

\newcommand{\twoD}{\mathbb{D}}
\newcommand{\twoE}{\mathbb{E}}

\newcommand{\twoG}{\mathbb{G}}
\newcommand{\twoK}{\mathbb{K}}
\newcommand{\twoI}{\mathbb{I}}

\newcommand{\twoL}{\mathbb{L}}

\newcommand{\twoKer}{\mathbb{K}\mathrm{er}}
\newcommand{\twoCo}{\mathbb{C}\mathrm{oker}}
\newcommand{\twoP}{\mathbb{P}}
\newcommand{\twoQ}{\mathbb{Q}}
\newcommand{\twoR}{\mathbb{R}}
\newcommand{\twoS}{\mathbb{S}}
\newcommand{\twoT}{\mathbb{T}}
\newcommand{\twocatD}{\twocat D}
\newcommand{\twocatK}{\twocat K}
\newcommand{\twocatT}{\twocat T}

\newcommand{\Tot}{\mathrm{Tot}}
\newcommand{\der}{\mathcal{D}(\mathbf{S})}

\newcommand{\kesir}[2]{\genfrac{}{}{0pt}{}{#1}{#2}}
\xyoption{2cell}

\begin{document}
\UseAllTwocells
\title{Higher dimensional study of extensions via torsors}
\author{Cristiana Bertolin}
\address{Dipartimento di Matematica, Universit\`a di Torino, Via Carlo Alberto 10, 
Italy}
\email{cristiana.bertolin@googlemail.com}

\author{Ahmet Emin Tatar}
\address{Department of Mathematics and Statistics, KFUPM, Dhahran, KSA}
\email{atatar@kfupm.edu.sa}


\date{}

%----------------------------------------------------------------------------------------------------------

\begin{abstract}
Let $\ES$ be a site. First we define the 3-category of torsors under a Picard $\ES$-2-stack and we compute its homotopy groups.  Using calculus of fractions we define also a pure algebraic analogue of the 3-category of torsors under a Picard $\ES$-2-stack.
Then we describe extensions of Picard $\ES$-2-stacks as torsors endowed with a group law on the fibers. As a consequence of such a description, we show that any Picard $\ES$-2-stack admits a canonical free partial left resolution that we compute explicitly. Moreover we get an explicit right resolution of the 3-category of extensions of Picard $\ES$-2-stacks in terms of 3-categories of torsors. Using 
the homological interpretation of Picard $\ES$-2-stacks, we rewrite this three categorical dimensions higher right resolution in the derived category $\der$ of abelian sheaves on $\ES$.
\end{abstract}

%------------------------------------------------------------------------------------------------------

\maketitle
\tableofcontents

%-------------------------------------------------------------------------------------------------
\section*{Introduction}

Let $\ES$ be a site. Picard $\ES$-2-stacks might be succinctly described as the 2-categorical analogue of abelian groups within the context of stacks. Thus they are to be thought of as a generalization of an abelian sheaf on $\ES$, but two categorical dimensions higher. This paper studies Picard $\ES$-2-stacks as part of the larger program of translating between algebro-geometric information and categorical information.
Picard $\ES$-2-stacks reside on the categorical side, while the derived category of abelian sheaves on $\ES$ with cohomology in the range $[-2,0]$ resides on the algebro-geometric side.
In \cite{beta} we have introduced and studied extensions of Picard $\ES$-2-stacks which resides on the categorical side, and we have computed the homological interpretation of such extensions (see \cite[Thm. 1.1]{beta}) which resides on the algebro-geometric side.
In this paper we introduce and study torsors under Picard $\ES$-2-stacks which resides on the categorical side, and we compute the homological interpretation of such torsors (see \ref{thm:introtor}) which resides on the algebro-geometric side. This result on torsors under Picard $\ES$-2-stacks allows us to obtain the two categorical dimensions higher generalization of Grothendieck's study of extensions via torsors done in \cite{SGA7}.
In this setting of translating between algebro-geometric information and categorical information we can cite also the
paper \cite[p. 64]{Mumford65} where Mumford introduced the notion of invertible sheaves on a $\ES$-stack 
(categorical side) and the paper \cite[Prop. 2.1.2]{Brochard09} where Brochard computed the homological interpretation of such invertible sheaves (algebro-geometric side).

Before to describe more in details the results of this paper we recall the notion of gr-$\ES$-2-stack and of 
Picard $\ES$-2-stack. A \emph{gr-$\ES$-2-stack} $\twoG=(\twoG,\otimes,\ass,\pi)$ is an $\ES$-2-stack in 2-groupoids $\twoG$ equipped with a morphism of $\ES$-2-stacks $ \otimes: \twoG \times \twoG \ra \twoG$, called the group law of $\twoG$, with a natural 2-transformation of $\ES$-2-stacks $\ass$, called the \emph{associativity}, which  expresses the associativity constraint of the group law $ \otimes$ of $\twoG$, and with a modification of $\ES$-2-stacks $\pi$ which expresses the obstruction to the coherence of the associativity $\ass$ (i.e. the obstruction to the pentagonal axiom) and which satisfies the coherence axiom of Stasheff's polytope (see (\ref{diagram:stasheff}) or \cite[\S4]{MR1191733} for more details). Moreover we require that for any object $ X $ of $\twoG(U)$ with $U$ an object of $\ES$, the  morphisms of $\ES$-2-stacks $X \otimes - :\twoG \rightarrow \twoG$ and $- \otimes  X:\twoG \rightarrow \twoG$, called respectively the left and the right multiplications by $X$, are equivalences of $\ES$-2-stack. 

A \emph{strict Picard $\ES$-2-stack} (just called Picard $\ES$-2-stack) $\twoP=(\twoP,\otimes,\ass,\pi,\comm,\zeta,\gotikh_1,\gotikh_2,\eta)$ is a gr-$\ES$-2-stack $(\twoP,\otimes,\ass,\pi)$ equipped with a natural 2-transformation of $\ES$-2-stacks $\comm$, called the \emph{braiding}, which expresses the commutativity constraint for the group law $\otimes$ of $\twoP$, with a modification of $\ES$-2-stacks $\zeta$ which expresses the obstruction to the coherence of the braiding $\comm$, with two modifications of $\ES$-2-stacks $\gotikh_1,\gotikh_2$ which express the obstruction to the compatibility between $\ass$ and $\comm$ (i.e. the obstruction to the hexagonal axiom), and finally with a modification of $\ES$-2-stacks $\eta$ which expresses the obstruction to the strictness of the braiding $\comm$. We require also that the modifications $\zeta,\gotikh_1,\gotikh_2$ and $\eta$ satisfy some compatibility conditions. Picard 2-stacks form a 3-category $2\PICARD$ whose hom-2-groupoid consists of additive 2-functors, morphisms of additive 2-functors and modifications of morphisms of additive 2-functors.

Picard $\ES$-2-stacks are the categorical analogue of length 3 complexes of abelian sheaves over $\ES$. In fact in \cite{MR2735751}, it is proven the existence of an equivalence of categories
\begin{equation}\label{intro:2st_flat_flat}
2\st^{\flat\flat} \colon \xymatrix@1{\twocatD \ar[r] & 2\FPICARD}
\end{equation} 
where $\twocatD$ is the full subcategory of the derived category $\der$ of complexes of abelian sheaves over $\ES$ such that $\mathrm{H}^{-i}(A) \neq 0$ for $i=0,1,2$, and $2\FPICARD$ is the category of Picard 2-stacks whose objects are Picard 2-stacks and whose arrows are equivalence classes of additive 2-functors. We denote by $[\,\,]^{\flat\flat}$ the inverse equivalence of $2\st^{\flat\flat}$.

Let $\twoG$ be a gr-$\ES$-2-stack. A \emph{right $\twoG$-torsor} $\twoP=(\twoP,m,\mu,\Theta)$ is an $\ES$-2-stack in 2-groupoids $\twoP$ equipped with a morphism of $\ES$-2-stacks $ m: \twoP \times \twoG \ra \twoP$, called the action of $\twoG$ on $\twoP$, with a natural 2-transformation of $\ES$-2-stacks $\mu$ which expresses the compatibility between the action $m$ and the group law of $\twoG$, with a modification of $\ES$-2-stacks $\Theta$ which expresses the obstruction to the compatibility between $\mu$ and the associativity $\ass$ underlying $\twoG$ (i.e. the obstruction to the pentagonal axiom) and which satisfies the coherence axiom of  Stasheff's polytope. Moreover we require that $\twoP$ is locally equivalent to $\twoG$ and also that $\twoP$ is locally not empty. If $\twoG$ acts on the left side, we get the notion of left $\twoG$-torsor. A \emph{ $\twoG$-torsor} $\twoP=(\twoP,m^l,m^r,\mu^l,\mu^r,\Theta^l,\Theta^r, \kappa, \Omega^r,\Omega^l)$ is an ${\ES}$-2-stack in 2-groupoids $\twoP$ endowed with a structure of left $\twoG$-torsor $(\twoP,m^l,\mu^l,\Theta^l)$, with a structure of right $\twoG$-torsor $(\twoP,m^r,\mu^r,\Theta^r)$, with a natural 2-transformation $\kappa$ which expresses the compatibility between the left and the right action of $\twoG$ on $\twoP$, and finally with two modification of $\ES$-2-stacks $\Omega^l$ and $\Omega^r$ which express the obstruction to the compatibility between the natural 2-transformation $\kappa$ and the natural 2-transformations $\mu^l$ and $\mu^r$ respectively. We require also that the two modification $\Omega^l$ and $\Omega^r$ satisfy some compatibility conditions. $\twoG$-torsors build a 3-category $\TORS(\twoG)$ whose objects are $\twoG$-torsors and whose hom-2-groupoid $\HomTORS(\twoP,\twoQ)$ of morphisms of $\twoG$-torsors between two $\twoG$-torsors is defined in Definitions \ref{def:1_morphisms_of_torsors}, \ref{def:2_morphisms_of_torsors}, \ref{def:3_morphisms_of_torsors}, and \ref{def-prop:2-groupoid_of_torsors}.

Using regular morphisms of length 3 complexes of abelian sheaves it is not possible to obtain all additive 2-functors between Picard 2-stacks. In order to get all of them, in \cite{MR2735751} the second author introduces the tricategory $\twocatT$ of length 3 complexes of abelian sheaves over $\ES$, in which arrows between length 3 complexes are fractions, and he shows that there is a triequivalence 
\begin{equation}\label{intro:2st}
2\st \colon \xymatrix@1{\twocatT \ar[r] & 2\PICARD ,}
\end{equation}
between the tricategory $\twocatT$ and the 3-category $2\PICARD$ of Picard 2-stacks. At the end of section 3 we sketch the definition of $G$-torsor with $G$ a length 3 complex of the tricategory $\twocatT$ (Def. \ref{def:right_torsor_via_fractions}). These $G$-torsors build a tricategory $\TORS(G)$ which is the pure algebraic analogue of the 3-category $\TORS(\twoG)$ of $\twoG$-torsors (Prop. \ref{proposition:triequivalence_of_extensions}).

From now on we assume $\twoG$ to be a Picard $\ES$-2-stack. The hom-2-groupoid $\HomTORS(\twoP,\twoP)$ of morphisms of $\twoG$-torsors from a $\twoG$-torsor $\twoP$ to itself is endowed with a Picard $\ES$-2-stack structure (Lem. \ref{lemma:HomTors(PP)isPicard}) and so its homotopy groups $\pi_{i}(\HomTORS(\twoP,\twoP))$ (for $i=0,1,2$) are abelian groups. We define 
\begin{itemize}
\item $\TORS^1(\twoG)$ is the group of equivalence classes of $\twoG$-torsors: its abelian group law is furnished by the contracted product of $\twoG$-torsors (Def. \ref{def:sum_G_torsors}).
\item $\TORS^i(\twoG)$ (for $i=0,-1,-2$) is the homotopy group $\pi_{-i}(\HomTORS(\twoP,\twoP))$ for any 
$\twoG$-torsor $\twoP$.
\end{itemize}

If $K$ is a complex of abelian sheaves over $\ES$, we denote by $\HH^i(K)$ the $i$-th cohomology group 
$\HH^i\big({\mathrm{R}}\Gamma(K)\big)$ of the derived functor of the functor of global sections applied to $K$. With these notation, we can finally state our first Theorem, which furnishes a parametrization of the elements of $\TORS^i(\twoG)$ by the $i$-th cohomology group $ \HH^i([\twoG]^{\flat\flat})$, and a categorical description of the elements of $\HH^i(K)$, with $K$ a length 3 complex of abelian sheaves, via torsors under Picard $\ES$-2-stacks.

\begin{theorem}\label{thm:introtor}
Let $\twoG$ be a Picard $\ES$-2-stack. Then we have the following isomorphisms  
\[\TORS^i(\twoG) \cong \HH^i([\twoG]^{\flat\flat}) \qquad for \; i=1,0,-1,-2.\]
\end{theorem}

Gr-$\ES$-3-stacks are not defined yet. Assuming their existence, the contracted product of $\twoG$-torsors, which equips the set $\TORS^1(\twoG)$ of equivalence classes of $\twoG$-torsors with an abelian group law, should define a structure of gr-$\ES$-3-stack on the 3-category $\TORS(\twoG)$. In this setting our Theorem \ref{thm:introtor} says that the 3-category $\TORS(\twoG)$ of $\twoG$-torsors should be actually the gr-$\ES$-3-stack associated to the object of $\mathcal{D}^{[-3,0]}(\ES)$
\[\tau_{\leq 0} {\mathrm{R}}\Gamma([\twoG]^{\flat\flat}[1])\]
via the generalization of the equivalence $2\st^{\flat\flat}$ (\ref{intro:2st_flat_flat}) to gr-$\ES$-3-stacks and to length 4 complexes of sheaves of sets on $\ES$ (here $\tau_{\leq 0}$ is the good truncation in degree 0). Moreover, in order to define the groups $\TORS^i(\twoG)$ we could use the homotopy groups $\pi_{i}$ of the gr-$\ES$-3-stack $\TORS(\twoG)$: in fact 
$\TORS^i(\twoG)=\pi_{-i+1}(\TORS(\twoG))$  for $ i=1,0,-1,-2.$

If $\twoP$ and $\twoG$ are two Picard $\ES$-2-stacks, an extension $(\twoE,I,J,\varepsilon)$ of $\twoP$ by $\twoG$ consists of a Picard $\ES$-2-stack $\twoE$, two additive 2-functors $I:\twoG \ra \twoE$ and $J:\twoE \ra \twoP$, and  a morphism of additive 2-functors $\varepsilon:J \circ I \Rightarrow 0$, such that the following equivalent conditions are satisfied:
\begin{itemize}
	\item $\pi_0(J): \pi_0(\twoE) \ra \pi_0(\twoP)$ is surjective and $I$ induces an equivalence of Picard $\ES$-2-stacks between $\twoG$ and $\twoKer(J),$ 
	\item $\pi_2(I): \pi_2(\twoG) \ra \pi_2(\twoE)$ is injective and $J$ induces an equivalence of Picard $\ES$-2-stacks between $\twoCo(I)$ and $\twoP$.
\end{itemize}
In \cite{beta} we have proved that extensions of $\twoP$ by $\twoG$ form a 3-category $\Ext(\twoP,\twoG)$ and we have computed the homotopy groups $\pi_{i}(\Ext(\twoP,\twoG))$ for $i=0,1,2,3$. In this paper, we  describe extensions of Picard $\ES$-2-stacks in terms of torsors under Picard $\ES$-2-stacks.  
We start with a special case of extensions, which involve a Picard $\ES$-2-stack generated by an $\ES$-2-stack in 2-groupoids (see Def. \ref{definition:Z[P]}), and whose description in terms of torsors is a direct consequence of Theorem \ref{thm:introtor}:

\begin{corollary}\label{corollary:Z[I]exttor}
Let $\twoG$ be a Picard $\ES$-2-stacks. Consider a gr-$\ES$-2-stack $\twoP$, associated to a length 3 complex of sheaves of groups on $\ES$, and the Picard $\ES$-2-stack $\mathbb{Z}[\twoP]$ generated it. We have the following tri-equivalence of 3-categories 
\[\Ext(\mathbb{Z}[\twoP],\twoG) \cong \TORS(\twoG_\twoP) \]
where $\TORS(\twoG_\twoP)$ denotes the 3-category of $\twoG_\twoP$-torsors over $\twoP$ (see Def. \ref{def:torsorover2stack}).
\end{corollary}

Now, for the general case, if $\twoP$ and $\twoG$ are two Picard $\ES$-2-stacks, we find an explicit description of extensions of $\twoP$ by $\twoG$ in terms of $\twoG_\twoP$-torsors over $\twoP$ which are endowed with an abelian group law on the fibers. More precisely, it exists a tri-equivalence of 3-categories between the 3-category $\Ext(\twoP,\twoG)$ and the 3-category consisting of the data $(\twoE,M,\alpha,\gotika,\chi,\gotiks,\gotikc_1,\gotikc_2)$, where $\twoE$ is a $\twoG_\twoP$-torsors over $\twoP$, $M : p_1^* \; \twoE \wedge p_2^* \; \twoE \ra \otimes^* \; \twoE$ is a morphism of $\twoG_{\twoP^2}$-torsors over $\twoP \times \twoP$ defining a group law on the fibers of $\twoE$ (here $\otimes$ is the group law of $\twoP$ and $p_i: \twoP \times \twoP \ra \twoP$ are the projections), $\alpha$ is a 2-morphism of $\twoG_{\twoP^3}$-torsors expressing the associativity constraint of this group law defined by $M$, $\chi$ is a 2-morphism of $\twoG_{\twoP^2}$-torsors expressing the braiding constraint of this group law defined by $M$, and finally $\gotika,\gotiks, \gotikc_1,\gotikc_2$ are 3-morphisms of $\twoG_{\twoP^i}$-torsors (with $i=4,2,3,3$ respectively) expressing respectively the obstruction to the coherence of $\alpha$, the obstruction to the coherence of $\chi$, and the obstruction to the compatibility between $\alpha$ and $\chi$. We require also that these 3-morphisms of $\twoG_{\twoP^i}$-torsors satisfy some coherence and compatibility conditions. Summarizing we have 

\begin{theorem}\label{thm:introexttor}
Let $\twoP$ and $\twoG$ be two Picard $\ES$-2-stacks. Then we have the following tri-equivalence of 3-categories
\begin{equation*}
 \Ext(\twoP,\twoG)\simeq \left\{\begin{tabular}{l}$(\twoE,M,\alpha,\gotika,\chi,\gotiks,\gotikc_1,\gotikc_2) ~~ \Big|  ~~ \twoE=\twoG_\twoP-\mathrm{torsor~over~} \twoP, $  \\
 $M : p_1^* \; \twoE \wedge p_2^* \; \twoE \ra \otimes^* \; \twoE, \alpha,\gotika,\chi,\gotiks,\gotikc_1,\gotikc_2 \mathrm{~ described~ in ~Prop. \ref{thm:ext-tor}}$ 
 \end{tabular} \right\}
\end{equation*}
\end{theorem}

This Theorem generalizes to Picard $\ES$-2-stacks the following result of Grothendieck in \cite[Expos\'e VII 1.1.6 and 1.2]{SGA7}: if $P$ and $G$ are two abelian sheaves, to have an extension of $P$ by $G$ is the same thing as to have a $G_P$-torsor $E$ over $P$, and an isomorphism $ pr_1^*E ~ pr_2^*E \ra +^* E$ of $G_{P^2}$-torsors over $P\times P$ satisfying some associativity and commutativity constraints.

As a consequence of the description of extensions of Picard $\ES$-2-stacks in terms of torsors (Cor. \ref{corollary:Z[I]exttor} and Thm. \ref{thm:introexttor}), we have

\begin{corollary}\label{corollary:resolutionP}
Any Picard $\ES$-2-stack $\twoP$ admits as canonical free partial left resolution in the category $2\FPICARD$ the following complex of Picard $\ES$-2-stack:
\begin{equation*}
\twoL.(\twoP) : \qquad 0 \lra \twoL^5(\twoP)\stackrel{D_5}{\lra}\twoL^4(\twoP) \stackrel{D_4}{\lra} \twoL^3(\twoP) \stackrel{D_3}{\lra} \twoL^2(\twoP) \stackrel{D_2}{\lra} \twoL^1(\twoP) \lra 0
\end{equation*}
with
\begin{align*}
\twoL^1(\twoP)&={\mathbb{Z}}[\twoP];\\ 
\twoL^2(\twoP)&={\mathbb{Z}}[\twoP^2];\\ 
\twoL^3(\twoP)&={\mathbb{Z}}[\twoP^3] \oplus {\mathbb{Z}}[\twoP^2];\\
\twoL^4(\twoP)&={\mathbb{Z}}[\twoP^4]\oplus{\mathbb{Z}}[\twoP^3]\oplus{\mathbb{Z}}[\twoP^3]\oplus{\mathbb{Z}}[\twoP^2]\oplus{\mathbb{Z}}[\twoP];\\
\twoL^5(\twoP)&={\mathbb{Z}}[\twoP^5]\oplus{\mathbb{Z}}[\twoP^4]\oplus{\mathbb{Z}}[\twoP^4]\oplus{\mathbb{Z}}[\twoP^4]\oplus{\mathbb{Z}}[\twoP^3]\oplus{\mathbb{Z}}[\twoP^3]\oplus{\mathbb{Z}}[\twoP^3]\oplus{\mathbb{Z}}[\twoP^2]\oplus{\mathbb{Z}}[\twoP] \oplus {\mathbb{Z}}[\twoP^2];
\end{align*}
in degrees 1,2,3,4, and 5 respectively, and with the differential operators defined by
\begin{align}\label{complex_L(P)}
D_2[p |_1 q] &= [p+q] -[p]-[q];\\[0.1cm]
\nonumber D_3[p |_2 q] &= [p |_1 q] -[q |_1 p];\\[0.1cm]
\nonumber D_3[p |_1 q |_1r] &= [p+q |_1 r] -[p |_1 q+r]+[p |_1 q]-[q |_1 r];\\[0.1cm]
\nonumber D_4[p |_1 q |_1 r |_1 s] &= [p |_1 q |_1 r] +[p |_1 q+r |_1 s]+[q |_1 r |_1 s] -[p+q |_1 r |_1 s] -[p |_1 q |_1 r+s ];\\[0.1cm]
\nonumber D_4[p |_2 q |_1 r] &= [q |_1 r |_1 p] +[p |_2 q+r]+[p |_1 q |_1 r]-[q |_1 p |_1 r]-[p |_2 q]-[p |_2 r];\\[0.1cm]
\nonumber D_4[p |_1 q |_2 r] &= [p |_1 r |_1 q] +[p+q |_2 r]-[p  |_1 q |_1 r]-[r |_1 p |_1 q]-[p |_2 r]-[q |_2 r];\\[0.1cm]
\nonumber D_4[p |_3 q]&=-[p |_2 q]-[q |_2 p];\\[0.1cm]
\nonumber D_4[p]&=-[p |_2 p];\\[0.1cm]
\nonumber D_5[p |_1 q |_1 r |_1 s |_1 t] &=[q |_1 r |_1 s |_1 t]+[p |_1 q+r |_1 s |_1 t]+[p |_1 q |_1 r |_1 s+t] -[p |_1 q |_1 r+s |_1 t] \\[0.1cm]
\nonumber &-[p |_1 q |_1 r |_1 s]-[p+q |_1 r |_1 s |_1 t];\\[0.1cm]
\nonumber D_5[p |_2 q |_1 r |_1 s] &=[p |_1 q |_1 r |_1 s]+[p |_2 q |_1 r+s]+[p |_2 r |_1 s] -[q |_1 p |_1 r |_1 s]- [p |_2 q+r |_1 s] \\[0.1cm]
\nonumber &-[q |_1 r |_1 s |_1 p]+[q |_1 r |_1 p |_1 s]-[p |_2 q |_1 r];\\[0.1cm]
\nonumber D_5[p |_1 q |_1 r |_2 s ] &= -[p |_1 q |_1 r |_1 s]+[p+q  |_1 r |_2 s]+[p |_1 q |_1 s |_1 r]+[p |_1 q |_2 s]+[s |_1 p |_1 q |_1 r]\\[0.1cm]
\nonumber &-[p |_1 q+r |_2 s]-[p |_1 s |_1 q |_1 r]-[q |_1 r |_2 s];\\[0.1cm]
\nonumber D_5[p |_1 q |_2 r |_1 s ] &= [p+q |_2 r |_1 s]-[p |_2 r |_1 s]-[q |_2 r |_1 s]-[p |_1 q |_2 r+s]+[p |_1 q |_2 r]+[p |_1 q |_2 s]\\[0.1cm]
\nonumber & +[p |_1 q |_1 r |_1 s]+[p |_1 r |_1 s |_1 q]+[r |_1 s |_1 p |_1 q]+[r |_1 p |_1 q |_1 s]-[p |_1 r |_1 q |_1 s]\\[0.1cm]
\nonumber & -[r |_1 p |_1 s |_1 q];\\[0.1cm]
\nonumber D_5[p  |_3 q  |_1 r]&=[p  |_3 q+r]+[p  |_2 q  |_1 r]+[q  |_1 r  |_2 p]-[p  |_3 r]-[p  |_3 q];\\[0.1cm]
\nonumber D_5[p  |_1 q  |_3 r]&=[p+q  |_3 r]+[p  |_1 q  |_2 r]+[r  |_2 p  |_1 q]-[p  |_3 r]-[q  |_3 r];\\[0.1cm]
\nonumber D_5[p  |_2 q  |_2 r]&=[p  |_2 q  |_1 r]-[p  |_2 r  |_1 q]+[p  |_1 q  |_2 r]-[q  |_1 p  |_2 r];\\[0.1cm]
\nonumber D_5[p |_4 q]&=[p  |_3 q]-[q  |_3 p];\\[0.1cm]
\nonumber D_5[p]&=[p]+[p]-[p  |_3 p];\\[0.1cm]
\nonumber D_5[p |^4 q]&=-[p |_1 q |_1 p |_1 q]+[p |_1 q |_2 p+q] + [p |_2 p |_1 q]+[q |_2 p |_1 q] - [q |_3 p] + [p + q]-[p]-[q].
\end{align}
The augmentation map is given by the additive 2-functor $\epsilon : {\mathbb{Z}}[\twoP] \ra \twoP, \epsilon([p])=p,$ for any $p \in \twoP$.
\end{corollary}

In the above Corollary, adopting Eilenberg-MacLane's bar notation, we give an explicit definition of the differential operators $D_i$ in terms of objects. Their definitions on 1- and 2-arrows are formally identical to the ones on the objects because of the peculiar nature of the free Picard $\ES$-2-stacks involved in $\twoL.(\twoP)$. We find the explicit definitions of the differentials by translating the data underlying the notion of Picard $\ES$-2-stack and also the constraints that those data have to satisfy: $D_2$ corresponds to the group law $\otimes$ underlying $\twoP$, $D_3[p |_2 q]$ corresponds to the braiding $\comm$, $D_3[p |_1 q  |_1 r]$ corresponds to the associativity $\ass$, $D_4[p  |_1 q  |_1 r  |_1 s]$ corresponds to the modification of $\ES$-2-stacks $\pi$ (\ref{ObstructionCoherenceAss}) which expresses the obstruction to the coherence of the associativity $\ass$ (i.e. the obstruction to the pentagonal axiom), $D_4[p  |_2 q  |_1 r]$ and
$D_4[p  |_1 q  |_2 r]$ correspond respectively to the modifications $\gotikh_1$ and $\gotikh_2$ (\ref{ObstructionCompatibilityAssComm}) which expresses the obstruction to the compatibility between $\ass$ and $\comm$ (i.e. the obstruction to the hexagonal axiom), $D_4[p  |_3 q]$ corresponds to the modification $\zeta$ (\ref{ObstructionCoherenceComm}) which expresses the obstruction to the coherence of the braiding $\comm$, $D_4[p]$ corresponds to the modification $\eta$ (\ref{Strictness}) which expresses the obstruction to the strictness of $\comm$, $D_5[p  |_1 q  |_1 r  |_1 s  |_1 t]$ corresponds to the Stasheff's polytope (\ref{diagram:stasheff}) which expresses the coherence of the modification $\pi$, $D_5[p  |_2 q  |_1 r  |_1 s]$ and $D_5[p  |_1 q  |_1 r  |_2 s]$ correspond respectively to the diagrams (\ref{diagram:hexagone1}), (\ref{diagram:hexagone2}) which express the compatibility of the modifications $\gotikh_1$ and $\gotikh_2$ with the modification $\pi$, $D_5[p  |_1 q  |_2 r  |_1 s]$ corresponds to the equality of the diagrams (\ref{diagram:hexagone3}) and (\ref{diagram:hexagone3bis}) which expresses the comparability of the modifications $\gotikh_1$ and $\gotikh_2$, $D_5[p |_3 q |_1 r]$ and $D_5[p |_1 q |_3 r]$ correspond respectively to the diagrams (\ref{diagram:hexagone1_vs_hexagone2}) and (\ref{diagram:hexagone1_vs_hexagone2_bis}) which express the compatibility between $\gotikh_1$ and $\gotikh_2$ under the above comparison, $D_5[p |_2 q |_2 r]$ corresponds to the diagram (\ref{diagram:Z_system}) which expresses the compatibility of Z-systems, $D_5[p |_4 q]$ corresponds to the equation of 2-arrow (\ref{CoherenceZeta}) which expresses the coherence of $\zeta$, $D_5[p]$ corresponds to the relation $\eta*\eta=\zeta$, and finally $D_5[p|^4q]$ corresponds to the diagram (\ref{diagram:additivity_of_braiding}) which expresses the additive nature of $\eta$.

Remark that the differential $D_2$ corresponds to a morphism of $\ES$-2-stacks, the group law, the differentials $D_3$ correspond to natural 2-transformations, the associativity $\ass$ and the braiding $\comm$, the differentials $D_4$ correspond to modifications, which express the obstructions to the coherence axioms or the compatibility conditions for natural 2-transformations, and finally the differentials $D_5$ correspond to the coherence axioms or the compatibility conditions for modifications.

In \cite[Expos\'e VII, Remark 3.5.4]{SGA7} Grothendieck pointed out that it would be interesting to have for any abelian sheaf $P$ a resolution $L.(P)$, which depends functorially on $P$, and whose entries are sums of free $\mathbb{Z}$-modules generated by cartesian products of $P$. The same issue is addressed in Illusie's book \cite{Illusie}, see in particular Chapter VI page 132 line 13 and Section 11.4. Working with abelian sheaves, in \cite[Expos\'e VII, (3.5.2)]{SGA7} Grothendieck got the first two differential operators $D_2$ and $D_3$ of the resolution $L.(P)$. Working with Picard stacks, in \cite{Breen} and \cite{Breen69} Breen has computed the differential operator $D_4$ of this resolution.
Corollary \ref{corollary:resolutionP} is the authors' contribution to Grothendieck's remark: working with Picard 2-stacks, in this paper we have computed the differential operator $D_5$.

If we denote by $3{\mathcal{P}\textsc{icard}^{\flat \flat \flat}(\ES)}$ the category of Picard 3-stacks whose objects are Picard 3-stacks and whose arrows are equivalence classes of additive 3-functors, another consequence of the description of extensions of Picard $\ES$-2-stacks in terms of torsors (Cor. \ref{corollary:Z[I]exttor} and Thm. \ref{thm:introexttor}) is

\begin{corollary}\label{corollary:resolution}
Let $\twoP$ and $\twoG$ be two Picard $\ES$-2-stacks. The complex
\begin{align*}
 0 &\ra  \TORS(\twoG_\twoP) \stackrel{D_2^*}{\ra}  \TORS(\twoG_{\twoP^2}) \stackrel{D_3^*}{\ra} \TORS(\twoG_{\twoP^3})  \times \TORS(\twoG_{\twoP^2}) \stackrel{D_4^* }{\ra}...\\
\nonumber ...& \stackrel{D_4^* }{\ra} \TORS(\twoG_{\twoP^4})  \times \TORS(\twoG_{\twoP^3})^2 \times \TORS(\twoG_{\twoP^2}) \times \TORS(\twoG_{\twoP})\stackrel{D_5^* }{\ra} ... \\ 
\nonumber ...&   \stackrel{D_5^* }{\ra}   \TORS(\twoG_{\twoP^5})  \times \TORS(\twoG_{\twoP^4})^3 \times \TORS(\twoG_{\twoP^3})^3 \times \TORS(\twoG_{\twoP^2})  \times \\
\nonumber & \TORS(\twoG_{\twoP})  \times \TORS (\twoG_{\twoP^2})\ra 0
\end{align*}
is a right resolution of the 3-category $\Ext(\twoP,\twoG)$ of extensions of $\twoP$ by $\twoG$ in the category $3{\mathcal{P}\textsc{icard}^{\flat \flat \flat}(\ES)}$. Here $D_i^* $ denotes the pull-back via the differential operator $D_i$ (\ref{complex_L(P)}) (for $i=2,3,4,5$).
\end{corollary}

This last result can be rewritten in the derived category $\der$ of abelian sheaves on $\ES$, using the homological interpretation of extensions of Picard $\ES$-2-stacks \cite[Thm. 1.1]{beta} and  of torsors under Picard $\ES$-2-stacks (Thm. \ref{thm:introtor}): 

\begin{corollary}
Let $P$ and $G$ be length 3 complexes of abelian sheaves on $\ES$. The complex 
\begin{align*}
0 &\ra \tau_{\leq 0}{\mathrm{R}}\Gamma(G_P[1])  \stackrel{d_2}{\ra}  \tau_{\leq 0}{\mathrm{R}}\Gamma(G_{P^2}[1]) \stackrel{d_3}{\ra}  \tau_{\leq 0}{\mathrm{R}}\Gamma(G_{P^3}[1]) \times \tau_{\leq 0}{\mathrm{R}}\Gamma(G_{P^2}[1])  
\stackrel{d_4}{\ra} ...\\ 
\nonumber ... &\stackrel{d_4}{\ra} {\mathrm{R}}\Gamma(G_{P^4 }[1]) \times \tau_{\leq 0}{\mathrm{R}}\Gamma(G_{P^3}[1])^2
\times \tau_{\leq 0}{\mathrm{R}}\Gamma(G_{P^2}[1]) \times \tau_{\leq 0}{\mathrm{R}}\Gamma(G_{P}[1]) \stackrel{d_5}{\ra} ...\\
\nonumber ... &\stackrel{d_5}{\ra} \tau_{\leq 0}{\mathrm{R}}\Gamma(G_{P^5 }[1]) \times
 \tau_{\leq 0}{\mathrm{R}}\Gamma(G_{P^4}[1])^3  \times
 \tau_{\leq 0}{\mathrm{R}}\Gamma(G_{P^3}[1])^3 \times
 \tau_{\leq 0}{\mathrm{R}}\Gamma(G_{P^2}[1]) \times\\
\nonumber & \tau_{\leq 0}{\mathrm{R}}\Gamma(G_{P}[1]) \times 
 \tau_{\leq 0}{\mathrm{R}}\Gamma(G_{P^2}[1]) \ra 0
\end{align*}
 is a right resolution of the object $\tau_{\leq 0}{\mathrm{R}}\hhom(P,G[1])$ of $\mathcal{D}^{[-3,0]}(\ES)$.
\end{corollary}

In \cite{MR2995663} the first author describes explicitly extensions of Picard $\ES$-stacks in terms of torsors under Picard $\ES$-stacks which are endowed with an abelian group law on the fibers (see in particular \cite[Thm. 4.1]{MR2995663}). In order to generalize from $\ES$-stacks to $\ES$-2-stacks the notions of \cite{MR2995663} that we need in this paper (as, for example, the definition of torsor) we proceed as follows: the data involving 1-arrows and 2-arrows remain the same, but the coherence axioms or the compatibility conditions, that 2-arrows have to satisfy and that are given via equations of 1-arrows, are replaced by 3-arrows which express the obstruction to the above coherence axioms or compatibility conditions for 2-arrows, and we require that these 3-arrows satisfies some coherence axioms or compatibility conditions that are given via equations of 2-arrows.

We hope that this work will shed some light on the notions of ``torsor" for higher categories with group-like operation. In particular, as in \cite{beta}, we pay a lot of attention to write down the proofs in such a way that they can be easily generalized to Picard $\ES$-n-stacks and to length n+1 complexes of abelian sheaves on $\ES$.

Theorem \ref{thm:introtor} plays an important role in the proof of Theorem 0.1 of \cite{bega} which states that
 the Picard 2-stack of $F$-gerbes $\bGerbe_{\ES} (F)$, with $F$ an abelian sheaf on a site $\ES$, is equivalent (as Picard 2-stack) to the Picard 2-stack associated to the complex $\tau_{\leq 0}{\mathrm{R}}\Gamma(\ES,F[2])$, where $F[2]=[F \to 0 \to 0]$ with $F$ in degree -2. In particular, our Theorem \ref{thm:introtor} allows the first author to obtain a purely categorical proof of the classical fact that $F$-equivalence classes of $F$-gerbes, which are the elements of the 0th-homotopy group of $\bGerbe_{\ES} (F)$, are parametrized by the elements of the cohomological group $\HH^2(\ES,F)$.

The study of torsors under Picard $\ES$-2-stacks is a first step toward the theory of biextensions of Picard $\ES$-2-stacks: in fact, if $\twoP,\twoQ$ and $\twoG$ are Picard $\ES$-2-stacks, a biextension of $(\twoP,\twoQ)$ by $\twoG$ is a $\twoG_{\twoP \times \twoQ}$-torsor over $\twoP \times \twoQ $  endowed with two compatible group laws on the fibers. Using the canonical free partial resolution $\twoL.(\twoP)$ of $\twoP$ (Cor. \ref{corollary:resolutionP}) and the 3-category $\Psi_{\twoL.(\twoP)\otimes \twoL.(\twoQ)}(\twoG)$ introduced in Definition \ref{psi}, we get easily the homological interpretation of biextensions of $(\twoP,\twoQ)$ by $\twoG$: $\pi_{-i+1}(\mathcal{B}\mathrm{iext} (\twoP,\twoQ;\twoG)) \cong \hhom_{\der}\big([\twoP]^{\flat\flat} \otimes [\twoQ]^{\flat\flat},[\twoG]^{\flat\flat}[i]\big)$ for $i=1,0,-1,-2$, where $\pi_{-i+1}(\mathcal{B}\mathrm{iext} (\twoP,\twoQ;\twoG))$ are the homotopy groups of the 3-category of biextensions of $(\twoP,\twoQ)$ by $\twoG$. The theory of biextensions has important applications in the theory of motives since biextensions define bilinear morphisms between motives.

%-----------------------------------------------
\section*{Acknowledgment}
We thank Larry Breen and Pierre Deligne for their interesting comments about the resolution $\twoL.(\twoP)$ obtained in Corollary \ref{corollary:resolutionP}. We are very grateful to the anonymous referee for the useful remarks concerning the differential operators involved in the resolution $\twoL.(\twoP).$ The second author is supported by KFUPM under research grant RG1322-1 and -2.

 %-------------------------------------------------------------------------------------------------
\section*{Notation}

In this paper $\ES$ will be any site whose topology is precanonical so that the representable presheaves are sheaves.

We denote by ${\mathcal{K}}(\ES)$ the category of (cochain) complexes of abelian sheaves on the site $\ES$. Let $\twocatK$ be the subcategory of ${\mathcal{K}}(\ES)$ consisting of complexes $K=(K^i)_{i \in \mathbb{Z}}$ such that $K^i=0$ for $i \not= -2,-1$ or $0$. The good truncation $ \tau_{\leq n} K$ of a complex $K$ of ${\mathcal{K}}(\textbf{S})$ is the following complex: $ (\tau_{\leq n} K)^i= K^i$ for $i <n,  (\tau_{\leq n} K)^n= \ker(d^n)$, and $(\tau_{\leq n} K)^i= 0$ for $i > n$. For any $i \in {\mathbb{Z}}$, the shift functor $[i]:{\mathcal{K}}(\ES) \ra {\mathcal{K}}(\ES) $ acts on a  complex $K=(K^n)_{n \in \mathbb{Z}}$ as $(K[i])^n=K^{i+n}$ and $d^n_{K[i]}=(-1)^{i} d^{n+i}_{K}$.

Denote by $\der$ the derived category of abelian sheaves on $\ES$, and let $\twocatD$ be the full subcategory of $\der$ consisting of complexes $K$ such that ${\mathrm{H}}^i (K)=0$ for $i \not= -2,-1$ or $0$. If $K$ and $L$ are complexes of $\der$, the group $\eext^i(K,L)$ is by definition $\hhom_{\der}(K,L[i])$ for any $i \in {\mathbb{Z}}$. Let ${\mathrm{R}}\hhom(-,-)$ be the derived functor of the bifunctor $\hhom(-,-)$. The $i$-th cohomology group ${\mathrm{H}}^i\big({\mathrm{R}}\hhom(K,L) \big)$ of ${\mathrm{R}}\hhom(K,L)$ is isomorphic to $\hhom_{\der}(K,L[i])$.
The functor $\Gamma(-)$ of global sections is isomorphic to the functor $\hhom(\mathbf{e},-),$ where $\mathbf{e}$ is the final object of the category of abelian sheaves on $\ES.$ Let ${\mathrm{R}}\Gamma(-)$ be the derived functor of the functor $\Gamma(-)$ of global sections. The $i$-th cohomology group $\HH^i\big({\mathrm{R}}\Gamma(K)\big)$ of ${\mathrm{R}}\Gamma(K)$ is denoted by $\HH^i(K)$.

In this paper, by an $\ES$-2-(pre)stack we will always mean an $\ES$-2-(pre)stack in 2-groupoids.

\section{Recollections on Picard 2-Stacks}\label{section:recall_on_Picard_2_stacks}

The notion of Picard 2-stacks is well known \cite[Def. 8.4]{MR1301844}. In simplest words, it is a 2-stack over a site equipped with a commutative group-like structure. In the literature, there are no references that the authors are aware of where the details of the commutative group-like structure of a 2-stack is stated explicitly. Although we believe that it is known by the experts, since it will be needed in the paper, in this section we unravel the details of this structure. In the following definitions, $U$ will denote an object of the site $\ES$. Moreover in the diagrams involving 2-arrows, we will put the symbol $\cong $ in the cells which commute up to a modification of $\ES$-2-stacks coming from the Picard structure.

A \emph{strict Picard $\ES$-2-stack} (just called Picard $\ES$-2-stack) $\twoP=(\twoP,\otimes,\ass,\pi,\comm,\zeta,\gotikh_1,\gotikh_2,\eta)$ is an $\ES$-2-stack $\twoP$ equipped with 
\begin{enumerate}
\item a morphism of $\ES$-2-stacks $\otimes : \twoP \times \twoP \ra \twoP$, called the \emph{group law} of $\twoP$.  For simplicity instead of $X \otimes Y$ we write just $XY$ for all $X,Y \in \twoP (U)$;
\item two natural 2-transformations of $\ES$-2-stacks
$\ass :\otimes \circ (\otimes \times \id_\twoP) \Ra \otimes \circ (\id_\twoP \times \otimes)$, called the \emph{associativity}, and 
$\comm : \otimes \circ \mathsf{s} \Ra \otimes$ with $ \mathsf{s}(X,Y)=(Y,X)$ for all $X,Y \in \twoP(U)$, called the \emph{braiding}, which express respectively the associativity and the commutativity constraints of the group law $\otimes$ of $\twoP$; 
\item a modification $\pi$ of $\ES$-2-stacks whose component at $(X,Y,Z,W) \in \twoP^4(U)$ is the 2-arrow

\begin{equation}\label{ObstructionCoherenceAss}
\begin{tabular}{c}
\xymatrix@!=0.75cm@R=1cm{ &((XY)Z)W \ar[dl]_{\ass_{(XY,Z,W)}} \ar[dr]^{\ass_{(X,Y,Z)}W} \ar@{}[dd]|(0.6){\kesir{\mbox{$\La$}}{\pi_{(X,Y,Z,W)}}} &\\(XY)(ZW) \ar[d]_{\ass_{(X,Y,ZW)}} & & (X(YZ))W \ar[d]^{\ass_{(X,YZ,W)}}\\X(Y(ZW)) & & X((YZ)W) \ar[ll]^{X\ass_{(Y,Z,W)}}}
\end{tabular}
\end{equation}		
and which expresses the obstruction to the coherence of the associativity $\ass$ (i.e. the obstruction to the pentagonal axiom);
\item a modification $\zeta$ of $\ES$-2-stacks  whose component at $(X,Y) \in \twoP^2(U)$ is the 2-arrow 
\begin{equation}\label{ObstructionCoherenceComm}
\zeta_{(X,Y)}: \id_{XY} \Ra \comm_{(Y,X)} \circ \comm_{(X,Y)}
\end{equation}
and which expresses the obstruction to the coherence of the braiding $\comm$. The modification $\zeta$ implies the weak invertability of the braiding $\comm$;
\item  two modifications $\gotikh_1,\gotikh_2$ of $\ES$-2-stacks whose components at $(X,Y,Z) \in \twoP^3(U)$ are the 2-arrows
\begin{equation}\label{ObstructionCompatibilityAssComm}
\begin{tabular}{cc}
\xymatrix@C=0.25cm@R=1cm{ & **[l]X(YZ) \ar[r]^{\comm_{(X,YZ)}}  \ar@{}[ddr]|{\kesir{\mbox{$\Da$}}{\gotikh_{1(X,Y,Z)}}}  & **[r](YZ)X\ar[dr]^{\ass_{(Y,Z,X)}}&\\
(XY)Z  \ar[ur]^{\ass_{(X,Y,Z)}} \ar[dr]_{\comm_{(X,Y)}Z}&&& Y(ZX)\\
&**[l](YX)Z \ar[r]_{\ass_{(Y,X,Z)}} & **[r]Y(XZ) \ar[ur]_{Y\comm_{(X,Z)}}&} 
&
\xymatrix@C=0.25cm@R=1cm{&**[l](XY)Z \ar[r]^{\comm_{(XY,Z)}} \ar@{}[ddr]|{\kesir{\mbox{$\Da$}}{\gotikh_{2(X,Y,Z)}}} & **[r]Z(XY)\ar[dr]^{\invass_{(Z,X,Y)}}&\\
X(YZ) \ar[ur]^{\invass_{(X,Y,Z)}} \ar[dr]_{X\comm_{(Y,Z)}}&&&(ZX)Y\\
&**[l]X(ZY) \ar[r]_{\invass_{(X,Z,Y)}}&**[r](XZ)Y\ar[ur]_{\comm_{(X,Z)}Y}&} 
\end{tabular}
\end{equation}
and which express the obstruction to the compatibility between the associativity $\ass$ and the braiding $\comm$ (i.e. the obstruction to the hexagonal axiom);
\item a modification $\eta$ of $\ES$-2-stacks whose component at $X \in \twoP(U)$ is the 2-arrow
\begin{equation}\label{Strictness}
\eta_{X} :\id_{XX} \Ra \comm_{(X,X)}
\end{equation}
and which expresses the obstruction to the strictness of the braiding $\comm$. 
\end{enumerate}

These data satisfy the following compatibility conditions:
\begin{enumerate}[(i)]
\item for any $X \in \twoP(U)$, the morphism of $\ES$-2-stacks $X \otimes -:\twoP \ra \twoP$, called the left multiplication by $X$, is an equivalence of $\ES$-2-stacks;
\item the modification $\pi$ is coherent, i.e. it satisfies the coherence axiom of Stasheff's polytope (see \cite[\S 4]{MR1278735}): for all $X,Y,Z,W,T \in \twoP(U)$ the following equation of 2-arrows holds

%%%%%%tikzpicture%%%%%%%%%
\begin{equation}\label{diagram:stasheff}
\begin{tabular}{c}
\begin{tikzpicture}[baseline=(current bounding box.center), descr/.style={fill=white,inner sep=2.5pt}, scale=1.45]
%right side
\node (A) at (1.5,3.75) {\scriptsize{$X(Y(Z(WT)))$}};\node (B) at (0,3) {\scriptsize{$(XY)(Z(WT))$}};\node (C) at (3,3) {\scriptsize{$X(Y((ZW)T))$}};\node (D) at (0,1.5) {\scriptsize{$((XY)Z)(WT)$}};\node (E) at (1.5,1.5) {\scriptsize{$X((YZ)(WT))$}};\node (F) at (3,1.5) {\scriptsize{$X((Y(ZW))T)$}};\node (G) at (0,0.75) {\scriptsize{$(((XY)Z)W)T$}};\node (H) at (1.5,0.75) {\scriptsize{$(X(YZ))(WT)$}};\node (I) at (3,0.75) {\scriptsize{$X(((YZ)W)T)$}};\node (J) at (0,0) {\scriptsize{$((X(YZ))W)T$}};\node (K) at (3,0) {\scriptsize{$(X((YZ)W))T$}};
\node (O) at (4,1.5) {$=$};
%left side
\node (a) at (6.5,3.75) {\scriptsize{$X(Y(Z(WT)))$}};\node (b) at (5,3) {\scriptsize{$(XY)(Z(WT))$}};\node (c) at (8,3) {\scriptsize{$X(Y((ZW)T))$}};\node (d) at (6.5,2.25) {\scriptsize{$(XY)((ZW)T)$}};\node (e) at (5,1.5) {\scriptsize{$((XY)Z)(WT)$}};\node (f) at (6.5,1.5) {\scriptsize{$((XY)(ZW))T$}};\node (g) at (8,1.5) {\scriptsize{$X((Y(ZW))T)$}};\node (h) at (5,0.75) {\scriptsize{$(((XY)Z)W)T$}};\node (i) at (6.5,0.75) {\scriptsize{$(X(Y(ZW)))T$}};\node (j) at (8,0.75) {\scriptsize{$X(((YZ)W)T)$}};\node (k) at (5,0) {\scriptsize{$((X(YZ))W)T$}};\node (l) at (8,0) {\scriptsize{$(X((YZ)W))T$}};
%2arrows
\node (X) at (0.75,2.1) {$\kesir{\mbox{$\La$}}{\pi_{(X,Y,Z,WT)}}$};
\node (Y) at (2.25,2.1) {$\kesir{\mbox{$\La$}}{X\pi_{(Y,Z,W,T)}}$};
\node (Z) at (2.2,0.42) {$\kesir{\mbox{$\La$}}{\pi_{(X,YZ,W,T)}}$};
\node (W) at (0.75,0.75) {\scriptsize{$\cong$}};
\node (x) at (5.75,1.85) {$\kesir{\mbox{$\La$}}{\pi_{(XY,Z,W,T)}}$};
\node (y) at (7.25,1.85) {$\kesir{\mbox{$\La$}}{\pi_{(X,Y,ZW,T)}}$};
\node (z) at (5.9,0.33) {$\kesir{\mbox{$\Ua$}}{\pi_{(X,Y,Z,W)}T}$};
\node (w) at (6.5,3) {\scriptsize{$\cong$}};
\node (v) at (7.25,0.75) {\scriptsize{$\cong$}};
%edges
\path[->, font=\scriptsize]
(B) edge (A)(C) edge (A)(D) edge (B)(D) edge (H)(E) edge (A)(F) edge (C)(G) edge (D)(G) edge (J)(H) edge (E)(I) edge (E)(I) edge (F)(J) edge (H)(J) edge (K)(K) edge (I)(b) edge (a)(c) edge (a)(d) edge (b)(d) edge (c)(e) edge (b)(f) edge (d)(f) edge (i)(g) edge (c)(h) edge (e)(h) edge (f)(h) edge (k)(i) edge (g)(j) edge (g)(k) edge (l)(l) edge (i)(l) edge (j);
\end{tikzpicture}
\end{tabular}
\end{equation}
 %%%%%%tikzpicture%%%%%%%%%% 
 \item the modification $\zeta$ is coherent, i.e. for all $X,Y,Z \in \twoP(U)$ the following equation of 2-arrows holds
 \begin{equation}\label{CoherenceZeta}
\zeta_{(Y,X)} * \comm_{(X,Y)}=\comm_{(X,Y)} * \zeta_{(X,Y)},
\end{equation}
 \item the modification $\gotikh_1$ is compatible with $\pi$, i.e. for all $X,Y,Z,W \in \twoP(U)$ the following equation of 2-arrows is satisfied
%%%%%%tikzpicture%%%%%%%%%
\begin{equation}\label{diagram:hexagone1}
\begin{tikzpicture}[baseline=(current bounding box.center), descr/.style={fill=white,inner sep=2.5pt}, scale=2.5]
%left side
\node (1) at (0.5,2.5) {\scriptsize{$(X(YZ))W$}};\node (2) at (1.5,2.5) {\scriptsize{$X((YZ)W)$}};\node (3) at (0,2) {\scriptsize{$((XY)Z)W$}};\node (4) at (1,2) {\scriptsize{$(XY)(ZW)$}};\node (5) at (2,2) {\scriptsize{$X(Y(ZW)) $}};\node (6) at (0,1.5) {\scriptsize{$((YX)Z)W$}};\node (7) at (1,1.5) {\scriptsize{$(YX)(ZW)$}};\node (8) at (2,1.5) {\scriptsize{$(Y(ZW))X$}};\node (9) at (0,1) {\scriptsize{$(Y(XZ))W$}};\node (10) at (1,1) {\scriptsize{$Y(X(ZW))$}};\node (11) at (2,1) {\scriptsize{$Y((ZW)X)$}};\node (12) at (0,0.5) {\scriptsize{$Y((XZ)W)$}};\node (13) at (2,0.5) {\scriptsize{$Y(Z(WX))$}};\node (14) at (0.5,0) {\scriptsize{$Y((ZX)W)$}};\node (15) at (1.5,0) {\scriptsize{$Y(Z(XW))$}};
\node (0) at (2.5,1.25) {$=$};
%right side
\node (1') at (3.5,2.5) {\scriptsize{$(X(YZ))W$}};\node (2') at (4.5,2.5) {\scriptsize{$X((YZ)W)$}};\node (3') at (3,2) {\scriptsize{$((XY)Z)W$}};\node (4') at (3.65,2) {\scriptsize{$((YZ)X)W$}};\node (5') at (4.35,2) {\scriptsize{$((YZ)W)X$}};\node (6') at (5,2) {\scriptsize{$X(Y(ZW))$}};\node (7') at (3,1.5) {\scriptsize{$((YX)Z)W$}};\node (8') at (3.5,1.25) {\scriptsize{$(Y(ZX))W$}};\node (9') at (4,1) {\scriptsize{$(YZ)(XW)$}};\node (10') at (4.5,0.75) {\scriptsize{$(YZ)(WX)$}};\node (11') at (5,1.5) {\scriptsize{$(Y(ZW))X $}};\node (12') at (3,1) {\scriptsize{$(Y(XZ))W$}};\node (13') at (5,1) {\scriptsize{$Y((ZW)X)$}};\node (14') at (3,0.5) {\scriptsize{$Y((XZ)W)$}};\node (15') at (5,0.5) {\scriptsize{$Y(Z(WX))$}};\node (16') at (3.5,0) {\scriptsize{$Y((ZX)W)$}};\node (17') at (4.5,0) {\scriptsize{$Y(Z(XW))$}};
%2-arrows
\node (X) at (0.5,1.75) {\scriptsize{$\cong$}};
\node (Y) at (1.1,2.27) {$\kesir{\mbox{$\Da$}}{\pi_{(X,Y,Z,W)}}$};
\node (Z) at (0.5,1.2) {$\kesir{\rotatebox{30}{$\Ra$}}{\pi_{(Y,X,Z,W)}}$};
\node (W) at (1.6,1.45) {$\kesir{\mbox{$\Da$}}{\gotikh_{1(X,Y,ZW)}}$};
\node (V) at (1.1,0.5) {$\kesir{\mbox{$\Da$}}{Y\gotikh_{1(X,Z,W)}}$};
\node (x) at (4.67,2) {\scriptsize{$\cong$}};
\node (y) at (3.25,0.75) {\scriptsize{$\cong$}};
\node (z) at (4.5,0.5) {\scriptsize{$\cong$}};
\node (w) at (3.85,0.4) {$\kesir{\mbox{$\Ra$}}{\pi_{(Y,Z,X,W)}}$};
\node (v) at (3.3,1.72) {$\kesir{\rotatebox{-30}{$\Da$}}{\gotikh_{1(X,Y,Z)}W}$};
\node (u) at (4.1,1.72) {$\kesir{\rotatebox{-30}{$\Da$}}{\gotikh_{1(X,YZ,W)}}$};
\node (t) at (4.73,1.24) {$\kesir{\mbox{$\La$}}{\pi_{(Y,Z,W,X)}}$};
%edges
\path[->, font=\scriptsize]
(1) edge (2)(2) edge (5)(3) edge (1)(3) edge (4)(3) edge (6)(4) edge (5)(4) edge (7)(5) edge (8)(6) edge (7)(6) edge (9)(7) edge (10)(8) edge (11)(9) edge (12)(10) edge (11)(11) edge (13)(12) edge (10)(12) edge (14)(14) edge (15)(15) edge (13)(1') edge (2')(1') edge (4')(2') edge (5')(2') edge (6')(3') edge (1')(3') edge (7')(4') edge (8')(4') edge (9')(5') edge (10')(5') edge (11')(6') edge (11')(7') edge (12')(8') edge (16')(9') edge (17')(9') edge (10')(10') edge (15')(11') edge (13')(12') edge (8')(12') edge (14')(13') edge (15')(14') edge (16')(16') edge (17')(17') edge (15');
\end{tikzpicture}
\end{equation}
%%%%%%tikzpicture%%%%%%%%%
and the modification $\gotikh_2$ is compatible with $\pi$, i.e. for all $X,Y,Z,W \in \twoP(U)$ the following equation of 2-arrows is satisfied
%%%%%%tikzpicture%%%%%%%%%
\begin{equation}\label{diagram:hexagone2}
\begin{tikzpicture}[baseline=(current bounding box.center), descr/.style={fill=white,inner sep=2.5pt}, scale=2.5]
%left side
\node (1) at (0.5,2.5) {\scriptsize{$X((YZ)W)$}};\node (2) at (1.5,2.5) {\scriptsize{$(X(YZ))W$}};\node (3) at (0,2) {\scriptsize{$X(Y(ZW))$}};\node (4) at (1,2) {\scriptsize{$(XY)(ZW)$}};\node (5) at (2,2) {\scriptsize{$((XY)Z)W$}};\node (6) at (0,1.5) {\scriptsize{$X(Y(WZ))$}};\node (7) at (1,1.5) {\scriptsize{$(XY)(WZ)$}};\node (8) at (2,1.5) {\scriptsize{$W((XY)Z)$}};\node (9) at (0,1) {\scriptsize{$X((YW)Z)$}};\node (10) at (1,1) {\scriptsize{$((XY)W)Z$}};\node (11) at (2,1) {\scriptsize{$(W(XY))Z$}};\node (12) at (0,0.5) {\scriptsize{$(X(YW))Z$}};\node (13) at (2,0.5) {\scriptsize{$((WX)Y)Z$}};\node (14) at (0.5,0) {\scriptsize{$(X(WY))Z$}};\node (15) at (1.5,0) {\scriptsize{$((XW)Y)Z$}};
\node (0) at (2.5,1.25) {$=$};
%right side
\node (1') at (3.5,2.5) {\scriptsize{$X((YZ)W)$}};\node (2') at (4.5,2.5) {\scriptsize{$(X(YZ))W$}};\node (3') at (3,2) {\scriptsize{$X(Y(ZW))$}};\node (4') at (3.65,2) {\scriptsize{$X(W(YZ))$}};\node (5') at (4.35,2) {\scriptsize{$W(X(YZ))$}};\node (6') at (5,2) {\scriptsize{$((XY)Z)W$}};\node (7') at (3,1.5) {\scriptsize{$X(Y(WZ))$}};\node (8') at (3.5,1.25) {\scriptsize{$X((WY)Z)$}};\node (9') at (4,1) {\scriptsize{$(XW)(YZ)$}};\node (10') at (4.5,0.75) {\scriptsize{$(WX)(YZ)$}};\node (11') at (5,1.5) {\scriptsize{$W((XY)Z)$}};\node (12') at (3,1) {\scriptsize{$X((YW)Z)$}};\node (13') at (5,1) {\scriptsize{$(W(XY))Z$}};\node (14') at (3,0.5) {\scriptsize{$(X(YW))Z$}};\node (15') at (5,0.5) {\scriptsize{$((WX)Y)Z$}};\node (16') at (3.5,0) {\scriptsize{$(X(WY))Z$}};\node (17') at (4.5,0) {\scriptsize{$((XW)Y)Z$}};
%2-arrows
\node (X) at (0.5,1.75) {\scriptsize{$\cong$}};
\node (Y) at (1.1,2.27) {\scriptsize{$\pi^*_{(X,Y,Z,W)}$}};\node (Y') at (0.75,2.25) {$\Da$};
\node (Z) at (0.5,1.2) {\scriptsize{$\pi^*_{(X,Y,W,Z)}$}};\node (Z') at (0.5,1.3) {$\Ra$};
\node (W) at (1.6,1.4) {\scriptsize{$\gotikh_{2(XY,Z,W)}$}};\node (W') at (1.5,1.5) {$\Da$};
\node (V) at (1.1,0.5) {\scriptsize{$\gotikh_{2(X,Y,W)}Z$}};\node (V') at (0.75,0.5) {$\Da$};
\node (x) at (4.67,2) {\scriptsize{$\cong$}};
\node (y) at (3.25,0.75) {\scriptsize{$\cong$}};
\node (z) at (4.5,0.5) {\scriptsize{$\cong$}};
\node (w) at (3.83,0.4) {\scriptsize{$\pi^*_{(X,W,Y,Z)}$}};\node (w') at (3.85,0.5) {$\Ra$};
\node (v) at (3.35,1.68) {\scriptsize{$X\gotikh_{2(Y,Z,W)}$}};\node (v') at (3.3,1.8) {\rotatebox{225}{$\Ra$}};
\node (u) at (4.1,1.68) {\scriptsize{$\gotikh_{2(X,YZ,W)}$}};\node (u') at (4,1.8) {\rotatebox{225}{$\Ra$}};
\node (t) at (4.73,1.2) {\scriptsize{$\pi^*_{(W,X,Y,Z)}$}};\node (t') at (4.68,1.3) {$\La$};
%edges
\path[->, font=\scriptsize]
(1) edge (2)(2) edge (5)(3) edge (1)(3) edge (4)(3) edge (6)(4) edge (5)(4) edge (7)(5) edge (8)(6) edge (7)(6) edge (9)(7) edge (10)(8) edge (11)(9) edge (12)(10) edge (11)(11) edge (13)(12) edge (10)(12) edge (14)(14) edge (15)(15) edge (13)(1') edge (2')(1') edge (4')(2') edge (5')(2') edge (6')(3') edge (1')(3') edge (7')(4') edge (8')(4') edge (9')(5') edge (10')(5') edge (11')(6') edge (11')(7') edge (12')(8') edge (16')(9') edge (17')(9') edge (10')(10') edge (15')(11') edge (13')(12') edge (8')(12') edge (14')(13') edge (15')(14') edge (16')(16') edge (17')(17') edge (15');
\end{tikzpicture}
\end{equation}
%%%%%%tikzpicture%%%%%%%%%
where the modification $\pi^*$ is obtained from $\pi$ by inverting some or all $\ass$'s. The modifications $\gotikh_1$ and $\gotikh_2$ are comparable in the sense that the pasting of the 2-arrows in the diagram 
%%%%%%tikzpicture%%%%%%%%%
\begin{equation}\label{diagram:hexagone3}
\begin{tikzpicture}[baseline=(current bounding box.center), descr/.style={fill=white,inner sep=2.5pt}, scale=2]
%left side
\node (1) at (2,4) {\scriptsize{$X((YZ)W)$}};
\node (2) at (3,4) {\scriptsize{$X((ZY)W)$}};
\node (3) at (4,4) {\scriptsize{$X(Z(YW))$}};
\node (4) at (1,3) {\scriptsize{$(X(YZ))W$}};
\node (5) at (2,3) {\scriptsize{$(X(ZY))W$}};
\node (6) at (3,3) {\scriptsize{$((XZ)Y)W$}};
\node (7) at (5,3) {\scriptsize{$(XZ)(YW)$}};
\node (8) at (0,2) {\scriptsize{$((XY)Z)W$}};
\node (9) at (1,2) {\scriptsize{$(Z(XY))W$}};
\node (10) at (2,2) {\scriptsize{$((ZX)Y)W$}};
\node (11) at (4,2) {\scriptsize{$(ZX)(YW)$}};
\node (12) at (6,2) {\scriptsize{$(XZ)(WY)$}};
\node (13) at (0,1) {\scriptsize{$(XY)(ZW)$}};
\node (14) at (1,1) {\scriptsize{$Z((XY)W)$}};
\node (15) at (3,1) {\scriptsize{$Z(X(YW))$}};
\node (16) at (4,1) {\scriptsize{$Z(X(WY))$}};
\node (17) at (5,1) {\scriptsize{$(ZX)(WY)$}};
\node (18) at (6,1) {\scriptsize{$((XZ)W)Y$}};
\node (19) at (0,0) {\scriptsize{$(ZW)(XY)$}};
\node (20) at (1,0) {\scriptsize{$Z(W(XY))$}};
\node (21) at (2,0) {\scriptsize{$Z((WX)Y)$}};
\node (22) at (4,0) {\scriptsize{$Z((XW)Y)$}};
\node (23) at (5,0) {\scriptsize{$(Z(XW))Y$}};
\node (24) at (6,0) {\scriptsize{$((ZX)W)Y$}};
%2-arrows
\node (c_1) at (2,3.5) {$\cong$};\node (c_2) at (3.5,2.5) {$\cong$};\node (c_3) at (5,2) {$\cong$};\node (c_4) at (4,1.5) {$\cong$};\node (c_5) at (5.5,1) {$\cong$};
\node (p1) at (3.5,3.4) {\scriptsize{$\pi^*{}^{-1}_{(X,Z,Y,W)}$}};\node (p1') at (3.5,3.6){$\Ua$};
\node (h1) at (1.5,2.4) {\scriptsize{$\gotikh_{2(X,Y,Z)}W$}};\node (h1') at (1.5,2.6){$\Ua$};
\node (h2) at (0.5,0.85) {\scriptsize{$\gotikh_{1(XY,Z,W)}$}};\node (h2') at (0.5,1){$\Ra$};
\node (p2) at (2.5,1.4) {\scriptsize{$\pi^*{}_{(Z,X,Y,W)}$}};\node (p2') at (2.5,1.6){$\Ra$};
\node (h3) at (2.5,0.4) {\scriptsize{$Z\gotikh_{2(X,Y,W)}$}};\node (h3') at (2.5,0.6){$\Ra$};
\node (p3) at (4.75,0.4) {\scriptsize{$\pi^*{}^{-1}_{(Z,X,W,Y)}$}};\node (p3') at (4.75,0.6){$\Da$};
%edges
\path[->, font=\scriptsize]
(1) edge (2)(1) edge (4)(2) edge (3)(2) edge (5)(3) edge (7)(4) edge (5)(4) edge (8)(5) edge (6)(6) edge (7)(6) edge (10)(7) edge (11)(7) edge (12)(8) edge (9)(8) edge (13)(9) edge (10)(9) edge (14)(10) edge (11)(11) edge (15)(11) edge (17)(12) edge (17)(12) edge (18)(13) edge (19)(14) edge (20)(15) edge (14)(15) edge (16)(16) edge (22)(17) edge (16)(17) edge (24)(18) edge (24)(19) edge (20)(20) edge (21)(22) edge (21)(23) edge (22)(24) edge (23);
\end{tikzpicture}
\end{equation}
%%%%%%tikzpicture%%%%%%%%%
is equal to the pasting of the 2-arrows in the diagram
%%%%%%tikzpicture%%%%%%%%%
\begin{equation}\label{diagram:hexagone3bis}
\begin{tikzpicture}[baseline=(current bounding box.center), descr/.style={fill=white,inner sep=2.5pt}, scale=2.2]
%left side
\node (1) at (1,4) {\scriptsize{$X((YZ)W)$}};
\node (2) at (2,4) {\scriptsize{$X((ZY)W)$}};
\node (3) at (3,4) {\scriptsize{$X(Z(YW))$}};
\node (4) at (0,4) {\scriptsize{$(X(YZ))W$}};
\node (5) at (4,4) {\scriptsize{$(XZ)(YW)$}};
\node (6) at (0,3) {\scriptsize{$((XY)Z)W$}};
\node (7) at (1,3) {\scriptsize{$X(Y(ZW))$}};
\node (8) at (2,3) {\scriptsize{$X((ZW)Y)$}};
\node (9) at (3,3) {\scriptsize{$X(Z(WY))$}};
\node (10) at (4,3) {\scriptsize{$(XZ)(WY)$}};
\node (11) at (0,2) {\scriptsize{$(XY)(ZW)$}};
\node (12) at (2,2) {\scriptsize{$(X(ZW))Y$}};
\node (13) at (4,2) {\scriptsize{$((XZ)W)Y$}};
\node (14) at (0,1) {\scriptsize{$(ZW)(XY)$}};
\node (15) at (1,1) {\scriptsize{$((ZW)X)Y$}};
\node (16) at (2,1) {\scriptsize{$(Z(WX))Y$}};
\node (17) at (4,1) {\scriptsize{$((ZX)W)Y$}};
\node (18) at (0,0) {\scriptsize{$Z(W(XY))$}};
\node (19) at (2,0) {\scriptsize{$Z((WX)Y)$}};
\node (20) at (3,0) {\scriptsize{$Z((XW)Y)$}};
\node (21) at (4,0) {\scriptsize{$(Z(XW))Y$}};
%2-arrows
\node (c_1) at (3.5,3.5) {$\cong$};\node (c_2) at (2.5,0.4) {$\cong$};
\node (p1) at (0.5,3.4) {\scriptsize{$\pi^*{}^{-1}_{(X,Y,Z,W)}$}};\node (p1') at (0.5,3.6){$\Ra$};
\node (h1) at (2,3.4) {\scriptsize{$X\gotikh_{1(Y,Z,W)}$}};\node (h1') at (2,3.6){$\Ua$};
\node (h2) at (1,1.9) {\scriptsize{$X\gotikh_{2(X,Y,ZW)}$}};\node (h2') at (1,2.1){$\Ra$};
\node (p2) at (3,2.4) {\scriptsize{$\pi^*{}^{-1}_{(X,Z,W,Y)}$}};\node (p2') at (3,2.6){$\Ra$};
\node (h3) at (3,1.1) {\scriptsize{$\gotikh_{1(X,Z,W)}$Y}};\node (h3') at (3,1.3){$\Ra$};
\node (p3) at (1,0.4) {\scriptsize{$\pi^*{}^{-1}_{(Z,W,X,Y)}$}};\node (p3') at (1,0.6){$\Ua$};
%edges
\path[->, font=\scriptsize]
(1) edge (2)(1) edge (4)(1) edge (7)(2) edge (3)(3) edge (5)(3) edge (9)(4) edge (6)(5) edge (10)(6) edge (11)(7) edge (8)(7) edge (11)(8) edge (9)(8) edge (12)(9) edge (10)(10) edge (13)(11) edge (14)(12) edge (15)(13) edge (12)(13) edge (17)(14) edge (18)(14) edge (15)(15) edge (16)(16) edge (19)(17) edge (21)(18) edge (19)(20) edge (19)(21) edge (20)(21) edge (16);
\end{tikzpicture}
\end{equation}
%%%%%%tikzpicture%%%%%%%%%

Moreover the modifications $\gotikh_1$ and $\gotikh_2$ are compatible with each other under the above comparison, i.e. the pasting of the 2-arrows in the diagram below, denoted by $\gotikh_1 \square \gotikh_2$, is the identity
\begin{equation}\label{diagram:hexagone1_vs_hexagone2}
\begin{tikzpicture}[baseline=(current bounding box.center), descr/.style={fill=white,inner sep=2.5pt}, scale=1.75]
%nodes
\node (1) at (0,0) {$(XY)Z$};\node (2) at (1,1) {$X(YZ)$};\node (3) at (2,1) {$(YZ)X$};\node (4) at (1,-1) {$(YX)Z$};\node (5) at (2,-1) {$Y(XZ)$};\node (6) at (3,0) {$Y(ZX)$};\node (7) at (4,1) {$(YZ)X$};\node (8) at (5,1) {$X(YZ)$};\node (9) at (4,-1) {$Y(XZ)$};\node (10) at (5,-1) {$(YX)Z$};\node (11) at (6,0) {$(XY)Z$};
%edges
\draw[->,rounded corners](0,0.15) -- (0,2.5) -- (6,2.5) -- (6,0.15);
\draw[->,rounded corners](1,1.15) -- (1,1.75) -- (5,1.75) -- (5,1.15);
\draw[->,rounded corners](0,-0.15) -- (0,-2.5) -- (6,-2.5) -- (6,-0.15);
\draw[->,rounded corners](1,-1.15) -- (1,-1.75) -- (5,-1.75) -- (5,-1.15);
\path[->, font=\scriptsize]
(1) edge (2)(1) edge (4)(2) edge (3)(4) edge (5)(3) edge (6)(5) edge (6)(6) edge (7)(6) edge (9)(7) edge (8)(9) edge (10)(8) edge (11)(10) edge (11)(3) edge (7)(5) edge (9);
%2-arrows
\node (x) at (1.5,0) {\scriptsize{$\kesir{\mbox{$\Da$}}{\gotikh_{1(X,Y,Z)}}$}};
\node (x) at (4.5,0) {\scriptsize{$\kesir{\mbox{$\Da$}}{\gotikh_{2(Y,Z,X)}}$}};
\node (x) at (3,-0.6) {\scriptsize{$\kesir{\mbox{$\Da$}}{Y\zeta^{-1}_{(X,Z)}}$}};
\node (x) at (3,1.4) {\scriptsize{$\kesir{\mbox{$\Da$}}{\zeta_{(X,YZ)}}$}};
\node (x) at (3,-2.2) {\scriptsize{$\kesir{\mbox{$\Da$}}{\zeta^{-1}_{(X,Y)}Z}$}};
\node (x) at (3,2.2) {\scriptsize{$\simeq$}};
\node (x) at (3,-1.4) {\scriptsize{$\simeq$}};
\node (x) at (3,0.6) {\scriptsize{$\simeq$}};
\end{tikzpicture}
\end{equation}
%%%%%%tikzpicture%%%%%%%%%
and an analogous pasting of 2-arrows, denoted by $\gotikh_2 \square \gotikh_1$, is the identity 
%%%%%%tikzpicture%%%%%%%%%
\begin{equation}\label{diagram:hexagone1_vs_hexagone2_bis}
\begin{tikzpicture}[baseline=(current bounding box.center), descr/.style={fill=white,inner sep=2.5pt}, scale=1.75]
%nodes
\node (1) at (0,0) {$X(YZ)$};\node (2) at (1,1) {$(XY)Z$};\node (3) at (2,1) {$Z(XY)$};\node (4) at (1,-1) {$X(ZY)$};\node (5) at (2,-1) {$(XZ)Y$};\node (6) at (3,0) {$(ZX)Y$};\node (7) at (4,1) {$Z(XY)$};\node (8) at (5,1) {$(XY)Z$};\node (9) at (4,-1) {$(XZ)Y$};\node (10) at (5,-1) {$X(ZY)$};\node (11) at (6,0) {$X(YZ)$};
%edges
\draw[->,rounded corners](0,0.15) -- (0,2.5) -- (6,2.5) -- (6,0.15);
\draw[->,rounded corners](1,1.15) -- (1,1.75) -- (5,1.75) -- (5,1.15);
\draw[->,rounded corners](0,-0.15) -- (0,-2.5) -- (6,-2.5) -- (6,-0.15);
\draw[->,rounded corners](1,-1.15) -- (1,-1.75) -- (5,-1.75) -- (5,-1.15);
\path[->, font=\scriptsize]
(1) edge (2)(1) edge (4)(2) edge (3)(4) edge (5)(3) edge (6)(5) edge (6)(6) edge (7)(6) edge (9)(7) edge (8)(9) edge (10)(8) edge (11)(10) edge (11)(3) edge (7)(5) edge (9);
%2-arrows
\node (x) at (1.5,0) {\scriptsize{$\kesir{\mbox{$\Da$}}{\gotikh_{2(X,Y,Z)}}$}};
\node (x) at (4.5,0) {\scriptsize{$\kesir{\mbox{$\Da$}}{\gotikh_{1(Z,X,Y)}}$}};
\node (x) at (3,-0.6) {\scriptsize{$\kesir{\mbox{$\Da$}}{\zeta^{-1}_{(X,Z)}Y}$}};
\node (x) at (3,1.4) {\scriptsize{$\kesir{\mbox{$\Da$}}{\zeta_{(XY,Z)}}$}};
\node (x) at (3,-2.2) {\scriptsize{$\kesir{\mbox{$\Da$}}{X\zeta^{-1}_{(Y,Z)}}$}};
\node (x) at (3,2.2) {\scriptsize{$\simeq$}};
\node (x) at (3,-1.4) {\scriptsize{$\simeq$}};
\node (x) at (3,0.6) {\scriptsize{$\simeq$}};
\end{tikzpicture}
\end{equation}
%%%%%%tikzpicture%%%%%%%%%
Finally using the terminology of Kapranov and Voevodsky in \cite{MR1278735}, we require that the 2-arrows defining the Z-systems coincide, i.e. for all $X,Y,Z \in \twoP(U)$ the following equation of 2-arrows holds

%%%%%%tikzpicture%%%%%%%%%
\begin{equation}\label{diagram:Z_system}
\begin{tikzpicture}[baseline=(current bounding box.center), descr/.style={fill=white,inner sep=2.5pt}, scale=1.75]
%left side
\node (1) at (1.5,1.5) {\scriptsize{$(XZ)Y$}};\node (2) at (1,1) {\scriptsize{$X(ZY)$}};\node (3) at (2,1) {\scriptsize{$(ZX)Y$}};\node (4) at (0.5,0.5) {\scriptsize{$X(YZ)$}};\node (5) at (2.5,0.5) {\scriptsize{$Z(XY)$}};\node (6) at (0,0) {\scriptsize{$(XY)Z$}};\node (7) at (3,0) {\scriptsize{$Z(YX)$}};\node (8) at (0.5,-0.5) {\scriptsize{$(YX)Z$}};\node (9) at (2.5,-0.5) {\scriptsize{$(ZY)X$}};\node (10) at (1,-1) {\scriptsize{$Y(XZ)$}};\node (11) at (2,-1) {\scriptsize{$(YZ)X$}};\node (12) at (1.5,-1.5) {\scriptsize{$Y(ZX)$}};
\node (0) at (3.5,0) {$=$};
%right side
\node (1') at (5.5,1.5) {\scriptsize{$(XZ)Y$}};\node (2') at (5,1) {\scriptsize{$X(ZY)$}};\node (3') at (6,1) {\scriptsize{$(ZX)Y$}};\node (4') at (4.5,0.5) {\scriptsize{$X(YZ)$}};\node (5') at (6.5,0.5) {\scriptsize{$Z(XY)$}};\node (6') at (4,0) {\scriptsize{$(XY)Z$}};\node (7') at (7,0) {\scriptsize{$Z(YX)$}};\node (8') at (4.5,-0.5) {\scriptsize{$(YX)Z$}};\node (9') at (6.5,-0.5) {\scriptsize{$(ZY)X$}};\node (10') at (5,-1) {\scriptsize{$Y(XZ)$}};\node (11') at (6,-1) {\scriptsize{$(YZ)X$}};\node (12') at (5.5,-1.5) {\scriptsize{$Y(ZX)$}};
%2-arrows
\node (X) at (1.5,0) {\scriptsize{$\cong$}};
\node (x) at (5.5,0) {\scriptsize{$\cong$}};
\node (Y) at (1.9,0.65) {\rotatebox{225}{$\Ra$}};\node (Y') at (2.08,0.33) {\scriptsize{$\gotikh_{1(X,Z,Y)}^{-1}$}};
\node (Z) at (1.1,-0.5) {\rotatebox{225}{$\Ra$}};\node (Z') at (1.25,-0.67) {\scriptsize{$\gotikh_{1(X,Y,Z)}$}};
\node (y) at (5.5,0.5) {\rotatebox{315}{$\Ra$}};\node (y') at (5.65,0.75) {\scriptsize{$\gotikh_{2(X,Y,Z)}^{-1}$}};
\node (z) at (5.7,-0.7) {\rotatebox{315}{$\Ra$}};\node (z') at (5.8,-0.45) {\scriptsize{$\gotikh_{2(Y,X,Z)}$}};
%edges
\path[->, font=\scriptsize]
(1) edge (3)(2) edge (1)(2) edge (9)(3) edge (5)(4) edge (2)(4) edge (11)(5) edge (7)(6) edge (8)(6) edge (4)(8) edge (10)(9) edge (7)(10) edge (12)(11) edge (9)(12) edge (11)(1') edge (3')(2') edge (1')(5') edge (7')(6') edge (5')(6') edge (4')(6') edge (8')(3') edge (5')(4') edge (2')(8') edge (7')(8') edge (10')(9') edge (7')(10') edge (12')(11') edge (9')(12') edge (11');
\end{tikzpicture}
\end{equation}
%%%%%%tikzpicture%%%%%%%%%
\item the modification $\eta$ satisfies the following two compatibility conditions:
the first one is that $\eta \ast \eta=\zeta$, the second one is that for all $X,Y \in \twoP(U)$ there is an additive relation between $\eta_X$,$\eta_Y$ and $\eta_{X Y}$, i.e. $\eta_{XY}$ is equal to the pasting of the 2-arrows in the following diagram

%%%%%%tikzpicture%%%%%%%%%
\begin{equation}\label{diagram:additivity_of_braiding}
\begin{tikzpicture}[baseline=(current bounding box.center), descr/.style={fill=white,inner sep=2.5pt}, scale=2]
%nodes left side
\node (1) at (1.5,0.5) {\scriptsize{$X((XY)Y)$}};
\node (2) at (3.5,0.5) {\scriptsize{$(X(XY))Y$}};
\node (3) at (1,1) {\scriptsize{$X(Y(XY))$}};
\node (4) at (2,1) {\scriptsize{$X(X(YY))$}};
\node (5) at (3,1) {\scriptsize{$((XX)Y)Y$}};
\node (6) at (4,1) {\scriptsize{$((XY)X)Y$}};
\node (7) at (2.5,1.5) {\scriptsize{$(XY)(XY)$}};
\node (8) at (1,2) {\scriptsize{$X((YX)Y)$}};
\node (9) at (2,2) {\scriptsize{$X(X(YY))$}};
\node (10) at (3,2) {\scriptsize{$((XX)Y)Y$}};
\node (11) at (4,2) {\scriptsize{$(X(YX))Y$}};
\node (12) at (1.5,2.5) {\scriptsize{$X((XY)Y)$}};
\node (13) at (3.5,2.5) {\scriptsize{$(X(XY)Y$}};
%nodes right side
\node (0) at (4.5,1.5) {\scriptsize{$=$}};
\node (1') at (5.5,0.5) {\scriptsize{$X((XY)Y)$}};
\node (2') at (7.5,0.5) {\scriptsize{$(X(XY))Y$}};
\node (3') at (5,1) {\scriptsize{$X(Y(XY))$}};
\node (4') at (6,1) {\scriptsize{$(XY)(XY)$}};
\node (5') at (7,1) {\scriptsize{$(XY)(XY)$}};
\node (6') at (8,1) {\scriptsize{$((XY)X)Y$}};
\node (7') at (5,2) {\scriptsize{$X((YX)Y)$}};
\node (8') at (8,2) {\scriptsize{$(X(YX))Y$}};
\node (9') at (6.5,2) {\scriptsize{$(X(YX))Y$}};
\node (10') at (5.5,2.5) {\scriptsize{$X((XY)Y)$}};
\node (11') at (7.5,2.5) {\scriptsize{$(X(XY)Y$}};
%2-arrows
\node (c_1) at (2.08,1.5) {\scriptsize{$\cong$}}; \node (c_2) at (2.92,1.5) {\scriptsize{$\cong$}}; \node (c_3) at (6,2.2) {\scriptsize{$\cong$}};
\node (a) at (1.88,1.55) {$\Ra$};\node (a') at (1.88,1.4) {\scriptsize{$\eta_{Y}^{-1}$}};
\node (b) at (3.15,1.55) {$\Ra$};\node (b') at (3.15,1.4) {\scriptsize{$\eta_{X}$}};
\node (c) at (1.4,1.5) {\rotatebox{45}{$\Ra$}};\node (c') at (1.4,1.3) {\scriptsize{$X\gotikh_{1(Y,X,Y)}$}};
\node (d) at (3.65,1.5) {\rotatebox{135}{$\Ra$}};\node (d') at (3.65,1.3) {\scriptsize{$\gotikh_{1(X,X,Y)}Y$}};
\node (e) at (2.5,1) {$\Ua$};\node (e') at (2.5,0.8) {\scriptsize{$\pi^*{}^{-1}_{(X,X,Y,Y)}$}};
\node (e) at (2.5,2) {$\Ua$};\node (e') at (2.5,2.2) {\scriptsize{$\pi^*_{(X,X,Y,Y)}$}};
\node (d) at (6.5,0.8) {$\Ua$};\node (d') at (6.5,0.7) {\scriptsize{$\gotikh^{-1}_{2(X,X,XY)}$}};
\node (d) at (6.3,1.2) {$\Ua$};\node (d') at (6.6,1.2) {\scriptsize{$\eta^{-1}_{(XY)}$}};
\node (d) at (6.5,1.7) {$\Ua$};\node (d') at (6.5,1.55) {\scriptsize{$\pi^{*}_{(X,Y,X,Y)}$}};
\node (c_3) at (7.2,2.2){$\Ua$};\node (c_3) at (7.5,2.2) {\scriptsize{$\zeta_{(Y,X)}$}};
%edges
\path[->, font=\scriptsize]
(1) edge (2)(1) edge (4)(2) edge (6)(3) edge (1)(4) edge (7)(5) edge (2)(5) edge[bend right=43]  (10) (5) edge (10)(6) edge (11)(7) edge (5)(7) edge (10)(8) edge (3)(8) edge (12)(9) edge (4)(9) edge[bend right=43] (4)(9) edge (7)(10) edge (13)(12) edge (9)(12) edge (13)(13) edge (11)(1') edge (2')(2') edge (6')(3') edge (1')(3') edge (4')(4') edge (5')(4') edge[bend left=55] (5')(5') edge(6') (6') edge (8')(7') edge (9')(7') edge (3')(7') edge (10')(9') edge (8')(9') edge (11')(10') edge (11')(11') edge (8');
\end{tikzpicture}
\end{equation}
 \end{enumerate}

Picard $\ES$-2-stacks over $\ES$ form a 3-category $2\PICARD$ whose objects are Picard $\ES$-2-stacks and whose hom-2-groupoid consists of additive 2-functors, morphisms of additive 2-functors, and modifications of morphisms of additive 2-functors (see \cite[\S 3]{beta}).

The automorphisms $\Aut(e)$ of the neutral object of a Picard $\ES$-2-stack form a Picard $\ES$-stack. The homotopy groups $\pi_i(\twoP)$ of a Picard $\ES$-2-stack $\twoP$ are
\begin{itemize}
	\item $ \pi_0(\twoP)$ which is the sheafification of the pre-sheaf which associates, to each object $U$ of $\ES$, the group of equivalence classes of objects of $\twoP(U)$;
	\item $\pi_1(\twoP) =\pi_0 (\Aut(e))$, with $\pi_0 (\Aut(e))$ the sheafification of the pre-sheaf which associates, to each object $U$ of $\ES$, the group of isomorphism classes of objects of $\Aut(e)(U)$;
	\item $\pi_2(\twoP) =\pi_1 (\Aut(e))$, with  $\pi_1(\Aut(e))$ the sheaf of automorphisms of the neutral object of $\Aut(e)$.
\end{itemize}

We will denote by $\mathbf{0}$ the Picard $\ES$-2-stack whose only object is the neutral object and whose only 1- and 2-arrows are the identities. The complex $[\mathbf{0}]^{\flat\flat}$ of $\twocatD$ corresponding to the Picard $\ES$-2-stack $\mathbf{0}$ via the equivalence of categories $2\st^{\flat\flat}$ (\ref{intro:2st_flat_flat}) is $\mathbf{E}=[\mathbf{e} \stackrel{id_\mathbf{e}}{\rightarrow}  \mathbf{e} \stackrel{id_\mathbf{e}}{\rightarrow} \mathbf{e}]$ with $\mathbf{e}$ the final object of the category of abelian sheaves on $\ES$.

%--------------------------------------------------------------------------------------------
\section{The 3-category $\TORS(\twoG)$ of $\twoG$-torsors}

In this section, we categorify the notion of $\oneG$-torsors where $\oneG$ is a gr-$\ES$-stack (see  \cite{MR1086889}). We define in detail the 3-category of $\twoG$-torsors where $\twoG$ is a gr-$\ES$-2-stack. At the end of the section, using the triequivalence (\ref{intro:2st}), we give without details a description of how to define the notion of torsor in terms of length 3-complexes of abelian sheaves.

{\subsection{Geometric Case}
As in Section \ref{section:recall_on_Picard_2_stacks}, in the following definitions $U$ will denote an object of the site $\ES$ and in the diagrams involving 2-arrows, we will put the symbol $\cong $ in the cells which commute up to a modification of $\ES$-2-stacks coming from the group like structure.

Let $\twoG=(\twoG,\otimes,\ass,\pi)$ be a gr-$\ES$-2-stack. For simplicity instead of $g_1 \otimes g_2 $ we will write just $g_1g_2$ for all $ g_1,g_2 \in \twoG(U)$. The equivalences of $\ES$-2-stacks $g \otimes -: \twoG \rightarrow \twoG$ and $-\otimes g : \twoG \rightarrow \twoG$ imply that any gr-$\ES$-2-stack admits a global neutral object $1_{\twoG}$ (denoted simply by 1) endowed with two natural 2-transformations of $\ES$-2-stacks $\gotikl: e \otimes - \Ra \id$ and $\gotikr: -\otimes e \Ra \id$, which express the left and the right unit constraints, and which satisfy some higher compatibility conditions (see \cite{MR2342833}).
  
\begin{definition}\label{def:right_torsor_over_a_group_like_2_stack}
A \emph{right $\twoG$-torsor} is given by a collection $\twoP=(\twoP,m,\mu,\Theta)$ where
\begin{itemize}
\item $\twoP$ is an $\ES$-2-stack;
\item $m: \twoP \times \twoG \ra \twoP$ is a morphism of $\ES$-2-stacks, called the \emph{action of $\twoG$ on $\twoP$}. For simplicity instead of $m(p,g)$ we write just $p.g$ for any $(p,g) \in \twoP \times \twoG(U)$;
\item $\mu: m \circ (\id_\twoP \times \otimes) \Ra m \circ (m \times \id_\twoG)$ is a natural 2-transformation of $\ES$-2-stacks whose component at $(p,g_1,g_2) \in \twoP \times \twoG^2(U)$ is the 1-arrow $\mu_{(p,g_1,g_2)}:p.(g_1g_2) \ra (p.g_1).g_2$ of $\twoP(U)$ and which expresses the compatibility between the group law $\otimes$ of $\twoG$ and the action $m$ of $\twoG$ on $\twoP$;
\item $\Theta$ is a modification of $\ES$-2-stacks whose component at $(p,g_1,g_2,g_3) \in \twoP \times \twoG^3(U)$ is the 2-arrow
\begin{equation*}\label{diagram:right_G_torsor_1}
\xymatrix@!=0.5cm{ &p.((g_1g_2)g_3)\ar[dl]_{\mu_{(p,g_1g_2,g_3)}} \ar[dr]^{p.\ass_{(g_1,g_2,g_3)}} \ar@{}[dd]|(0.6){\kesir{\mbox{$\La$}}{\Theta_{(p,g_1,g_2,g_3)}}}&\\
**[l](p.  (g_1g_2)).g_3 \ar[d]_{\mu_{(p,g_1,g_2)}g_3} & & **[r]p.(g_1(g_2g_3)) \ar[d]^{\mu_{(p,g_1,g_2g_3)}}\\
**[l]((p. g_1).g_2).g_3 && **[r](p.g_1).(g_2g_3) \ar[ll]^{\mu_{(p.g_1,g_2,g_3)}}}
\end{equation*}
and which expresses the obstruction to the compatibility between the natural 2-transformation $\mu$ and the associativity $\ass$ underlying $\twoG$ (i.e. the obstruction to the pentagonal axiom);
\end{itemize}
such that the following conditions are satisfied:
\begin{itemize}
\item $\twoP$ is locally equivalent to $\twoG$, i.e. $(m,\pr_{\twoP}): \twoP \times \twoG \ra \twoP \times \twoP$ is an equivalence of $\ES$-2-stacks (here $\pr_{\twoP}:\twoP \times \twoG \ra \twoP$ denotes the projection to $\twoP$);
\item $\twoP$ is locally not empty, i.e. it exists a covering sieve $R$ of the site $\ES$ such that for any object $V$ of $R$ the 2-category $\twoP(V)$ is not empty;
\item the modification $\Theta$ is coherent, i.e. it satisfies the coherence axiom of Stasheff's polytope (\ref{diagram:stasheff});
\item the restriction of $m$ to $\twoP \times 1_\twoG$ is equivalent to the identity, i.e. there exists a natural 2-transformation of $\ES$-2-stacks $\gotikd : m_{|(\twoP \times 1_{\twoG})} \Ra \id_{\twoP}$ whose component at $(p,1_{\twoG}) \in \twoP\times 1_{\twoG}(U)$ is the 1-arrow $\gotikd_{p}: p.1_{\twoG} \ra p$ of $\twoP(U)$. We require also the existence of two modifications of $\ES$-2-stacks $\gotikR$ and $\gotikL$, which express the obstruction to the compatibility between the restriction of $m$ to $\twoP \times 1_\twoG$ and the restrictions of $\mu$ to $\twoP \times \twoG \times 1_\twoG$ and $\twoP \times 1_\twoG \times \twoG$ respectively, and which satisfy three compatibility conditions: the first one is  between $\gotikL$ and $\gotikR$, the second one is between $\Theta$ and $\gotikR$, and the third one is between $\Theta$ and $\gotikL$. We left to the reader the explicit description of the modifications $\gotikR$ and $\gotikL$ with their compatibility conditions. 
\end{itemize}
\end{definition}

\begin{definition}\label{def:1_morphisms_of_right_torsors}
A \emph{morphism of right $\twoG$-torsors} from $\twoP=(\twoP,m_\twoP,\mu_\twoP,\Theta_\twoP)$ to $\twoQ=(\twoQ,m_\twoQ,\mu_\twoQ,\Theta_\twoQ)$ is given by the triplet $(F,\gamma,\Psi)$ where
\begin{itemize}
\item $F:\twoP \ra \twoQ$ is a morphism of $\ES$-2-stacks;
\item $\gamma: m_\twoQ \circ (F \times \id_{\twoG}) \Ra F \circ m_\twoP $ is a natural 2-transformation of $\ES$-2-stacks whose component at $(p,g) \in \twoP \times \twoG(U)$ is the 1-arrow $\gamma_{(p,g)}:Fp.g \ra F(p.g)$ (for simplicity we use the notation $.$ for both actions of $\twoG$ on $\twoP$ and on $\twoQ$) and which expresses the compatibility between the morphism of $\ES$-2-stacks $F$ and the two actions $m_\twoP$ and $m_\twoQ$ of $\twoG$ on $\twoP$ and on $\twoQ$;
\item $\Psi$ is a modification of $\ES$-2-stacks whose component at $(p,g_1,g_2) \in \twoP \times \twoG^2(U)$ is the 2-arrow
\begin{equation*}\label{diagram:1-morphism_of_right_G_torsor_1}
\xymatrix@!=0.5cm{ &Fp.(g_1g_2) \ar[dl]_{\gamma_{(p,g_1g_2)}} \ar@{}[dd]|(0.6){\kesir{\mbox{$\La$}}{\Psi_{(p,g_1,g_2)}}} \ar[dr]^{\mu_{\twoQ(Fp,g_1,g_2)}}&\\
**[l]F(p.(g_1g_2)) \ar[d]_{F(\mu_{\twoP(p,g_1,g_2)})} && **[r](Fp.g_1).g_2 \ar[d]^{\gamma_{(p,g_1)}.g_2}\\
**[l]F((p.g_1).g_2) && **[r]F(p.g_1).g_2 \ar[ll]^{\gamma_{(p.g_1,g_2)}}}
\end{equation*}
\end{itemize}
and which expresses the obstruction to the compatibility between the natural 2-transformation $\gamma$ and the natural 2-transformations $\mu_\twoP$ and $\mu_\twoQ$ underlying $\twoP$ and $\twoQ$. Moreover we require that the modification $\Psi$ is compatible with the modifications $\Theta_\twoP$ and $\Theta_\twoQ$, i.e. we have the following equation of 2-arrows
\[ F(\Theta_{\twoP,(p,g_1,g_2,g_3)}) * \Psi_{(p.g_1,g_2,g_3)}* 
\mu^{-1}_{(\gamma_{(p,g_1)},g_2,g_3)}* \Psi_{(p,g_1,g_2g_3)} * 
\gamma_{(p,\ass_{(g_1,g_2,g_3)})}= \]
\[ \Psi_{(p,g_1g_2,g_3)} * \gamma_{(\mu_{(p,g_1,g_2)},g_3)} * \Psi_{(p,g_1,g_2)}.g_3 * \Theta_{\twoQ(Fp,g_1,g_2,g_3)} .\]
\end{definition}

Let $(F,\gamma_F,\Psi_F)$ and $(G,\gamma_G,\Psi_G)$ be two morphisms of right $\twoG$-torsors from $\twoP$ to $\twoQ$.

\begin{definition}\label{def:2_morphisms_of_right_torsors}
 A \emph{2-morphism of right $\twoG$-torsors} from $(F,\gamma_F,\Psi_F)$ to $(G,\gamma_G,\Psi_G)$ is given by the pair $(\alpha,\Phi)$ where
\begin{itemize}
\item $\alpha:F \Ra G$ is a natural 2-transformation of $\ES$-2-stacks,
\item $\Phi$ is a modification of $\ES$-2-stacks whose components at $(p,g) \in \twoP \times \twoG(U)$ is the 2-arrow 
\begin{equation*}\label{diagram:2_morphisms_of_right_torsors_1}
\begin{tabular}{c}
\xymatrix@!=1.25cm{Fp.g \ar[r]^{\gamma_{F_{(p,g)}}} \ar[d]_{\alpha_p.g} \ar@{}[dr]|{\kesir{\mbox{$\Da$}}{\Phi_{(p,g)}}}& F(p.g) \ar[d]^{\alpha_{p.g}} \\Gp.g \ar[r]_{\gamma_{G_{(p,g)}}}& G(p.g)}
\end{tabular}
\end{equation*}
\end{itemize}
and which expresses the obstruction to the compatibility between the natural 2-transformation $\alpha$ and the natural 2-transformations $\gamma_F$ and $\gamma_G$ underlying $F$ and $G$. We require that the modification $\Phi$ is compatible with the modifications $\Psi_F$ and $\Psi_G$, i.e. we have the following equation of 2-arrows
\[\Phi_{(p,g_1g_2)}*\alpha_{\mu_{(p,g_1,g_2)}}*\Psi_{{F}_{(p,g_1,g_2)}}= \Psi_{G_{{(p,g_1,g_2)}}} * \Phi_{(p.g_1,g_2)}* \Phi_{(p,g_1)}.g_2 *
\mu^{-1}_{(\alpha_p,g_1,g_2)}.\]
\end{definition}

Let $(\alpha,\Phi_\alpha)$ and $(\beta, \Phi_\beta)$ be two 2-morphisms of right $\twoG$-torsors from $(F,\gamma_F,\Psi_F): \twoP \ra \twoQ$ to $(G,\gamma_G,\Psi_G): \twoP \ra \twoQ$.

\begin{definition}\label{def:3_morphisms_of_right_torsors}
A \emph{3-morphism of right $\twoG$-torsors} from $(\alpha,\Phi_\alpha)$ to $(\beta, \Phi_\beta)$ is given by a modification of $\ES$-2-stacks $\Delta:\alpha \Rra \beta$ which is compatible with the modifications $\Phi_\alpha$ and $\Phi_\beta$, i.e. $\Phi_{\beta(p,g)}*\Delta_{p.g}=\Delta_{p}.g*\Phi_{\alpha(p,g)}.$
\end{definition}

If the gr-$\ES$-2-stack $\twoG$ acts on the left side instead of the right side, we get the definitions of left $\twoG$-torsor, morphism of left $\twoG$-torsors, 2-morphism of left $\twoG$-torsors and 3-morphism of left $\twoG$-torsors. 

\begin{definition}\label{def:torsor}
A \emph{$\twoG$-torsor} $\twoP=(\twoP,m^l,m^r,\mu^l,\mu^r,\Theta^l,\Theta^r, \kappa, \Omega^r,\Omega^l)$ consists of an ${\ES}$-2-stack $\twoP$ endowed with a structure of left $\twoG$-torsor $(\twoP,m^l,\mu^l,\Theta^l)$ and with a structure of right $\twoG$-torsor $(\twoP,m^r,\mu^r,\Theta^r)$ which are compatible with each other. This compatibility is given by a natural 2-transformation of $\ES$-2-stacks $\kappa: m^l \circ (\id_{\twoG} \times m^r)  \Ra  m^r \circ (m^l \times \id_{\twoG})$ whose component at $(g_1,p,g_2) \in \twoG \times \twoP \times \twoG(U)$ is the 1-arrow $\kappa_{(g_1,p,g_2)} : g_1.(p.g_2) \ra (g_1.p).g_2$. We require also the existence of two modifications of $\ES$-2-stacks, $\Omega^l$ whose component at $(g_1,g_2,p,g_3) \in \twoG^2 \times \twoP \times \twoG(U)$ is the 2-arrow 
\begin{equation*}\label{diagram:G_torsor_1}
\begin{tabular}{c}
\xymatrix@!=0.5cm{**[l](g_1g_2).(p.g_3) \ar[rr]^{\kappa_{(g_1g_2,p,g_3)}} \ar[d]_{\mu^l_{{(g_1,g_2,p.g_3)}}} &\ar@{}[dd]|{\kesir{\mbox{$\La$}}{\Omega^l_{(g_1,g_2,p,g_3)}}}& **[r]((g_1g_2).p).g_3 \ar[d]^{\mu^l_{(g_1,g_2,p)}.g_3} \\
**[l]g_1.(g_2.(p.g_3)) \ar[dr]_{g_1.\kappa_{(g_2,p,g_3)}} && **[r](g_1.(g_2.p)).g_3\\
&g_1.((g_2.p).g_3) \ar[ur]_{\kappa_{(g_1,g_2.p,g_3)}}&}
\end{tabular}
\end{equation*}
and $\Omega^r$ whose component at $(g_1,p,g_2,g_3) \in \twoG \times \twoP \times \twoG^2(U)$ is the 2-arrow
\begin{equation*}\label{diagram:G_torsor_2}
\begin{tabular}{c}
\xymatrix@!=0.5cm{**[l]g_1.(p.(g_2g_3)) \ar[rr]^{\kappa_{(g_1,p,g_2g_3)}} \ar[d]_{g_1.\mu^r_{(p,g_2,g_3)}} &\ar@{}[dd]|{\kesir{\mbox{$\La$}}{\Omega^r_{(g_1,p,g_2,g_3)}}}& **[r](g_1.p).(g_2g_3) \ar[d]^{\mu^r_{(g_1.p,g_2,g_3)}} \\
**[l]g_1.((p.g_2).g_3) \ar[dr]_{\kappa_{(g_1,p.g_2,g_3)}} && **[r]((g_1.p).g_2).g_3  \\
&(g_1.(p.g_2)).g_3 \ar[ur]_{\kappa_{(g_1,p,g_2)}.g_3}&}
\end{tabular}
\end{equation*}
which express the obstruction to the compatibility between the natural 2-transformation $\kappa$ and the natural 2-transformations $\mu^l$ and $\mu^r$ respectively. Moreover $\Omega^r$ and $\Omega^l$ satisfy three compatibility conditions: the first one is between $\Omega^r$ and $\Theta^r$, the second one is between  $\Omega^l$ and $\Theta^l$, and the third one is between $\Omega^r$ and $\Omega^l$. 
\end{definition}

Any gr-$\ES$-2-stack $\twoG=(\twoG,\otimes,\ass,\pi)$ is a left $\twoG$-torsor and a right $\twoG$-torsor: the action of $\twoG$ on $\twoG$ is just the group law $\otimes $ of $\twoG,$ the natural 2-transformation $\mu$ is the associativity $\ass$ and the modification $\Theta$ is $\pi$. Any Picard $\ES$-2-stack $\twoG$ is a $\twoG$-torsor: in fact, the gr-structure underlying $\twoG$ furnishes the structures of left and right $\twoG$-torsor and the braiding implies that these two structures are compatible.

Let $\twoP=(\twoP,m^l_\twoP,m^r_\twoP,\mu^l_\twoP,\mu^r_\twoP,\Theta^l_\twoP,\Theta^r_\twoP, \kappa_\twoP, \Omega^r_\twoP,\Omega^l_\twoP)$ and $\twoQ=(\twoQ,m^l_\twoQ,m^r_\twoQ,\mu^l_\twoQ,\mu^r_\twoQ,\Theta^l_\twoQ,\Theta^r_\twoQ, \kappa_\twoQ, \Omega^r_\twoQ,\Omega^l_\twoQ)$ be two $\twoG$-torsors.

\begin{definition}\label{def:1_morphisms_of_torsors}
A \emph{morphism of $\twoG$-torsors} from $\twoP$ to $\twoQ$ consists of the collection  $(F,\gamma^l,\gamma^r,\Psi^l,\Psi^r,\Sigma)$ where 
\begin{itemize}
	\item $(F,\gamma^l,\Psi^l): (\twoP,m^l_\twoP,\mu^l_\twoP,\Theta^l_\twoP) \rightarrow (\twoQ,m^l_\twoQ,\mu^l_\twoQ,\Theta^l_\twoQ)$ and $(F,\gamma^r,\Psi^r): (\twoP,m^r_\twoP,\mu^r_\twoP,\Theta^r_\twoP) \rightarrow (\twoQ,m^r_\twoQ,\mu^r_\twoQ,\Theta^r_\twoQ) $ are morphisms of left and right $\twoG$-torsors respectively;
	\item $\Sigma$ is a modification of $\ES$-2-stacks whose component at $(g_1,p,g_2) \in \twoG \times \twoP \times \twoG (U)$ is the 2-arrow
	\[\Sigma_{(g_1,p,g_2)}: F(\kappa_{\twoP(g_1,p,g_2)}) \circ \gamma^l_{(p.g_2,g_1)} \circ g_1 . \gamma^r_{(p,g_2)} \Ra \gamma^r_{(g_1.p,g_2)} \circ \gamma^l_{(p,g_1)}. g_2 \circ \kappa_{\twoQ(g_1,Fp,g_2)}\]
	 and which expresses the obstruction to the compatibility between the natural 2-transformations $\gamma^l, \gamma^r, \kappa_\twoP$ and $\kappa_\twoQ$. Moreover we require that the modification $\Sigma$ is compatible with the modifications $\Psi^l, \Psi^r, \Omega^l$ and $\Omega^r$. We leave the explicit description of these compatibilities to the reader.
\end{itemize}
\end{definition}

Any morphism of $\twoG$-torsors $F:\twoP \ra \twoQ$ is an equivalence of $\ES$-2-stacks. Therefore,

\begin{definition}\label{def:equiv_of_torsors}
Two $\twoG$-torsors $\twoP$ and $\twoQ$ are \emph{equivalent as $\twoG$-torsors} if there exists a morphism of $\twoG$-torsors from $\twoP$ and $\twoQ$.
\end{definition}

Let $(F,\gamma^l_F,\gamma^r_F,\Psi^l_F,\Psi^r_F,\Sigma_F)$ and $(G,\gamma^l_G,\gamma^r_G,\Psi^l_G,\Psi^r_G,\Sigma_G)$ be two parallel morphisms of $\twoG$-torsors from $\twoP$ to $\twoQ$.

\begin{definition}\label{def:2_morphisms_of_torsors}
A \emph{2-morphism of $\twoG$-torsors} from $(F,\gamma^l_F,\gamma^r_F,\Psi^l_F,\Psi^r_F,\Sigma_F)$ to $(G,\gamma^l_G,\gamma^r_G,\Psi^l_G,\Psi^r_G,\Sigma_G)$ is given by the triplet $(\alpha,\Phi^l,\Phi^r)$ where $(\alpha,\Phi^l):(F,\gamma^l_F,\Psi^l_F)\Rightarrow (G,\gamma^l_G,\Psi^l_G) $ and $(\alpha,\Phi^r):(F,\gamma^r_F,\Psi^r_F)\Rightarrow (G,\gamma^r_G,\Psi^r_G) $ are 2-morphisms of left and right $\twoG$-torsors respectively. Moreover we require that the modifications $\Phi^l$ and $\Phi^r$ are compatible with the modifications $\Sigma_{F}$ and $\Sigma_{G},$ i.e. we have the following equation of 2-arrows
\[g_1.\Phi^r_{(p,g_2)} * \Phi^l_{(g_1,p.g_2)} * \alpha_{\kappa_{(g_1,p,g_2)}} * \Sigma_{F(g_1,p,g_2)} = \Sigma_{G(g_1,p,g_2)} * \Phi^r_{(g_1.p,g_2)} *
\Phi^l_{(g_1,p)}.g_2 * \kappa^{-1}_{(g_1,\alpha_p,g_2)} .\] 
\end{definition}
 
Let $(\alpha,\Phi^l_\alpha,\Phi^r_\alpha)$ and $(\beta,\Phi^l_\beta,\Phi^r_\beta)$ be two 2-morphisms of $\twoG$-torsors from $F$ to $G$.

\begin{definition}\label{def:3_morphisms_of_torsors}
A \emph{3-morphism of $\twoG$-torsors} from $(\alpha,\Phi^l_\alpha,\Phi^r_\alpha)$ to $(\beta,\Phi^l_\beta,\Phi^r_\beta)$ is given by a modification of  $\ES$-2-stacks $\Delta: \alpha \Rra \beta$ such that $\Delta :(\alpha,\Phi^l_\alpha) \Rra (\beta,\Phi^l_\beta)$ and $\Delta :(\alpha,\Phi^r_\alpha) \Rra (\beta,\Phi^r_\beta)$ are 3-morphisms of left and right $\twoG$-torsors respectively.
\end{definition}
 
\begin{def-prop}\label{def-prop:2-groupoid_of_torsors}
Let $\twoP$ and $\twoQ$ be $\twoG$-torsors. Then the 2-category $\HomTORS(\twoP,\twoQ)$ whose
\begin{itemize}
	\item objects are morphisms of $\twoG$-torsors from $\twoP$ to $\twoQ$ ,
	\item 1-arrows are 2-morphisms of $\twoG$-torsors,
	\item 2-arrows are 3-morphisms of $\twoG$-torsors,
\end{itemize}
is a 2-groupoid, called the 2-groupoid of morphisms of $\twoG$-torsors from $\twoP$ to $\twoQ$. 
\end{def-prop}

In Lemma \ref{lemma:HomTors(PP)isPicard} we show that $\HomTORS(\twoP,\twoP)$ is a Picard $\ES$-2-stack. In general we expect to have at least an $\ES$-2-stack structure on $\HomTORS(\twoP,\twoQ)$.

$\twoG$-torsors over $\ES$ form a 3-category $\TORS(\twoG)$ where the objects are $\twoG$-torsors and the hom-2-groupoid of two $\twoG$-torsors $\twoP$ and $\twoQ$ is $\HomTORS(\twoP,\twoQ).$

We define the sum of two $\twoG$-torsors $\twoP$ and $\twoQ$ as the fibered sum (or the push-down) of 
 $\twoP$ and $\twoQ$ under $\twoG$. In the context of torsors, the fibered sum is called the contracted product:

\begin{definition}\label{def:sum_G_torsors}
The \emph{contracted product $\twoP \wedge^{\twoG} \twoQ$ (or just $\twoP \wedge \twoQ$) of $\twoP$ and $\twoQ$} is the $\twoG$-torsor whose underlying $\ES$-2-stack is obtained by 2-stackyfying the following fibered 2-category in 2-groupoids $\twoD$: for any object $U$ of $\ES$,  
\begin{enumerate}
\item the objects of $\twoD(U)$ are the objects of the product $\twoP \times \twoQ(U)$, i.e. pairs $(p,q)$ with $p$ an object of $\twoP(U)$ and $q$ an object of $\twoQ(U)$;
\item a 1-arrow $(p_1,q_1) \ra (p_2,q_2)$ between two objects of $\twoD(U)$ is given by 
a triplet $(m,g,n)$ where $g$ is an object of $\twoG(U)$, $m: p_1.g \ra p_2$ is a 1-arrow in $\twoP(U)$ and $n: q_1 \ra g.q_2$ is a 1-arrow in $\twoQ(U)$;
\item a 2-arrow between two parallel 1-arrows $(m,g,n), (m',g',n'):(p_1,q_1) \ra (p_2,q_2)$ of $ \twoD(U)$ is given by an equivalence class of triplets $(\phi,l,\theta)$ with $l:g\ra g'$ a 1-arrow of $\twoG(U)$, $\phi: m' \circ p_1.l \Rightarrow m$ a 2-arrow of $\twoP(U)$ and $\theta: l .q_2 \circ n \Rightarrow n'$ a 2-arrow of $\twoQ(U)$. Two such triplets $(\phi,l,\theta)$ and $(\tilde{\phi},\tilde{l},\tilde{\theta})$ are equivalent if there exists a 2-arrow $\gamma: l \Ra \tilde{l}$ of $\twoG(U)$ such that $\tilde{\phi} * p_1.\gamma = \phi$ and $\gamma . q_2 * \tilde{\theta} = \theta$.
\end{enumerate}
\end{definition}

The contracted product of $\twoG$-torsors is endowed with a universal property similar to the one stated explicitly in \cite[Prop 10.1]{beta}.

\begin{proposition}\label{proposition:sumoftor}
Let $\twoG$ be a Picard $\ES$-2-stacks. The contracted product equips the set $\TORS^1(\twoG)$ of equivalence classes of $\twoG$-torsors with an abelian group law, where the neutral element is the equivalence class of the $\twoG$-torsor $\twoG$, and the inverse of the equivalence class of a $\twoG$-torsor $\twoP$ is the equivalence class of the $\mathrm{ad}(\twoP)$-torsor $\twoP$, with $\mathrm{ad}(\twoP)=\HomTORS(\twoP,\twoP)$ (recall that $\twoG$ and $\mathrm{ad}(\twoP)$ are equivalent via $g \ra (p \mapsto g . p)$).
\end{proposition}

\begin{definition}\label{definition:torsor_trivial}
A $\twoG$-torsor $\twoP$ is \emph{trivial} if $\twoP$ is globally equivalent as $\twoG$-torsor to $\twoG$ (recall that $\twoG$ is considered as a $\twoG$-torsor via its group law $\otimes: \twoG \times \twoG  \rightarrow \twoG$).
\end{definition}

In order to define the notion of $\twoG$-torsor over an $\ES$-2-stack, we need the definition of fibered product (or pull-back) for $\ES$-2-stacks. Let $\twoP, \twoQ$ and $\twoR$ be three $\ES$-2-stacks and consider two morphisms of $\ES$-2-stacks $F: \twoP \rightarrow \twoR$ and $G: \twoQ \rightarrow \twoR$.

\begin{definition}\label{def:fibered_product}
The \emph{fibered product of $\twoP$ and $\twoQ$ over $\twoR$} is the $\ES$-2-stack $\twoP \times_{\twoR} \twoQ$ defined as follows: for any object $U$ of $\ES$,
\begin{itemize}
\item an object of the 2-groupoid $(\twoP \times_{\twoR} \twoQ)(U)$ is a triple $(p,l,q)$ where $p$ is an object of $ \twoP(U)$, $q$ is an object of $ \twoQ(U)$ and $l:F p \ra G q$ is a 1-arrow in $\twoR(U)$;
\item a 1-arrow $(p_1,l_1,q_1) \ra (p_2,l_2,q_2)$ between two objects of $(\twoP \times_{\twoR} \twoQ)(U)$ is given by the triplet $(m,\alpha,n)$ where $m:p_1 \ra p_2$ and $n: q_1 \ra q_2$ are 1-arrows in $\twoP(U)$ and $\twoQ(U)$ respectively, and $\alpha \colon l_2 \circ Fm \Ra G n \circ l_1$ is a 2-arrow in $\twoR(U)$;
\item a 2-arrow between two parallel 1-arrows $(m,\alpha,n), (m',\alpha',n'):(p_1,l_1,q_1) \ra (p_2,l_2,q_2)$ of $(\twoP \times_{\twoR} \twoQ)(U)$ is given by the pair $(\theta,\phi)$ where $\theta:m \Ra m'$ and $\phi:n \Ra n'$ are 2-arrows in $\twoP(U)$ and $\twoQ(U)$ respectively, satisfying the equation $\alpha' \circ (l_2*F\theta) = (G\phi * l_1) \circ \alpha$ of 2-arrows.
\end{itemize}
\end{definition}

The fibered product $\twoP \times_\twoR \twoQ$ is also called the \emph{pull-back $F^*\twoQ$ of $\twoQ$ via $F:\twoP \ra \twoR$} or the \emph{pull-back $G^*\twoP$ of $\twoP$ via $G:\twoQ \ra \twoR$}. It satisfies a universal property similar to the one stated explicitly in \cite[\S4]{beta}. 

If $J: \twoE \ra \twoP$ is a morphism of $\ES$-2-stacks, the \emph{homotopy fiber} $\twoE_p$ of $\twoE$ over an object $p \in \twoP(U) $ (with $U$ an object of $\ES$) is the $\ES/U$-2-stack obtained as fibered product of $J:\twoE \ra \twoP$ and of the inclusion $p \ra \twoP$.

Let $\twoG$ be a gr-$\ES$-2-stack and let $\twoP$ be an $\ES$-2-stacks. Our next definition is inspired by the similar ones given in \cite[Expos\'e VII 1.1.2.1]{SGA7} and \cite[Def. 9.1]{Rousseau}.

\begin{definition}\label{def:torsorover2stack}
A \emph{$\twoG_\twoP$-torsor over $\twoP$ (or just $\twoG$-torsor over $\twoP$)} is an $\ES$-2-stack $\twoE$ endowed with a morphism of $\ES$-2-stacks $J:\twoE \ra \twoP$ so that for any object $U$ of $ \ES$ and for any $p \in \twoP(U)$, the homotopy fiber $\twoE_p$ over $p$ is a $\twoG(U)$-torsor (see Definition \ref{def:torsor}).
 \end{definition}

$\twoG_\twoP$-torsors over $\twoP$ form a 3-category, denoted $\TORS(\twoG_\twoP)$. 

Let $\twoP$ and $\twoR$ be two $\ES$-2-stacks and consider a morphism of $\ES$-2-stacks
$F: \twoR \rightarrow \twoP$. If $\twoQ$ is a $\twoG_\twoP$-torsor over $\twoP$, then the pull-back $F^*\twoQ$ of $\twoQ$ via $F: \twoR \rightarrow \twoP$ is a $\twoG_\twoR$-torsor over $\twoR$. In other words, the pull-back via $F: \twoR \rightarrow \twoP$ defines a 3-functor $F^*: \TORS(\twoG_\twoP) \longrightarrow \TORS(\twoG_\twoR).$

\subsection{Algebraic Case}

Let $G=[G^{-2} \ra G^{-1} \ra G^{0}]$ be a length 3 complex of sheaves of groups over $\ES$. We denote by $+: G \times G \ra G$ the morphism of complexes whose components are the operations on the groups $G^i$ for $i=-2,-1,0$.

\begin{definition}\label{def:right_torsor_via_fractions}
A \emph{right $G$-torsor} is given by a collection $P=(P,(q,M,p),(r,N,s),t)$ where
\begin{itemize}
\item $P=[P^{-2} \ra P^{-1} \ra P^{0}]$ is a length 3 complex of sheaves of sets;
\item $(q,M,p): P \times G \stackrel{q}{\la}M \stackrel{p}{\ra}P$ is a fraction, which we represent by $m:P \times G \ra G$;
\item $(r,N,s)$ is a 1-arrow from the composition of fractions $(q,M,p) \fcirc (\id_{P \times G \times G}, P \times G \times G, \id_{P} \times +)$ to the composition of the fractions $(q,M,p) \fcirc (q \times \id_G, M \times G, p \times \id_G)$ which can be depicted by the following commutative diagram
\begin{equation*}\label{diag:G-torsor-1}
\begin{tabular}{c}
\xymatrix{&& (P \times G^2)\times_{P \times G}M \ar[drr]^{p \circ \pr_2} \ar[dll]_{(\id_{P \times G \times G}) \circ \pr_1} && \\
P \times G \times G&& K \ar@{-->}[rr]^{s} \ar@{-->}[ll]_{r} \ar[u] \ar[d] && P\\
&& (M \times G )\times_{P \times G} M \ar[urr]_{p \circ \pr_2} \ar[ull]^{(q \times \id_G) \circ \pr_1} &&}
\end{tabular}
\end{equation*}
A more legible presentation of the 1-arrow $(r,N,s)$ would be the square
\begin{equation*}\label{diag:G-torsor-1.1}
\begin{tabular}{c}
\xymatrix{P \times G^2 \ar[r]^{\id_P \times +} \ar[d]_{m \times \id_G} \ar@{}[dr]|{{\kesir{\mbox{$\Da$}}{(r,N,s)}}}& P \times G \ar[d]^{m} \\P\times G \ar[r]_{m} & P}
\end{tabular}
\end{equation*}
where each arrow is a fraction.
\item $t$ is a 2-arrow of fractions which is the morphism of complexes from the vertical composition of the 1-arrow of fractions
\begin{equation*}
\xymatrix@!=0.75cm{&&& (P \times G^3) \times_{P \times G^2} ((P \times G^2)\times_{P \times G} M) \ar[ddrrr] \ar[ddlll] &&& \\
          &&& (P \times G^3) \times_{P \times G^2} K \ar[drrr]^{s_1}\ar[dlll]_{r_1} \ar@{-->}[u]_{u_1} \ar@{-->}[d]^{t_1} &&&\\
          P \times G^3 &&& ((P \times G^3) \times_{P \times G^2}(M \times G))\times_{P \times G} M \ar[rrr] \ar[lll] &&& P\\
          &&& (K\times G)\times_{P \times G} M \ar[urrr]_{s_1'} \ar[ulll]^{r_1'} \ar@{-->}[u]_{u_1'}\ar@{-->}[d]^{t_1'} &&&\\
          &&& ((M\times G^2)\times_{P \times G^2}(M \times G)) \times_{P \times G} M \ar[uurrr] \ar[uulll] &&&}
\end{equation*}
 to the vertical composition of the 1-arrow of fractions 
\begin{equation*}
\xymatrix@!=0.75cm{&&& (P \times G^3) \times_{P \times G^2} ((P \times G^2)\times_{P \times G} M) \ar[ddrrr] \ar[ddlll] &&& \\
          &&& (P \times G^3) \times_{P \times G^2} K \ar[drrr]^{s_2}\ar[dlll]_{r_2} \ar@{-->}[u]_{u_2} \ar@{-->}[d]^{t_2} &&&\\
          P \times G^3 &&& ((P \times G^3) \times_{P \times G^2}(M \times G))\times_{P \times G} M \ar[rrr] \ar[lll] &&& P\\
          &&& (K\times G)\times_{P \times G} M \ar[urrr]_{s_2'} \ar[ulll]^{r_2'} \ar@{-->}[u]_{u_2'}\ar@{-->}[d]^{t_2'} &&&\\
          &&& ((M\times G^2)\times_{P \times G^2}(M \times G)) \times_{P \times G} M \ar[uurrr] \ar[uulll] &&&}
\end{equation*}
The 2-arrow $t$ might be better understood if we represent it as a 3-morphism between the pasting of the 2-morphisms between the left and right diagrams below:
\begin{equation*}\label{diag:G-torsor-2}
\begin{tabular}{c}
\xymatrix@!=0.75cm{& P \times G^3 \ar@{}[dd]|{\kesir{\mbox{$\La$}}{(r,N,s)}}\ar[dr]|(0.4){\id_P\times(\id_G\times+)} \ar[dl]|(0.6){m \times \id_G \times \id_G} \ar[rr]^{\id_P \times (+ \times \id_G)}&& P \times G^2 \ar@{}[dl]|{\car}\ar[dr]|{\id_P \times +} &&& P \times G^3 \ar@{}[dr]|{\car}\ar[dl]|{m \times \id_G \times \id_G} \ar[rr]^{\id_P \times (+ \times \id_G)} && P \times G^2 \ar@{}[dd]|{\kesir{\mbox{$\La$}}{(r,N,s)}} \ar[dl]|{m \times \id_G} \ar[dr]|{\id_P \times +} &\\
P \times G^2 \ar[dr]|{m \times \id_G} && P \times G^2 \ar@{}[dr]|{\kesir{\mbox{$\Da$}}{(r,N,s)}}\ar[dl]|{m \times \id_G} \ar[rr]^{\id_P \times +} && P \times G \ar[dl]^{m} \ar@{}[r]|{\kesir{t}{\mbox{$\Rra$}}}&P \times G^2 \ar[dr]|{m \times \id_G} \ar[rr]^{\id_P \times F} && P \times G \ar@{}[dl]|{\kesir{\mbox{$\Da$}}{(r,N,s)}}\ar[dr]_{m} && P \times G \ar[dl]^{m}\\
& P \times G \ar[rr]_{m} && P&&& P \times G \ar[rr]_{m} && P}
\end{tabular}
\end{equation*}
\end{itemize}
\end{definition}
In order to define a right $G$-torsor using length 3 complexes we have substituted, in the Definition \ref{def:right_torsor_over_a_group_like_2_stack}, additive 2-functors by fractions, morphisms of additive 2-functors by 1-arrows of fractions, and modifications of morphisms of additive 2-functors by 2-arrows of fractions. One can find out the compatibility conditions, that the data underlying a right $G$-torsor have to satisfy, by applying the same arguments. Moreover, these arguments allow us to define 1-,2-, and 3-morphisms of right $G$-torsors. Hence, right $G$-torsors over $\ES$ form a 3-category. In a similar way we can define also left $G$-torsors. 

If $G$ is a length 3 complex of abelian sheaves, we can define the notion of $G$-torsor: it is a length 3 complex of sheaves of sets endowed with a structure of left $G$-torsor and with a structure of right $G$-torsor which are compatible with each other.  $G$-torsors over $\ES$ form a 3-category that we denote by $\TORS(G)$.

\begin{proposition}\label{proposition:triequivalence_of_extensions}
The triequivalence $2\st$ (\ref{intro:2st}) induces a triequivalence between $\TORS(\twoG)$ and $\TORS(G)$. 
\end{proposition}

%-----------------------------------------------------------------------------------

\section{Homological interpretation of $\twoG$-torsors}

Let $\twoG$ be a Picard $\ES$-2-stack. As observed at the end of Section 1, $[\mathbf{0}]^{\flat\flat}$ is the complex $\mathbf{E}=[\mathbf{e} \stackrel{id_\mathbf{e}}{\rightarrow}  \mathbf{e} \stackrel{id_\mathbf{e}}{\rightarrow} \mathbf{e}]$ of $\twocatD$ where $\mathbf{e}$ the final object of the category of abelian sheaves on $\ES$.

\begin{lemma}\label{lemma:HomTors(PP)isPicard}
For any $\twoG$-torsor $\twoP$, the Picard $\ES$-2-stack $\twoG$ is equivalent to $\HomTORS(\twoP,\twoP)$. In particular, $\HomTORS(\twoP,\twoP)$ is endowed with a Picard $\ES$-2-stack structure.
\end{lemma}
\begin{proof}
The additive 2-functor $ \twoG \rightarrow  \HomTORS(\twoP,\twoP), g \mapsto \big(p \mapsto g . p \big)$ furnishes the required equivalence. 
\end{proof}

By the above Lemma, the homotopy groups $\pi_i(\HomTORS(\twoP,\twoP))$ are abelian groups.
Since by definition $\TORS^{-i}(\twoG) =\pi_{i}(\HomTORS(\twoP,\twoP))$, we have

\begin{corollary} 
The sets $\TORS^{i}(\twoG)$, for $i=0,-1,-2$, are abelian groups.
\end{corollary}

\begin{proof}[Proof of Theorem \ref{thm:introtor} for i=0,-1,-2]
The Picard $\ES$-2-stack $\twoG$ is equivalent to the hom-2-groupoid $\Hom_{\ES-2-\mathrm{Stacks}}(\mathbf{0},\twoG)$ of morphisms of $\ES$-2-stacks from $\mathbf{0}$ to $\twoG$ via 
 the additive 2-functor $\twoG \rightarrow \Hom_{\ES-2-\mathrm{Stacks}}(\mathbf{0},\twoG), g \mapsto  \big(e \mapsto g \big)$. In particular, $\Hom_{\ES-2-\mathrm{Stacks}}(\mathbf{0},\twoG)$ is endowed with a Picard $\ES$-2-stack structure and $[\Hom_{\ES-2-\mathrm{Stacks}}(\mathbf{0},\twoG)]^{\flat\flat}= \tau_{\leq 0}{\mathrm{R}}\hhom(\mathbf{E},[\twoG]^{\flat\flat})$. By Lemma \ref{lemma:HomTors(PP)isPicard}, we have $\TORS^{i}(\twoG) = \pi_{-i}(\HomTORS(\twoP,\twoP)) \cong \pi_{-i}(\twoG) \cong \pi_{-i}(\Hom_{\ES-2-\mathrm{Stacks}}(\mathbf{0},\twoG)) = {\mathrm{H}}^i  \big(\tau_{\leq 0}{\mathrm{R}}\hhom(\mathbf{E},[\twoG]^{\flat\flat})\big) = {\mathrm{H}}^i  \big(\tau_{\leq 0}{\mathrm{R}}\Gamma([\twoG]^{\flat\flat})\big)=\HH^i([\twoG]^{\flat\flat})$.
\end{proof}

Before the proof of Theorem \ref{thm:introtor} for $i=1$, we record the following:

\begin{lemma}\label{lemma:generated_2_stack}
Let $\twoP$ be an $\ES$-2-stack. Then there exists a Picard $\ES$-2-stack $\mathbb{Z}[\twoP]$ whose fibers over any object  $U$ of $ \ES$ are the following 2-groupoids:
\begin{itemize}
\item an object of $\mathbb{Z}[\twoP](U)$ consists of a finite formal sum $\sum_{i\in I} n_i [p_i]$ with $n_i \in \mathbb{Z}$ and $p_i$ an object of $\twoP(U)$;
\item there exists a 1-morphism between any two objects $\sum_{i\in I} n_i [p_i]$ and $\sum_{j\in J} m_j [q_j]$ if $I=J$, $n_i=m_i$ for all $i \in I$, and there exists a morphism $f_i:p_i \ra q_i$ in $\twoP(U)$ for all $i \in I$. In this case, a 1-morphism $\sum_{i\in I} n_i [p_i] \ra \sum_{i\in I} n_i [q_i]$ is the finite formal sum 
$\sum_{i \in I}n_i[f_i]$; 
\item a 2-morphism between any two parallel 1-morphisms $\sum_{i \in I}n_i[f_i]$ and $\sum_{i \in I}n_i[g_i]$ from $\sum_{i\in I} n_i [p_i]$ to $\sum_{i\in I} n_i [q_i]$ is the finite formal sum $\sum_{i \in I}n_i[\alpha_i]$ where $\alpha_i$ is a 2-morphism in $\twoP(U)$ from $f_i$ to $g_i$ for all $i\in I$. 
\end{itemize}  
\end{lemma}

\begin{proof}
To verify that $\mathbb{Z}[\twoP]$ is a fibered 2-category in 2-groupoids over $\ES$ is straightforward. Let 
\begin{equation}\label{2-descent_datum_of_Z[P]}
(\sum_{i \in I}n_i[\varphi_i],\sum_{i \in I}n_i[\alpha_i])
\end{equation}
be a 2-descent datum for the object $\sum_{i\in I} n_i [p_i]$ of $\mathbb{Z}[\twoP](V_0)$ relative to the hypercover $\delta:V_{\bullet} \ra U$ where $\varphi_i:d_0^*p_i \ra d_1^*p_i$ is a 1-morphism in $\twoP(V_1)$ and $\alpha_i: d_1^*\varphi_i \Ra d_2^*\varphi_i \circ d_0^*\varphi_i$ is a 2-morphism in $\twoP(V_2)$. Since the collection (\ref{2-descent_datum_of_Z[P]}) satisfies the 2-cocycle condition 
\[
\big((d_2d_3)^* \sum_{i \in I}n_i[\varphi_i] * d_0^* \sum_{i \in I}n_i[\alpha_i]\big) \circ d_2^* \sum_{i \in I}n_i[\alpha_i]=\big(d_3^* \sum_{i \in I}n_i[\alpha_i] * (d_0d_1)^* \sum_{i \in I}n_i[\varphi_i]\big) \circ d_1^* \sum_{i \in I}n_i[\alpha_i]
\]
 so do the collections $(\varphi_i,\alpha_i)$ for all $i \in I$. This shows that for all $i \in I$, $(\varphi_i,\alpha_i)$ is a 2-descent datum for the object $p_i$ of $\twoP(V_0)$. Then for every $i \in I$ the 2-descent datum is effective, i.e. for every $i \in I$
it exists an object $q_i \in \twoP(U)$, a 1-morphism $\psi_i:\delta^*(q_i) \ra p_i$ in $\twoP(V_0)$, and a 2-morphism $\beta_i:\varphi_i \circ d_0^*\psi_i \Ra d_1^*\psi_i$ in $\twoP(V_1)$ so that the condition 
\[
(d_0^*\varphi_i * d_2^*\beta_i) \circ (d_0^*\beta_i * d_1^*\varphi_i) \circ (d_0^*d_0^*\psi_i * \alpha_i)=d_1^*\beta_i
\]
is satisfied. We observe that the formal sum of these effective data, i.e the collection $(\sum_{i \in I}n_i[q_i],\sum_{i \in I}n_i[\psi_i],\sum_{i \in I}n_i[\beta_i])$, is the effective data for the 2-descent datum (\ref{2-descent_datum_of_Z[P]}). We show using similar arguments that the finite formal sums of morphisms of $\twoP$ form an $\ES$-stack. Hence, $\mathbb{Z}[\twoP]$ is an $\ES$-2-stack. The Picard structure on $\mathbb{Z}[\twoP]$ is defined by concatenation.
\end{proof}

\begin{definition}\label{definition:Z[P]}
If $\twoP$ is an $\ES$-2-stack,  the \emph{Picard $\ES$-2-stack generated by $\twoP$} is the Picard $\ES$-2-stack $\mathbb{Z}[\twoP]$ constructed in Lemma \ref{lemma:generated_2_stack}.
\end{definition}

The Picard $\ES$-2-stack $\mathbb{Z}[\twoP]$ does not satisfy the universal property of a free object. Maybe the definition can be improved so that it works in the expected way, but this would be beyond the scope of the current paper.

\begin{lemma}\label{lemma:generated_Pic_2_stack}
If $\twoP=2\st\big([P^{-2} \ra  P^{-1} \ra P^0]\big)$, then $\mathbb{Z}[\twoP]=2\st\big([{\mathbb{Z}}[P^{-2}] \ra {\mathbb{Z}}[P^{-1}] \ra {\mathbb{Z}}[P^0]]\big)$, where ${\mathbb{Z}}[P^i]$ is the abelian sheaf generated by $P^i$ according to \cite[Expos\'e IV 11]{SGA4}.
\end{lemma}

\begin{proof}
An object of $\twoP(U)$ (with $U$ an object of $\ES$) is a collection $(V_{\bullet} \ra U, X, \varphi, \alpha)$ where $(X, \varphi, \alpha)$ is an effective 2-descent datum relative to the hypercover $V_{\bullet} \ra U$.  Then an object of $\mathbb{Z}[\twoP](U)$ is the formal sum $\sum_{i \in I}n_i[(V_{\bullet}^i \ra U,X_i,\varphi_i,\alpha_i)]$. The claim follows from the equality 
\[\sum_{i \in I}n_i[(V_{\bullet}^i \ra U,X_i,\varphi_i,\alpha_i)]=(V_{\bullet} \ra U,\sum_{i \in I}n_i[X_i],\sum_{i \in I}n_i[\varphi_i],\sum_{i \in I}n_i[\alpha_i]),\]
 where $V_{\bullet} \ra U$ is the refinement of the hypercovers $V_{\bullet}^i \ra U$.
\end{proof}

\begin{proof}[Proof of Theorem \ref{thm:introtor} for i=1]
The idea of the proof is to construct two morphisms 
\begin{align}
\nonumber	\Theta & \colon \TORS^1(\twoG) \longrightarrow \HH^1([\twoG]^{\flat\flat}),\\
\nonumber	\Psi & \colon \HH^1([\twoG]^{\flat\flat}) \longrightarrow \TORS^1(\twoG),
\end{align}
and to check that $\Theta \circ \Psi = \id = \Psi \circ \Theta $ and that $\Theta$ is an homomorphism of groups. We will just construct $\Theta$ and $\Psi$, since the remains of the proof are very similar to 
\cite[Thm 1.1 $proof\; i=1$]{beta}.

We fix the following notation: if $A$ is a complex of $\twocatD$, we set $\twoA=2\st^{\flat\flat}(A)$, and   if $f:A \ra B$ is a morphism in $\twocatD$, we denote by $F: \twoA \ra \twoB$ a representative of the equivalence class of additive 2-functors $2\st^{\flat\flat}(f)$. 

 Construction of $\Theta$: Let $\twoP$ be a $\twoG$-torsor and let $\mathbb{Z}[\twoP]$ be the Picard $\ES$-2-stack generated by $\twoP$. Consider the additive 2-functor 
$$H: \mathbb{Z}[\twoP] \longrightarrow \mathbb{Z}[\mathbf{0}]  $$ 
which associates to an object $\sum_i n_i [p_i] $ of $\mathbb{Z}[\twoP](U)$ the object $\sum_i n_i $ of $ \mathbb{Z}[\mathbf{0}](U)$, for $U$ an object of $\ES$. The homotopy kernel $\twoKer(H)$ of $H$ is the Picard $\ES$-2-stack whose objects are sums of the form
$[p]-[p']$, with $p,p'$ objects of $\twoP(U)$. Clearly $\mathbb{Z}[\twoP]$ is an extension of Picard $\ES$-2-stacks of $\mathbb{Z}[\mathbf{0}]  $ by $\twoKer(H)$. 
Consider now the additive 2-functor $L: \twoKer(H) \ra \twoG $ 
which associates to an object $[p]-[p']$ of $\twoKer(H)(U)$ the object $g$ of $ \twoG(U)$ such that $g . p=p'$. According to \cite[Def 7.3]{beta}, the push-down of the extension $\mathbb{Z}[\twoP]$ via $L: \twoKer(H) \ra \twoG $ is an extension $L_*\mathbb{Z}[\twoP]$ of $ \mathbb{Z}[\mathbf{0}] $ by $\twoG$. 
By \cite[Prop 6.7, Rem 6.6]{beta}, to this extension $L_*\mathbb{Z}[\twoP]$ of Picard $\ES$-2-stacks is associated the distinguished triangle $[\twoG]^{\flat\flat} \ra [L_*\mathbb{Z}[\twoP]]^{\flat\flat}  \ra \mathbf{E} \ra +$ in $\der$ which furnishes the long exact sequence
\begin{equation*}\label{eqn:(1)1}
 \xymatrix@1@C=0.6cm{\cdots \ar[r] & \HH^0([\twoG]^{\flat\flat})  \ar[r] & \HH^0([L_*\mathbb{Z}[\twoP]]^{\flat\flat}) \ar[r]& \HH^0(\mathbf{E}) \ar[r]^(0.40){\partial} & \HH^1([\twoG]^{\flat\flat}) \ar[r] & \cdots}
\end{equation*}
We set $\Theta(\twoP)= \partial (1),$ where the element $1$ of $ \HH^0(\mathbf{E})$ corresponds to the global neutral object $ e \in \Gamma(\mathbf{0})$ of the Picard $\ES$-2-stack $\mathbf{0}$.

 Construction of $\Psi$: Let $G$ be the complex $[\twoG]^{\flat\flat}$ of $\twocatD$ corresponding to the Picard $\ES$-2-stack $\twoG$. Choose a complex $I=[I^{-2} \rightarrow  I^{-1} \rightarrow I^0]$ of $\twocatD$ such that $I^{-2}$, $I^{-1},I^0$ are injective and such that there exists an injective morphism of complexes $s \colon G \ra I$. We complete $s$ into a distinguished triangle
$G \stackrel{s}{\rightarrow} I \stackrel{t}{\rightarrow} \MC(s)\rightarrow +$
in $\der$. Setting $K= \tau_{\geq -2} \MC(s)$, the above distinguished triangle furnishes an extension of Picard $\ES$-2-stacks $\twoG \stackrel{S}{\rightarrow}  \twoI \stackrel{T}{\rightarrow}  \twoK $
and the long exact sequence
\begin{equation*}
 \xymatrix@1@C=0.6cm{\cdots \ar[r] & \HH^0(G) \ar[r] & \HH^0(I) \ar[r]^{ t \circ }& \HH^0(K) \ar[r]^{\partial} & \HH^1(G) \ar[r] &0.}
\end{equation*}
Given an element $x$ of $\HH^1(G)$, choose an element $u$ of $\HH^0(K)$ such that $\partial (u) =x$. Remark that via the equivalence of categories $2\st^{\flat\flat}$ (\ref{intro:2st_flat_flat}), the element $u \in \HH^0(K)$ corresponds to a global section $U \in \Gamma(\twoK)$ of $\twoK$, i.e. to a morphism of $\ES$-2-stacks $U:\mathbf{0} \rightarrow \twoK $. 
Using the notion of pull-back (or fibered product) of $\ES$-2-stacks in 2-groupoids given in Definition \ref{def:fibered_product}, consider the pull-back $U^*\twoI$ of $\twoI$ via $U \colon \mathbf{0} \ra \twoK$. This pull-back $U^*\twoI$, which is an $\ES$-2-stack in 2-groupoids
not necessarily endowed with a Picard $\ES$-2-stack structure, is a $\twoG$-torsor: in fact, the action $\twoG \times U^*\twoI \rightarrow U^*\twoI$ of $\twoG$ on $U^*\twoI $ is given by $(g, i) \mapsto S(g).i$, where $g$ is an object of $\twoG$, $i$ is an object of $\twoI$ such that $T(i)=U(e)$, and $"."$ is the group law of the Picard $\ES$-2-stack $\twoI$.
We set $\Psi(x) =  U^*\twoI $ i.e. to be precise $\Psi(x)$ is the equivalence class of the $\twoG$-torsor $U^*\twoI$.
\end{proof}

\begin{proof}[Proof of Corollary \ref{corollary:Z[I]exttor}] Let $G=[\twoG]^{\flat\flat}$ and $P=[\twoP]^{\flat\flat}$. From Lemma \ref{lemma:generated_Pic_2_stack}, $[\mathbb{Z}[\twoP]]^{\flat\flat}={\mathbb{Z}}[P]$. By definition of ${\mathbb{Z}}[P]$, the functor $ G \rightarrow {\hhom}_{{\mathbb{Z}}}({\mathbb{Z}}[P],G)$
 is isomorphic to the functor $ G \rightarrow G(P)=\HH^0(P,G_P),$
with $G_P=[\twoG_\twoP]^{\flat\flat}$. Taking the derived functors and using  
the homological interpretation of torsors (Thm \ref{thm:introtor}) and of extensions of Picard $\ES$-2-stacks \cite[Thm 1.1]{beta}, we can conclude.
\end{proof}

%--------------------------------------------------------------------------------------------
\section{Description of extensions of Picard 2-stacks in terms of torsors}

Let $\twoP$ and $\twoG$ be two Picard $\ES$-2-stacks. If $K$ is a subset of a finite set $E$, $p_K: \twoP^E  \ra \twoP^K$ is the projection to the factors belonging to $K$, and $ \otimes_K: \twoP^E  \ra \twoP^{E-K+1} $ is the group law $\otimes: \twoP \times \twoP \ra \twoP$ on the factors belonging to $K$. If $\iota$ is a permutation of the set $E$, $\Perm(\iota): \twoP^E  \ra \twoP^{\iota(E)}$ is the permutation of the factors  according to $\iota$. Moreover let $\mathsf{s}: \twoP \times \twoP \ra \twoP \times \twoP$ be the morphism of $\ES$-2-stacks that exchanges the factors and let $D: \twoP \ra \twoP \times \twoP$ be the diagonal morphism of $\ES$-2-stacks.

\begin{proposition}\label{thm:ext-tor}
To have an extension $\twoE=(\twoE,I,J)$ of $\twoP$ by $\twoG$ is equivalent to have 
\begin{enumerate}
\item a $\twoG_\twoP$-torsor $\twoE$ over $\twoP$;
\item a morphism of $\twoG_{\twoP^2}$-torsors $M : p_1^* \; \twoE \wedge p_2^* \; \twoE \lra \otimes^* \; \twoE.$ Here $\otimes^* \; \twoE$ is the pull-back of $\twoE$ via the group law $\otimes: \twoP \times \twoP \ra \twoP $ of $\twoP$ and for $i=1,2$, $p_i^* \; \twoE$ is the pull-back of $\twoE$ via the $i$-th projection $p_i: \twoP \times \twoP \ra \twoP $ (these pull-backs are pull-backs of $\ES$-2-stacks in 2-groupoids according to Definition \ref{def:fibered_product});
\item a 2-morphism of $\twoG_{\twoP^3}$-torsors $\alpha: M \circ (M \wedge \id) \Ra M \circ (\id \wedge M);$
\item a 3-morphism of $\twoG_{\twoP^4}$-torsors 
$\gotika:  p^*_{234} \; \alpha \circ \otimes^*_{23} \; \alpha \circ p^*_{123}\; \alpha \Rrightarrow   \otimes^*_{34}\; \alpha \circ \otimes^*_{12} \; \alpha $
whose pull-back over $\twoP^5$ satisfies the equality 
\begin{equation}\label{ext-tor:1}
\otimes_{45}^* \gotika \circ \otimes_{23}^* \;\gotika \circ p_{2345}^* \;\gotika = \otimes_{12}^* \;\gotika \circ p_{1234}^* \;\gotika  \circ \otimes_{34}^* \;\gotika.
\end{equation}
\item a 2-morphism of $\twoG_{\twoP^2}$-torsors $\chi: M \Ra M \circ \mathsf{s} ;$ 
\item a 3-morphism of $\twoG_{\twoP^2}$-torsors $\gotiks: \chi \circ \chi \Rrightarrow \id$ satisfying the equation of 2-arrows obtained from (\ref{CoherenceZeta}) by replacing $\comm$ with $\chi$ and $\zeta$ with $\gotiks$;
\item two 3-morphisms of $\twoG_{\twoP^3}$-torsors 
\begin{align}
\nonumber \gotikc_1&: Perm(132)^* \alpha \circ \otimes^*_{23} \; \chi \circ \alpha \Rrightarrow  p^*_{13} \; \chi \circ Perm(12)^*  \alpha \circ p^*_{12} \; \chi\\
\nonumber \gotikc_2&: Perm(123)^* \alpha^{-1} \circ \otimes^*_{12}\mathsf{s}^* \; \chi^{-1} \circ \alpha^{-1} \Rrightarrow  p^*_{13}\mathsf{s}^* \; \chi^{-1} \circ Perm(23)^*  \alpha^{-1} \circ p^*_{23}\mathsf{s}^* \; \chi^{-1}
\end{align}
 which satisfy the compatibility conditions obtained from  (\ref{diagram:hexagone3}), (\ref{diagram:hexagone3bis}), (\ref{diagram:hexagone1_vs_hexagone2}),
 (\ref{diagram:hexagone1_vs_hexagone2_bis}), (\ref{diagram:Z_system}) by replacing $\zeta$ with $\gotiks$, $\gotikh_i$ with $\gotikc_i$ for $i=1,2$, and whose pull-backs over $\twoP^4$ satisfy
 \begin{align}
\label{ext-tor:2} Perm(12)^* \gotika \circ p^*_{134} \; \gotikc_1 \circ \otimes_{34}^* \; \gotikc_1 \circ \gotika &=
p^*_{123} \; \gotikc_1 \circ Perm(132)^* \gotika  \circ  Perm(1432)^* \gotika \circ \otimes_{23}^* \; \gotikc_1.\\
\label{ext-tor:4} Perm(34)^* \gotika \circ p^*_{124} \; \gotikc_2 \circ \otimes_{12}^* \; \gotikc_2 \circ \gotika &=
p^*_{234} \; \gotikc_2 \circ Perm(234)^* \gotika  \circ  Perm(1234)^* \gotika \circ \otimes_{23}^* \; \gotikc_2.
\end{align}
\item a 3-morphism of $\twoG_{\twoP}$-torsors $\gotikp: D^* \; \chi \Rrightarrow \id$
satisfying $\gotikp*\gotikp=\gotiks$ and the compatibility condition obtained from  (\ref{diagram:additivity_of_braiding}) by replacing $\pi$ with $\gotika$, $\zeta$ with $\gotiks$, 
 $\gotikh_i$ with $\gotikc_i$ for $i=1,2$, $\eta$ with $\gotikp$.
\end{enumerate}
\end{proposition}

\begin{proof} (I) Let $\twoE=(\twoE,I,J)$ be an extension of $\twoP$ by $\twoG$. Via the additive 2-functor $I: \twoG \ra \twoE$, the Picard $\ES$-2-stack $\twoG$ acts on the left side and on the right side of $\twoE$ inducing an action on the homotopy fiber $\twoE_p$ for any object $p \in \twoP$. Since the additive 2-functor $J: \twoE \ra \twoP$ induces a surjection $\pi_0(J): \pi_0(\twoE) \ra \pi_0(\twoP)$ on the $\pi_0$, $\twoE_p$ and $\twoE_{-p}$ are non empty. Choose an object $y$ in $\twoE_{-p}$. Then $y\otimes -:\twoE_p \ra \twoKer(J)(U)$ is a biequivalence. Hence,  $\twoE$ is a $\twoG_\twoP$-torsor over $\twoP$ (1). The group law $\otimes : \twoE \times \twoE \ra \twoE$ of $\twoE$ furnishes a morphism of $\ES$-2-stacks $p_1^* \; \twoE \times p_2^* \; \twoE \ra \otimes^* \; \twoE$ over $\twoP \times \twoP$. The existence for any $g \in \twoG$ and $a,b \in \twoE$ of the associativity constraint $\ass_{(a,g,b)}: (ag) b \rightarrow a(g b)$ implies that this morphism of $\ES$-2-stacks $p_1^* \; \twoE \times p_2^* \; \twoE \ra \otimes^* \; \twoE$ factorizes via the contracted product  $M : p_1^* \; \twoE \wedge p_2^* \; \twoE \ra \otimes^* \; \twoE$. The existence for any $g \in \twoG$ and $a,b \in \twoE$ of the associativity constraints $\ass_{(g,a,b)}: (ga) b \rightarrow g(a b)$ and  $\ass_{(a,b,g)}: (ab) g \rightarrow a(b g)$ implies that the morphism of $\ES$-2-stacks  $M : p_1^* \; \twoE \wedge p_2^* \; \twoE \ra \otimes^* \; \twoE$ is in fact a morphism of $\twoG_{\twoP^2}$-torsors once we consider on $p_1^* \; \twoE \wedge p_2^* \; \twoE$ the following structure of $\twoG_{\twoP^2}$-torsors: the left (resp. right) action of $\twoG_{\twoP^2}$ on $p_1^* \; \twoE \wedge p_2^* \; \twoE$ comes from the left (resp. right) action of $\twoG_{\twoP^2}$ on $p_1^* \; \twoE$ (resp. $ p_2^* \; \twoE$) (2). Now the associativity $\ass:\otimes \circ (\otimes \times \id_\twoE) \Ra \otimes \circ (\id_\twoE \times \otimes)$ implies the 2-morphism of $\twoG_{\twoP^3}$-torsors $\alpha: M \circ (M \wedge \id) \Ra M \circ (\id \wedge M)$ over $\twoP \times \twoP \times \twoP$ (3). The modification $\pi$ (\ref{ObstructionCoherenceAss}), satisfying the coherence axiom of Stasheff's polytope (\ref{diagram:stasheff}), is equivalent to the 3-morphism of $\twoG_{\twoP^4}$-torsors $\gotika$ satisfying the equality (\ref{ext-tor:1}) (4). The braiding $\comm : \otimes \circ \mathsf{s} \Ra \otimes$ furnishes the 2-morphism of $\twoG_{\twoP^2}$-torsors $\chi: M \Ra M \circ \mathsf{s} $ over $\twoP \times \twoP$ (5). The modification $\zeta$ (\ref{ObstructionCoherenceComm}),  satisfying the coherence condition (\ref{CoherenceZeta}), is equivalent to the 3-morphism of $\twoG_{\twoP^2}$-torsors $\gotiks$ with its coherence condition (6). The modifications $\gotikh_1$ and $\gotikh_2$ (\ref{ObstructionCompatibilityAssComm}), satisfying the compatibility conditions (\ref{diagram:hexagone1}), (\ref{diagram:hexagone2}), (\ref{diagram:hexagone3}), (\ref{diagram:hexagone3bis}),  (\ref{diagram:hexagone1_vs_hexagone2}), (\ref{diagram:hexagone1_vs_hexagone2_bis}), (\ref{diagram:Z_system}), are equivalent to the 3-morphisms of $\twoG_{\twoP^3}$-torsors $\gotikc_1$ and $\gotikc_2$ with their compatibility conditions (7) (remark that condition (\ref{diagram:hexagone1}) corresponds to (\ref{ext-tor:2}) and condition (\ref{diagram:hexagone2}) corresponds to (\ref{ext-tor:4})).  Finally, the modification $\eta$ (\ref{Strictness}), satisfying $\eta \ast \eta=\zeta$ and the compatibility condition (\ref{diagram:additivity_of_braiding}), is equivalent to the 3-morphism of $\twoG_{\twoP}$-torsors $\gotikp$ with its compatibility conditions (8).

(II) Now suppose we have the data $(\twoE,M,\alpha,\gotika,\chi,\gotiks,\gotikc_1,\gotikc_2)$ given in (1)-(8). The morphism of $\twoG_{\twoP^2}$-torsors $M : p_1^* \; \twoE \wedge p_2^* \; \twoE \lra \otimes^* \; \twoE$ over $\twoP \times \twoP$ defines a group law $\otimes: \twoE \times \twoE \ra \twoE$ on the $\ES$-2-stack of 2-groupoids $\twoE$.
The data $\alpha$ and $\chi$ furnish the associativity $\ass:\otimes \circ (\otimes \times \id_\twoE) \Ra \otimes \circ (\id_\twoE \times \otimes)$ and the braiding $\comm : \otimes \circ \mathsf{s} \Ra \otimes$ which express respectively the associativity and the commutativity constraints of the group law $\otimes$ of $\twoE$. As already observed in (I), the data $\gotika,\gotiks,\gotikc_1,\gotikc_2,\gotikp$ give respectively the modifications of $\ES$-2-stacks $\pi$ (\ref{ObstructionCoherenceAss}), $\zeta$ (\ref{ObstructionCoherenceComm}), $\gotikh_1$, $\gotikh_2$  (\ref{ObstructionCompatibilityAssComm}), $\eta$ (\ref{Strictness}), with their coherence and compatibility conditions.
Since any morphism of $\twoG$-torsors is an equivalence of $\ES$-2-stacks, the morphism of $\twoG_{\twoP^2}$-torsors $M: p_1^* \; \twoE \wedge p_2^* \; \twoE \lra \otimes^* \; \twoE$ implies that for any object $a \in \twoE$, the left multiplication by $a$, $a \otimes -:\twoE \ra \twoE,$ is an equivalence of $\ES$-2-stacks. By \cite{MR2342833} this property of the left multiplication to be an equivalence implies that $\twoE$ admits a global neutral object $e$ and that any object of $\twoE$ admits an inverse. \\
If $J: \twoE \ra \twoP$ denotes the morphism of $\ES$-2-stacks underlying the structure of $\twoG_\twoP$-torsor over $\twoP$, $J$ must be a surjection on the equivalence classes of objects, i.e. $\pi_0(J): \pi_0(\twoE) \rightarrow \pi_0(\twoP)$ is surjective. Moreover the compatibility of $J$ with the morphism of $\twoG_{\twoP^2}$-torsors $M : p_1^* \; \twoE \wedge p_2^* \; \twoE \lra \otimes^* \; \twoE$ over $\twoP \times \twoP$ implies that $J$ is an additive 2-functor. There is a global equivalence of $\twoG$-torsors between $\twoG$ and the pull-back $ \mathbf{0}^*\twoE$ of $\twoE$ via $\mathbf{0}: \mathbf{0} \rightarrow \twoP$ which is given by sending the global neutral object $0_{\twoG}$ of $\twoG$ to the global neutral object $(0_{\twoP},\ell,0_{\twoE})$ of $\mathbf{0}^*\twoE$, where $\ell$ is the 1-arrow $0_{\twoP} \ra J(0_{\twoE}) $ in $\twoP$. Let $I$ be the composite $\twoG \cong \mathbf{0}^*\twoE = \twoKer(J) \ra \twoE$. Clearly $I$ is an additive 2-functor. We can conclude that $(\twoE, I,J)$ is an extension of $\twoP$ by $\twoG$.
\end{proof}

As a consequence of this Proposition we get Theorem \ref{thm:introexttor}.

%--------------------------------------------------------------------------------------------
\section{Right Resolution of $\Ext(\twoP,\twoG)$}

A cochain complex of Picard $\ES$-2-stacks $ 
\hdots \rightarrow \twoL^{i-1} \stackrel{D^{i-1}}{\rightarrow}  \twoL^{i} \stackrel{D^{i}}{\rightarrow}  \twoL^{i+1} \stackrel{D^{i+1}}{\rightarrow}  \hdots,$
consists of Picard $\ES$-2-stacks $\twoL^i$ for $i \in \mathbb{Z}$,
 additive 2-functors $D^i: \twoL^i \ra \twoL^{i+1}$,
 morphisms of additive 2-functors $\partial^i: D^{i+1} \circ D^i \Ra 0$, and 
 modifications of morphisms of additive 2-functors 
\begin{equation}\label{diagram:cochain_complex_1}
\begin{tabular}{c}
\xymatrix@C=1.25cm{ (D^{i+2}D^{i+1})D^i \ar@2[r]^{\ass} \ar@2[d]_{\partial^{i+1} * D^i} &D^{i+2}(D^{i+1}D^i) \ar@2[r]^(0.6){D^{i+2} * \partial^i} \ar@{}[d]|{\kesir{\rotatebox{90}{$\Lla$}}{\Delta_{(i+2,i+1,i)}}}&  D^{i+2}0 \ar@2[d]\\
0 D^i \ar@2[rr] && 0 }
\end{tabular}
\end{equation}
which satisfy the following equation of modifications: the pasting of the modifications
\begin{equation*}
\begin{tabular}{c}
\xymatrix{ &(D^{i+3}(D^{i+2}D^{i+1}))D^{i} \ar@2[r] \ar@2[d] \ar@{}[dr]|{\kesir{\mbox{\rotatebox{90}{$\Lla$}}}{\circ_{(\ass,\partial^{i+1})}}} & D^{i+3}((D^{i+2}D^{i+1})D^{i}) \ar@2[d] \ar@2[r] & D^{i+3}(D^{i+2}(D^{i+1}D^{i})) \ar@2[d] \\
((D^{i+3}D^{i+2})D^{i+1})D^{i} \ar@{}[r]|(0.55){\kesir{\mbox{\rotatebox{90}{$\Lla$}}}{\Delta_{(i+3,i+2,i+1)}*D^{i}}}\ar@2[ur] \ar@2[d]&(D^{i+3}0)D^{i} \ar@2[r] \ar@2[d] & D^{i+3}(0D^{i}) \ar@2[dr] \ar@{}[dl]|{\simeq} \ar@{}[r]|{\kesir{\mbox{\rotatebox{90}{$\Lla$}}}{D^{i+3}*\Delta_{(i+2,i+1,i)}}}& D^{i+3}(D^{i+2}0) \ar@2[d]\\
(0D^{i+1})D^{i} \ar@2[r]&0D^{i} \ar@2[r] &0 \ar@2[r] &D^{i+3}0}
\end{tabular}
\end{equation*} 
is equal to the pasting of the modifications in the diagram below
\begin{equation*}
\begin{tabular}{c}
\xymatrix{&((D^{i+3}D^{i+2})D^{i+1})D^{i} \ar@2[dr] \ar@2[dl] \ar@2[r] \ar@{}[dd]|{\kesir{\mbox{{$\Lla$}}}{\circ_{(\ass,\partial^{i+2})}}} & (D^{i+3}(D^{i+2}D^{i+1}))D^{i} \ar@2[r] \ar@{}[dr]|{\kesir{\mbox{\rotatebox{90}{$\Lla$}}}{\pi_{(i+3,i+2,i+1,i)}}} &D^{i+3}((D^{i+2}D^{i+1})D^{i}) \ar@2[d]\\
(0D^{i+1})D^{i} \ar@2[d] \ar@2[dr] && (D^{i+3}D^{i+2})(D^{i+1}D^{i}) \ar@{}[dr]|{\kesir{\mbox{{$\Lla$}}}{\circ_{(\ass,\partial^{i})}}} \ar@2[dl] \ar@2[d] \ar@2[r] & D^{i+3}(D^{i+2}(D^{i+1}D^{i})) \ar@2[d]\\
0D^{i} \ar@2[dr] \ar@{}[r]|{\simeq} & 0(D^{i+1}D^{i}) \ar@2[d] \ar@{}[r]|{\kesir{\mbox{{$\Lla$}}}{\circ_{(\partial^{i+2},\partial^{i})}}} & (D^{i+3}D^{i+2})0 \ar@2[dl] \ar@2[r] \ar@{}[dr]|{\simeq} & D^{i+3}(D^{i+2}0)\ar@2[d]\\
& 0 \ar@2[rr] && D^{i+3}0.}
\end{tabular}
\end{equation*}

Let $\twoG$ be a Picard $\ES$-2-stack and let
$ \twoL^{.}:  0 \rightarrow \twoT   \stackrel{D^{\twoT}}{\rightarrow} \twoS \stackrel{D^{\twoS}}{\rightarrow} \twoR
\stackrel{D^{\twoR}}{\rightarrow} \twoQ \stackrel{D^{\twoQ}}{\rightarrow}  \twoP \rightarrow 0$
be a complex of Picard $\ES$-2-stacks with $\twoP$, $\twoQ$, $\twoR$, $\twoS$, and $\twoT$ in degrees 0,-1, -2, -3 and -4, respectively. To the complex $\twoL^{.}$ and to $\twoG$, we associate a 3-category $\Psi_{\twoL^{.}}(\twoG)$ which we can see as the 3-category of extensions of complexes of Picard $\ES$-2-stacks of $\twoL^{.}$ by $\twoG$, considering $\twoG$ as a complex concentrated in degree 0. This 3-category is a generalization to Picard $\ES$-2-stacks of the one introduced by Grothendieck in \cite{SGA7} for abelian sheaves.

\begin{definition}\label{psi}
Let $\Psi_{\twoL^{.}}(\twoG)$ be the 3-category 
\begin{itemize}
	\item whose objects are pairs $(\twoE,T)$ where 
 $\twoE=(I:\twoG \ra \twoE,\twoE,J:\twoE \ra \twoP,\varepsilon)$ is an extension of $\twoP$ by $\twoG$ and
 $T=(T,\mu,\Upsilon)$ is a trivialization of the extension $(D^{\twoQ})^*\twoE$ of $\twoQ$ by $\twoG$  obtained as pull-back of $\twoE$ via $D^{\twoQ}: \twoQ \rightarrow \twoP$. We require that the trivialization $T$ is compatible with the complex $\twoL^{.}$, i.e. it satisfies the following conditions:
\begin{enumerate}
	\item the trivialization $(D^{\twoR})^*T$ of $(D^{\twoR})^*  (D^{\twoQ})^*\twoE$ is the trivialization arising from the equivalence of transitivity $ (D^\twoR)^* (D^\twoQ)^* \twoE \cong  (D^\twoQ \circ D^\twoR)^*\twoE$ and from the morphism of additive 2-functors $\partial^\twoR: D^\twoQ \circ D^\twoR \Rightarrow 0$;
	\item the morphism of additive 2-functor $(D^\twoS)^* (D^\twoR)^* T \Rightarrow 0$ arises from the 2-isomorphism of transitivity $ (D^\twoS)^* (D^\twoR)^* T \cong  (D^\twoR \circ D^\twoS)^*T$ and from the morphism of additive 2-functors $\partial^\twoS: D^\twoR \circ D^\twoS \Rightarrow 0$;
	\item the morphism of additive 2-functor $(D^\twoT)^* (D^\twoS)^* (D^\twoR)^* T \Rightarrow 0$ is compatible with the modification of morphisms of additive 2-functors $\Delta_{(\twoT,\twoS,\twoR)}$ (\ref{diagram:cochain_complex_1}) underlying the complex $ \twoL^{.}$.
\end{enumerate}
	\item whose 1-arrows are given by triplets $(F,\sigma, \Sigma): (\twoE,T) \ra (\twoE',T')$ where $F:\twoE \ra \twoE'$ is a morphism of extensions of Picard $\ES$-2-stacks (inducing the identity on $\twoG$ and $\twoP$), $\sigma: F \circ T \Rightarrow T'$ is a morphism of additive 2-functors, and $\Sigma$ is a modification of morphisms of additive 2-functors 
\begin{equation*}
\begin{tabular}{c}
\xymatrix{(J'F)T \ar@2[r] \ar@2[d]_{}& J'(FT) \ar@2[r]^(0.6){} \ar@{}[d]|{\kesir{\rotatebox{90}{$\Lla$}}{\Sigma}}  & J'T' \ar@2[d]\\
JT \ar@2[rr]_{} &&\id_\twoP.}
\end{tabular}
\end{equation*}	
	\item whose 2-arrows are pairs $(\alpha, \Omega): (F,\sigma, \Sigma) \Rightarrow (F',\sigma', \Sigma')$ where $\alpha: F \Ra F'$ is a 2-morphism of extensions of Picard $\ES$-2-stacks, $\Omega: \sigma' \circ \alpha \Rra \sigma$ is a modification of morphisms of additive 2-functors which is compatible with the modifications $\Sigma$ and $\Sigma'$. 
	\item whose 3-arrows $\Delta: (\alpha, \Omega) \Rra (\alpha', \Omega') $  are 3-morphisms of extensions of Picard $\ES$-2-stacks $\Delta: \alpha \Rra \alpha'$ which are compatible with the modifications $\Omega$ and $\Omega'$. 
\end{itemize}
\end{definition}

For the notion of $i$-morphism of extensions of Picard $\ES$-2-stacks ($i=1,2,3$) we refer to \cite[\S 5]{beta}.

 Let $\Psi_{\twoL^{.}}^1(\twoG)$ be the abelian group of equivalence classes of objects of $\Psi_{\twoL^{.}}(\twoG)$ (its abelian group law is furnished by the sum of extensions of Picard $\ES$-2-stack \cite[Def 7.4]{beta}). For $i=0,-1,-2$ let $\Psi_{\twoL^{.}}^i(\twoG)$ be the abelian homotopy group $\pi_{-i}(\Hom_{\Psi_{\twoL^{.}}(\twoG)}((\twoE,T),(\twoE,T)))$ 
 of the hom-2-groupoid $\Hom_{\Psi_{\twoL^{.}}(\twoG)}((\twoE,T),(\twoE,T))$ of morphisms of an object $(\twoE,T)$ of $\Psi_{\twoL^{.}}(\twoG)$ to itself 
 (since $\Hom_{\Psi_{\twoL^{.}}(\twoG)}((\twoE,T),(\twoE,T))$ is equivalent to the homotopy kernel $\twoKer\big(D^\twoQ:{\Hom}(\twoP,\twoG)  \rightarrow {\hhom}(\twoQ,\twoG)\big)$, it is endowed with a Picard $\ES$-2-stack structure and its homotopy groups are abelian groups).
 Generalizing \cite[Thm 8.2]{MR2995663} to Picard $\ES$-2-stacks, we have the following homological description of $\Psi_{\twoL^{.}}^i(\twoG)$:
\begin{equation}\label{eq:psi-ext}
\Psi_{\twoL^.}^i(\twoG) \cong {\eext}^i\big({\Tot}([\twoL^.]),[\twoG]\big)= {\hhom}_{\der}\big({\Tot}([\twoL^.]),[\twoG][i]\big) \qquad i=-2,-1,0,1.
\end{equation}
In general, additive 2-functors do not correspond to morphisms of complexes. To simplify the computation of the isomorphisms (\ref{eq:psi-ext}), we assume that the additive 2-functors of the complex $ \twoL^.$ arise from morphisms of length 3 complexes (we have proceeded in this way also in \cite{MR2995663}). This is not restrictive since 
if $\twoP$ is a Picard $\ES$-2-stack, Lemma \ref{lemma:generated_Pic_2_stack} furnishes an explicit description of the length 3 complex associated to ${\mathbb{Z}}[\twoP]$, and this allows us to define degree-wise the differentials $D_i$  underlying the complex $\twoL.(\twoP)$ of Corollary \ref{corollary:resolutionP}, i.e. the differentials $D_i$ (\ref{complex_L(P)}) are in fact morphisms of complexes. 

Let
$ \twoL^{.}:  0 \rightarrow \twoT   \stackrel{D^{\twoT}}{\rightarrow} \twoS \stackrel{D^{\twoS}}{\rightarrow} \twoR
\stackrel{D^{\twoR}}{\rightarrow} \twoQ \stackrel{D^{\twoQ}}{\rightarrow}  \twoP \rightarrow 0$ and 
$ \twoL^{'.}:  0 \rightarrow \twoT'   \stackrel{D^{\twoT'}}{\rightarrow} \twoS' \stackrel{D^{\twoS'}}{\rightarrow} \twoR'
\stackrel{D^{\twoR'}}{\rightarrow} \twoQ' \stackrel{D^{\twoQ'}}{\rightarrow}  \twoP' \rightarrow 0$
be two complexes of Picard $\ES$-2-stacks with $\twoP, \twoP'$ in degree 0, $\twoQ, \twoQ'$ in degree -1,
$\twoR, \twoR'$ in degree -2, $\twoS, \twoS'$ in degree -3, and 
 $\twoT, \twoT'$ in degree -4. For any Picard $\ES$-2-stack $\twoG$,
a morphism $F^{.}=(F^{-4},F^{-3},F^{-2},F^{-1},F^0): \twoL^{'.} \rightarrow \twoL^{.} $ of complexes of Picard $\ES$-2-stacks induces a canonical 3-functor 
$(F^{.})^*: \Psi_{\twoL^{.}} (\twoG)\rightarrow \Psi_{ \twoL^{'.}}(\twoG):$
if $(\twoE,T)$ is object of $\Psi_{\twoL^{.}}(\twoG)$, we set $(F^.)^*(\twoE,T)=((F^0)^*\twoE,(F^{-1})^*T)$ with
  $(F^0)^*\twoE$ the extension of $\twoP'$ by $\twoG$
obtained as pull-back of $\twoE$ via $F^0: \twoP' \rightarrow \twoP$, and 
    $(F^{-1})^*T $ the trivialization of $(D^{\twoQ '})^*(F^0)^*\twoE$ induced by the trivialization $T$ of $(D^{\twoQ})^*\twoE$.

 \begin{lemma}\label{lemma:PsiHH}  
The 3-functor $(F^{.})^* :\Psi_{\twoL^{.}} (\twoG)\rightarrow \Psi_{ \twoL^{'.}}(\twoG)$ is a tri-equivalence if and only if 
$\HH^i(\Tot(F^{.})) : \HH^i(\Tot([\twoL^{'.}]^{\flat\flat})) \rightarrow \HH^i(\Tot([\twoL^{.}]^{\flat\flat}))$ is an isomorphism for any $i$. 
\end{lemma}

\begin{proof} For $i=-2,-1,0,1$ we have the following commutative diagram
\[
\begin{array}{ccc}
 \Psi_{\twoL^.}^i(\twoG)&\rightarrow &{\eext}^i\big({\Tot}([\twoL^.]),[\twoG]\big)\\
 \downarrow &  & \downarrow \\
 \Psi_{\twoL'^.}^i(\twoG)& \rightarrow & {\eext}^i\big({\Tot}([\twoL'^.]),[\twoG]\big),
\end{array}
\]
where the vertical arrow on the left side is induced by the 3-functor $(F^.)^*: \Psi_{\twoL^.}(\twoG) \rightarrow \Psi_{\twoL'^.}(\twoG)$, 
the vertical arrow on the right side is induced by the morphism of complexes  $F^.: \twoL'^. \rightarrow \twoL^.$, and the horizontal arrows are the isomorphisms~(\ref{eq:psi-ext}).
The 3-functor $(F^.)^*: \Psi_{\twoL^.}(\twoG) \rightarrow \Psi_{\twoL'^.}(\twoG)$ is a tri-equivalence if and only if the vertical arrow on the left side
 is an isomorphism for $i=-2,-1,0,1$. Hence 
we are reduced to prove that the vertical arrow on the right side
 is an isomorphism for $i=-2,-1,0,1$ if and only if 
 ${\HH}^i\big({\Tot}(F^.)\big):{\HH}^i\big({\Tot}([\twoL'^.])\big) \rightarrow {\HH}^i\big({\Tot}([\twoL^.])\big)$ 
are isomorphisms for each $i$. This last assertion is clearly true.
\end{proof}

Now we switch from cohomological to homological notation.
To any Picard $\ES$-2-stack $\twoP$, we associate 
 the complex $\twoL.(\twoP)$ of Picard $\ES$-2-stacks which is defined in Corollary \ref{corollary:resolutionP}. Let $\twoG$ be a Picard $\ES$-2-stack. We have the following geometrical description of the 3-category $\Psi_{\twoL.(\twoP)}(\twoG)$:

\begin{proposition}\label{prop:ext-geom}
The 3-category ${\Ext}(\twoP,\twoG)$ of extensions of $\twoP$ by $\twoG$ is tri-equivalent to the 3-category $\Psi_{\twoL.(\twoP)}(\twoG)$.
\end{proposition}

\begin{proof} 
By Corollary \ref{corollary:Z[I]exttor}, an object $(\twoE,T)$ of $\Psi_{\twoL.(\twoP)}(\twoG)$ consists of a $\twoG_\twoP$-torsor $\twoE$ and a trivialization $T$ of the $\twoG_{\twoP^2}$-torsor $D_2^*\twoE$ obtained as pull-back of $\twoE$ via $D_2$. This trivialization can be interpreted as a morphism of $\twoG_{\twoP^2}$-torsors $M: p_1^* \; \twoE \wedge p_2^* \; \twoE \ra \otimes^* \; \twoE$, where $p_i: \twoP \times \twoP  \ra \twoP$ is the $i$-th projection of $\twoP \times \twoP $ on $ \twoP$  and $ \otimes: \twoP \times \twoP  \ra \twoP $ is the group law of $ \twoP$.\\
Concerning the compatibility between the trivialization $T$ and the complex $\twoL.(\twoP)$, we have:
\begin{enumerate}
	\item through the two torsors over $\twoP^3$ and $\twoP^2$, the compatibility of $T$ with $\partial_1: D_2 \circ D_3 \Ra 0 $ imposes on the data $\twoE$ and $M$ the 2-morphism of $\twoG_{\twoP^3}$-torsors $\alpha$ described in Proposition \ref{thm:ext-tor} (3) and the 2-morphism of $\twoG_{\twoP^2}$-torsors $\chi$ described in Proposition \ref{thm:ext-tor} (5);
	\item through the five torsors over $\twoP^4,\twoP^3,\twoP^3,\twoP^2$ and $\twoP$, the compatibility between $D_4^*  D_3^* T \Ra 0 $ and $\partial_2: D_3 \circ D_4 \Ra 0 $ imposes on the data $\alpha$ and $\chi$ the 3-morphism of $\twoG_{\twoP^4}$-torsors $\gotika$, the two 3-morphisms of $\twoG_{\twoP^3}$-torsors $\gotikc_1$ and $\gotikc_2$ and the 3-morphism of $\twoG_{\twoP^2}$-torsors $\gotiks$ and the 3-morphism of $\twoG_{\twoP}$-torsors $\gotikp$,
		which are described respectively in Proposition \ref{thm:ext-tor} (4), (7), (6) and (8);
   \item through the ten  torsors over $\twoP^5, \twoP^4, \twoP^4, \twoP^4, \twoP^3, \twoP^3, \twoP^3, \twoP^2, \twoP$, and $\twoP^2$,  the compatibility between $D_5^*D_4^* D_3^* T \Ra 0 $ and  $\Delta_{(D_3, D_4,D_5)} $ imposes on the datum $\gotika$ the equality (\ref{ext-tor:1}), on the data $\gotikc_1, \gotikc_2$ the equalities (\ref{ext-tor:2}), (\ref{ext-tor:4}) and the compatibility condition obtained from (\ref{diagram:Z_system}) by replacing $\zeta$ with $\gotiks$, $\gotikh_i$ with $\gotikc_i$ (for $i=1,2$), on the datum $\gotiks$ the equation of 2-arrows obtained from (\ref{CoherenceZeta}) by replacing $\comm$ with $\chi$ and $\zeta$ with $\gotiks$, and finally on the datum $\gotikp$ the equality $\gotikp*\gotikp=\gotiks$ and the compatibility condition obtained from  (\ref{diagram:additivity_of_braiding}) by replacing $\pi$ with $\gotika$, $\zeta$ with $\gotiks$, $\gotikh_i$ with $\gotikc_i$ (for $i=1,2$), $\eta$ with $\gotikp$.
\end{enumerate}
Hence by Proposition \ref{thm:ext-tor} the object $(\twoE,M,\alpha,\gotika,\chi,\gotiks,\gotikc_1,\gotikc_2)$ of $\Psi_{\twoL.(\twoP)}(\twoG)$ is an extension of $\twoP$ by $\twoG$. The remaining detail are left to the reader. 
\end{proof}

\begin{proof}[Proof of Corollary \ref{corollary:resolutionP}]
Consider the morphism of complexes $\epsilon. : \twoL.(\twoP) \ra \twoP$ defined by the additive 2-functor $\epsilon : {\mathbb{Z}}[\twoP] \ra \twoP, \epsilon([p])=p$ for any $p \in \twoP$ (here we consider $\twoP$ as a complex concentrated in degree 0). Since by definition $\Psi_{\twoP} (\twoG)$ is tri-equivalent to  ${\Ext}(\twoP,\twoG)$, Proposition \ref{prop:ext-geom} implies that the 3-functor
$(\epsilon.)^*: \Psi_{\twoP} (\twoG) \rightarrow \Psi_{\twoL.(\twoP)} (\twoG)$ is a tri-equivalence. Hence by Lemma \ref{lemma:PsiHH}, $\HH_i(\Tot(\epsilon.)):\HH_i(\Tot([\twoL.(\twoP)]^{\flat\flat})) \rightarrow \HH_i(\Tot([\twoP]^{\flat\flat}))$ is an isomorphism for any $i$.
\end{proof}

Before to prove Corollary \ref{corollary:resolution}, let's first state the exactness in $2\PICARD$. A 2-functor $F:\twoA \ra \twoB$ is 
\begin{itemize}
\item \emph{essentially surjective} if for any object $x$ of $\twoB$, there exists an object $a$ of $\twoA$ so that $F(a)$ is equivalent to $x$;
\item \emph{full} if for any two objects $a,b$ of $\twoA$, the functor $F_{(a,b)}:H_{\twoA}(a,b) \ra H_{\twoB}(Fa,Fb)$ is essentially surjective and full.
\end{itemize}
Thus we say that a cochain complex of Picard Picard $\ES$-2-stacks $\hdots \rightarrow \twoL^{i-1} \stackrel{D^{i-1}}{\rightarrow}  \twoL^{i} \stackrel{D^{i}}{\rightarrow}  \twoL^{i+1} \stackrel{D^{i+1}}{\rightarrow}  \hdots$ is exact at $\twoL^i$ if the additive 2-functor $\tilde{D}^{i-1}: \twoL^{i-1} \ra \twoKer(D^i)$ is full and essentially surjective. We notice that we will work in $2\FPICARD$, so the notion of essentially surjective and full will be more strict. Upon defining correct notions of full and essentially surjective, one can generalize this definition to definition of exactness in $3{\mathcal{P}\textsc{icard}^{\flat \flat \flat}(\ES)}$.

\begin{proof}[Sketch of the proof of Corollary \ref{corollary:resolution}]
We have to show that the long sequence 
\begin{align}
 \nonumber 0 &\ra \Ext(\twoG,\twoP) \stackrel{U}{\ra} \TORS(\twoG_\twoP) \stackrel{D_2^*}{\ra}  \TORS(\twoG_{\twoP^2}) \stackrel{D_3^*}{\ra} \TORS(\twoG_{\twoP^3})  \times \TORS(\twoG_{\twoP^2}) \stackrel{D_4^* }{\ra}...\\
 \nonumber ... &\stackrel{D_4^* }{\ra} \TORS(\twoG_{\twoP^4})  \times \TORS(\twoG_{\twoP^3})^2 \times \TORS(\twoG_{\twoP^2}) \times \TORS(\twoG_{\twoP})\stackrel{D_5^* }{\ra} ... \\ 
\nonumber ...  &\stackrel{D_5^* }{\ra}   \TORS(\twoG_{\twoP^5})  \times \TORS(\twoG_{\twoP^4})^3 \times \TORS(\twoG_{\twoP^3})^3 \times \TORS(\twoG_{\twoP^2})  \times \\
&\nonumber  \TORS(\twoG_{\twoP}) \times \TORS(\twoG_{\twoP^2}) \ra 0,
\end{align}
where $U$ is the forgetful functor  and $D_i^* $ denotes the pull-back via the differential operator $D_i$, is exact:

- Exactness in ${\Ext}(\twoP,\twoG)$: By Theorem \ref{thm:introexttor}, an object in $\Ext(\twoG,\twoP)$ is an object in $\Ext(\mathbb{Z}[\twoG],\twoP)$ with some extra structure. Therefore we can define $U$ as the 2-functor that sends an extension to itself and forgets the extra structure. Then the 2-functor $0 \ra \twoKer(U)$ is clearly essentially surjective and full.

- Exactness in $\TORS(\twoG_\twoP)$: We need to show that $ \tilde{U} :\Ext(\twoG,\twoP) \ra \twoKer (D_0^* )$ is essentially surjective and full. Let $\twoE$ be an object in $\Ext(\mathbb{Z}[\twoG],\twoP)$ whose pull-back via $D_2$ becomes trivial, i.e. $D_2^*\twoE$ is endowed with a trivialization $T$. Then $(\twoE, T)$ is an object in $\Psi_{\twoL.(\twoP)}(\twoG)$. By Proposition \ref{prop:ext-geom}, there exists an object $\twoE'$ in $\Ext(\twoG,\twoP)$ so that $\tilde{U}(\twoE')=\twoE$. This shows that $ \tilde{U}$ is essentially surjective.

- Exactness at the other terms follows from the 
free resolution of the Picard 2-stack computed in Corollary \ref{corollary:resolutionP}.
\end{proof}

\section{Example: Higher Extensions of Abelian Sheaves}\label{sec:example}

\subsection{The canonical free resolution $\twoL.(-)$ in the case of an abelian sheaf}

Here we take a closer look at the resolution $\twoL.(\twoP)$ given in Corollary \ref{corollary:resolutionP} when the Picard $\ES$-2-stack $\twoP$ is an abelian sheaf $P$. In this case, we denote $\twoL.(\twoP)$ by $L.(P)$. 

In \cite{EilenbergMaclane} Eilenberg and MacLane attach to any abelian group $G$ a complex
of free abelian groups $A(G)$. As explained in \cite{Breen69}, Eilenberg and MacLane's construction extends by functoriality to abelian sheaves. If $P$ is an abelian sheaf, the entries of the Eilenberg and MacLane's complex $A(P)$ in lower degrees are
\begin{align}
\nonumber A(P)_i&=0, \mathrm{~for~}i \leq 0;\\
\nonumber A(P)_1&=\mathbb{Z}[P];\\
\nonumber A(P)_2&=\mathbb{Z}[P^2];\\
\nonumber A(P)_3&=\mathbb{Z}[P^3]\oplus \mathbb{Z}[P^2];\\
\nonumber A(P)_4&=\mathbb{Z}[P^4]\oplus \mathbb{Z}[P^3]\oplus \mathbb{Z}[P^3] \oplus \mathbb{Z}[P^2];\\
\nonumber A(P)_5&=\mathbb{Z}[P^5]\oplus \mathbb{Z}[P^4]\oplus \mathbb{Z}[P^4] \oplus \mathbb{Z}[P^4]\oplus \mathbb{Z}[P^3]\oplus \mathbb{Z}[P^3] \oplus \mathbb{Z}[P^3]\oplus \mathbb{Z}[P^2]
\end{align}
where the differentials $\partial_i: A(P)_i \ra A(P)_{i-1}$ defined on the generators are
\begin{align}\label{complex_A(P)}
\partial_1&=0;\\[0.1cm]
\nonumber \partial_2[p |_1 q] &= [p+q] -[p]-[q];\\[0.1cm]
\nonumber \partial_3[p |_2 q] &= [p |_1 q] -[q |_1 p];\\[0.1cm]
\nonumber \partial_3[p |_1 q |_1r] &= [p+q |_1 r] -[p |_1 q+r]+[p |_1 q]-[q |_1 r];\\[0.1cm]
\nonumber \partial_4[p |_1 q |_1 r |_1 s] &= [p |_1 q |_1 r] +[p |_1 q+r |_1 s]+[q |_1 r |_1 s] -[p+q |_1 r |_1 s] -[p |_1 q |_1 r+s ];\\[0.1cm]
\nonumber \partial_4[p |_2 q |_1 r] &= [q |_1 r |_1 p] +[p |_2 q+r]+[p |_1 q |_1 r]-[q |_1 p |_1 r]-[p |_2 q]-[p |_2 r];\\[0.1cm]
\nonumber \partial_4[p |_1 q |_2 r] &= [p |_1 r |_1 q] +[p+q |_2 r]-[p  |_1 q |_1 r]-[r |_1 p |_1 q]-[p |_2 r]-[q |_2 r];\\[0.1cm]
\nonumber \partial_4[p |_3 q]&=-[p |_2 q]-[q |_2 p];\\[0.1cm]
\nonumber \partial_5[p |_1 q |_1 r |_1 s |_1 t] &=[q |_1 r |_1 s |_1 t]+[p |_1 q+r |_1 s |_1 t]+[p |_1 q |_1 r |_1 s+t] -[p |_1 q |_1 r+s |_1 t] \\[0.1cm]
\nonumber &-[p |_1 q |_1 r |_1 s]-[p+q |_1 r |_1 s |_1 t];\\[0.1cm]
\nonumber \partial_5[p |_2 q |_1 r |_1 s] &=[p |_1 q |_1 r |_1 s]+[p |_2 q |_1 r+s]+[p |_2 r |_1 s] -[q |_1 p |_1 r |_1 s]- [p |_2 q+r |_1 s] \\[0.1cm]
\nonumber &-[q |_1 r |_1 s |_1 p]+[q |_1 r |_1 p |_1 s]-[p |_2 q |_1 r];\\[0.1cm]
\nonumber \partial_5[p |_1 q |_1 r |_2 s ] &= -[p |_1 q |_1 r |_1 s]+[p+q  |_1 r |_2 s]+[p |_1 q |_1 s |_1 r]+[p |_1 q |_2 s]+[s |_1 p |_1 q |_1 r]\\[0.1cm]
\nonumber &-[p |_1 q+r |_2 s]-[p |_1 s |_1 q |_1 r]-[q |_1 r |_2 s];\\[0.1cm]
\nonumber \partial_5[p |_1 q |_2 r |_1 s ] &= [p+q |_2 r |_1 s]-[p |_2 r |_1 s]-[q |_2 r |_1 s]-[p |_1 q |_2 r+s]+[p |_1 q |_2 r]+[p |_1 q |_2 s]\\[0.1cm]
\nonumber & +[p |_1 q |_1 r |_1 s]+[p |_1 r |_1 s |_1 q]+[r |_1 s |_1 p |_1 q]+[r |_1 p |_1 q |_1 s]-[p |_1 r |_1 q |_1 s]\\[0.1cm]
\nonumber & -[r |_1 p |_1 s |_1 q];\\[0.1cm]
\nonumber \partial_5[p  |_3 q  |_1 r]&=[p  |_3 q+r]+[p  |_2 q  |_1 r]+[q  |_1 r  |_2 p]-[p  |_3 r]-[p  |_3 q];\\[0.1cm]
\nonumber \partial_5[p  |_1 q  |_3 r]&=[p+q  |_3 r]+[p  |_1 q  |_2 r]+[r  |_2 p  |_1 q]-[p  |_3 r]-[q  |_3 r];\\[0.1cm]
\nonumber \partial_5[p  |_2 q  |_2 r]&=[p  |_2 q  |_1 r]-[p  |_2 r  |_1 q]+[p  |_1 q  |_2 r]-[q  |_1 p  |_2 r];\\[0.1cm]
\nonumber \partial_5[p |_4 q]&=[p  |_3 q]-[q  |_3 p];
\end{align}

While they are not exactly the same, the complex $L.(P)$ and the complex $A(P)$ posses similarities. In fact, we observe that the entries of the complex $L.(P)$ and the entries of the complex $A(P)$ are the same in degrees 1, 2, and 3, as well as the differentials $D_2$, $D_3$ and $\partial_2$, $\partial_3$, respectively. However, the entries in degrees 4 and 5 of the complex $L.(P)$ contain some extra terms in addition to the terms of $A(P)$'s entries in degrees 4 and 5. To be more precise, there is one extra generator $[p]$ in degree 4 and a differential $D_4[p]$ associated to it and in degree 5 there are two extra generators $[p]$ and $[p|^4q]$ and two differential operators, $D_5[p]$ and $D_5[p|^4q]$. These extra generators and differentials arise from the strictness of the Picard condition. 

\subsection{Computation of $\eext^3(P,G)$ using the canonical free resolution $L.(P)$ of $P$}

To understand better the complex $L.(P)$, we examine $\eext^3(P,G)$ with $P$ and $G$ abelian sheaves. From \cite{Huebs1980}, it is known that $\eext^3(P,G)$ classifies Yoneda extensions of the form 
\begin{equation}\label{eqn:1}
\xymatrix@1{0\ar[r] & G \ar[r]^i & A \ar[r]^{\delta} & B \ar[r]^{\lambda} & C \ar[r]^{j} & P \ar[r] & 0}.
\end{equation}
Moreover, from \cite[Theorem 0.1]{beta}, it is also known that $\Ext^i(\twoP,\twoG) \cong \eext^i([\twoP]^{\flat\flat} ,[\twoG]^{\flat\flat}[1]) \cong \HH^{i+1}(\mathrm{R}\hhom([\twoP]^{\flat\flat} ,[\twoG]^{\flat\flat}))$. In case, $\twoP$ is the abelian sheaf $P$ and $\twoG$ is the shifted abelian sheaf $G[2]$, 
\begin{equation}\label{extensions_vs_cocycles}
\eext^3(P,G) \cong \HH^4(\mathrm{R}\hhom(P,G)).
\end{equation}

To calculate the element of $\HH^4(\mathrm{R}\hhom(P,G))$ which corresponds to the extension (\ref{eqn:1}) via the isomorphism (\ref{extensions_vs_cocycles}), we choose a hypercover $V.$ of the complex $L.(P)$ as follows: let  $U.. \ra L.(P)$ be a cover of $L.(P)$ given by the simplicial object $U..$ in the topos of sheaves on $\ES$ (see \cite{MR676809}). The pullback along $U.. \ra L.(P)$ is performed by refining the cover as we move along the complex $L.(P)$. Moving to the next degree on $L.(P)$ corresponds to a horizontal movement on $U..$ and therefore increases the first index of $U..$ by 1 whereas refining the cover corresponds to a vertical movement on $U..$ and therefore increases the second index of $U..$ by 1. That is, the pullback along $U.. \ra L.(P)$ follows the diagonal of $U..$ which is also a simplicial object in the topos of sheaves on $\ES$.  Thus, we let $V.$ to be the diagonal of $U..$. We denote by $p_i$ the pullback of $p$ along the face map $d_i$ of $V.$, i.e. $d_i^*p:=p_i$, by $p_{ij}$ the pullback of $p$ first along $d_j$ then along $d_i$, i.e. $d_i^*d_j^*p:=p_{ij}$, and so on for the further pullbacks.

We choose a set-theoretic cross section $s: P \ra C$ of the surjective sheaf morphism $j:C \ra P$, i.e $j \circ s =\id_P$. For any $p \in P(V_0)$, $s(p_0+ p_1)$ and $s(p_0)+s(p_1)$ are not necessarily equal in $C(V_1)$ as the sheaf map $s$ is not a homomorphism. The obstruction to $s$ being a sheaf homomorphism is measured by a sheaf map $f^{-1}:P \times P \ra B$ so that the relation
\begin{equation*}\label{eqn:2}
s(p_0+p_1)=s(p_0)+s(p_1)+\lambda(f^{-1}(p_0,p_1)),
\end{equation*}
is satisfied in $P(V_1)$. 

The pullback of the elements $p_0$ and $p_1$ in $P(V_1)$ to $V_2$ are the elements $p_{00}$, $p_{01}$, and $p_{11}$. Using the associativity of the addition of $P(V_2)$ and  $s((p_{00}+p_{01})+p_{11})=s(p_{00}+(p_{01}+p_{11}))$, we find that
\begin{equation*}\label{eqn:3}
f^{-1}(p_{01},p_{11})-f^{-1}(p_{00}+p_{01},p_{11})+f^{-1}(p_{00},p_{01}+p_{11})-f^{-1}(p_{00},p_{01}),
\end{equation*}
is in $\ker(\lambda)(V_2)$, which implies the existence of a sheaf map $f^{-2}:P^3 \ra A$ satisfying the relation
\begin{equation*}\label{eqn:4}
\delta(f^{-2}(p_{00},p_{01},p_{11}))=f^{-1}(p_{01},p_{11})-f^{-1}(p_{00}+p_{01},p_{11})+f^{-1}(p_{00},p_{01}+p_{11})-f^{-1}(p_{00},p_{01}),
\end{equation*}
in $B(V_2)$. As a consequence, $f^{-2}$ should be interpreted as an obstruction to the associativity. To find a coherence on $f^{-2}$, we pull $f^{-2}$ back to $V_3$ and observe that the expression
\begin{gather}
\begin{aligned}\label{eqn:5}
\nonumber&f^{-2}(p_{001},p_{011},p_{111})-f^{-2}(p_{000}+p_{001},p_{011},p_{111})+f^{-2}(p_{000},p_{001}+p_{011},p_{111})\\
\nonumber-&f^{-2}(p_{000},p_{001},p_{011}+p_{111})+f^{-2}(p_{000},p_{001},p_{011}),
\end{aligned}
\end{gather}
is in $\ker(\delta)(V_3)$, hence it exists a sheaf map $c:P^4 \ra G$ satisfying the relation
\begin{gather}
\begin{aligned}\label{eqn:6}
 \nonumber i(c(p_{000},p_{001},p_{011},p_{111}))=f^{-2}(p_{000},p_{001},p_{011})+f^{-2}(p_{000},p_{001}+p_{011},p_{111})\\
\nonumber +f^{-2}(p_{001},p_{011},p_{111})-f^{-2}(p_{000}+p_{001},p_{011},p_{111})-f^{-2}(p_{000},p_{001},p_{011}+p_{111})
\end{aligned}
\end{gather}
over $V_3$. When pulled back to $V_4$, the map $c$, seen as an obstruction to the coherence of the associativity, satisfies the relation
\begin{gather}
\begin{aligned}\label{cocycle:1}
&c(p_{0001},p_{0011},p_{0111},p_{1111})-c(p_{0000}+p_{0001},p_{0011},p_{0111},p_{1111})\\+&c(p_{0000},p_{0001}+p_{0011},p_{0111},p_{1111})-
c(p_{0000},p_{0001},p_{0011}+p_{0111},p_{1111})\\+&c(p_{0000},p_{0001},p_{0011},p_{0111}+p_{1111})-c(p_{0000},p_{0001},p_{0011},p_{0111})=0.
\end{aligned}
\end{gather}

After the associativity constraint, we involve the commutativity constraint in the discussion. The equality $s(p_0+p_1)=s(p_1+p_0)$ in $C(V_1)$ requires $f^{-1}(p_0,p_1)-f^{-1}(p_1,p_0)$ to be in $\ker(\lambda)(V_1)$ which implies the existence of a sheaf map $g^{-2}:P \times P \ra A$ satisfying the relation
\begin{equation}\label{eqn:7}
\delta(g^{-2}(p_0,p_1))=f^{-1}(p_0,p_1)-f^{-1}(p_1,p_0),
\end{equation}
in $B(V_1)$ from which it follows that $g^{-2}(p_0,p_1)+g^{-2}(p_1,p_0)$ is in $\ker(\delta)(V_1)$. Then there is a sheaf map $c':P \times P \ra G$ satisfying
\begin{equation}\label{eqn:8}
i(c'(p_0,p_1))=-(g^{-2}(p_0,p_1)+g^{-2}(p_1,p_0)).
\end{equation}
The injectivity of $i$ gives the relation
\begin{equation}\label{cocycle:2}
c'(p_0,p_1)-c'(p_1,p_0)=0.
\end{equation}
In case $p_0=p_1$ over $V_1$, from (\ref{eqn:7}) we find $\delta(g^{-2}(p_0,p_0))=0$. Hence, there exists $c'':P \ra G$ so that  $i(c''(p_0))=-g^{-2}(p_0,p_0)$ in $A(V_1)$ which implies with (\ref{eqn:8}) the relation
\begin{equation}\label{cocycle:2_bis}
2c''(p_0)=c'(p_0,p_0).
\end{equation}

Next, we explore the compatibility between the associativity and the commutativity constraints.  As the pullbacks $p_{00}$, $p_{01}$, $p_{11}$ of the elements $p_0$ and $p_1$ in $P(V_1)$ to $P(V_2)$ satisfy  $s((p_{00}+p_{01})+p_{11})=s(p_{00}+(p_{01}+p_{11}))=s((p_{01}+p_{11})+p_{00})=s(p_{01}+(p_{11}+p_{00}))=s(p_{01}+(p_{00}+p_{11}))=s((p_{01}+p_{00})+p_{11})$, we find that the expressions 
\begin{align}
\nonumber &\delta(g^{-2}(p_{00},p_{01})-g^{-2}(p_{00},p_{01}+p_{11})+g^{-2}(p_{00},p_{11})),\\
\nonumber  &\delta(f^{-2}(p_{00},p_{01},p_{11})-f^{-2}(p_{01},p_{00},p_{11})+f^{-2}(p_{01},p_{11},p_{00})),
\end{align}
are equal over $V_2$. Hence, there exists $c''':P^3 \ra G$ so that 
\begin{gather}
\begin{aligned}\label{eqn:9}
 i(c'''(p_{00},p_{01},p_{11}))&=f^{-2}(p_{01},p_{11},p_{00})+g^{-2}(p_{00},p_{01}+p_{11})+f^{-2}(p_{00},p_{01},p_{11})\\
&-f^{-2}(p_{01},p_{00},p_{11})-g^{-2}(p_{00},p_{01})-g^{-2}(p_{00},p_{11}).
\end{aligned}
\end{gather}
$c'''$ can be interpreted as an obstruction to the compatibility between the associativity and the commutativity constraints. The map $c'''$ can also be \textit{seen} as the difference between the two moves that send the element $p_{00}$ to the end of the ordered list of elements $(p_{00},p_{01},p_{11})$ in $P(V_2)$. One of the moves sends $p_{00}$ to the end of the list by moving it over $p_{01}+p_{11}$, whereas the other sends $p_{00}$ to the end of the list by moving it first over $p_{01}$, then over $p_{11}$. To find a coherence condition to this obstruction we pull (\ref{eqn:9}) back to $V_3$ and observe that $c'''$ satisfies the relation
\begin{gather}
\begin{aligned}\label{cocycle:3}
&c(p_{000},p_{001},p_{011},p_{111})+c'''(p_{000},p_{001},p_{011}+p_{111})+c'''(p_{000},p_{011},p_{111})\\-&c(p_{001},p_{000},p_{011},p_{111})-c'''(p_{000},p_{001}+p_{011},p_{111})-c(p_{001},p_{011},p_{111},p_{000})\\+&c(p_{001},p_{011},p_{000},p_{111})-c'''(p_{000},p_{001},p_{011})=0.
\end{aligned}
\end{gather}
We can also describe how to move the element $p_{11}$ to the beginning of the ordered list $(p_{00},p_{01},p_{11})$. Using $s(p_{00}+(p_{01}+p_{11}))=s((p_{00}+p_{01})+p_{11})=s(p_{11}+(p_{00}+p_{01}))=s((p_{11}+p_{00})+p_{01})=s((p_{00}+p_{11})+p_{01})=s(p_{00}+(p_{11}+p_{01}))$, we find that the expressions
\begin{align}
\nonumber &\delta(-g^{-2}(p_{01},p_{11})+g^{-2}(p_{00}+p_{01},p_{11})-g^{-2}(p_{00},p_{11})),\\
\nonumber  &\delta(f^{-2}(p_{00},p_{01},p_{11})-f^{-2}(p_{00},p_{11},p_{01})+f^{-2}(p_{11},p_{00},p_{01})),
\end{align}
are equal over $V_2$. Hence, there exists $c'''':P^3 \ra G$ so that 
\begin{gather}
\begin{aligned}\label{eqn:10}
i(c''''(p_{00},p_{01},p_{11}))&=f^{-2}(p_{00},p_{11},p_{01})+g^{-2}(p_{00}+p_{01},p_{11})-f^{-2}(p_{00},p_{01},p_{11})\\
&-f^{-2}(p_{11},p_{00},p_{01})-g^{-2}(p_{00},p_{11})-g^{-2}(p_{01},p_{11}).
\end{aligned}
\end{gather}
We interpret $c''''$ as an another obstruction to the compatibility between the associativity and the commutativity constraints. It can also be seen as the difference between the moves that send $p_{11}$ to the beginning of the list. Upon pulling (\ref{eqn:10}) back to $V_3$, we observe that $c''''$ satisfies the coherence condition
\begin{gather}
\begin{aligned}\label{cocycle:4}
-&c(p_{000},p_{001},p_{011},p_{111})+c'''(p_{000}+p_{001},p_{011},p_{111})+c(p_{000},p_{001},p_{111},p_{011})\\+&c'''(p_{000},p_{001},p_{111})+c(p_{111},p_{000},p_{001},p_{011})-c'''(p_{000},p_{001}+p_{011},p_{111})\\-&c(p_{000},p_{111},p_{001},p_{011})-c'''(p_{001},p_{011},p_{111})=0.
\end{aligned}
\end{gather}

As both $c'''$ and $c''''$ are obstructions to the compatibility between the associativity and the commutativity constraints, it is expected to have compatibility between them. First of all, on an ordered list of four elements $(p_{000},p_{001},p_{011},p_{111})$ in $P(V_3)$ obtained by pulling the elements $p_{00}$, $p_{01}$, and $p_{11}$ in $P(V_2)$ back to $V_3$, to obtain the order $(p_{011},p_{111},p_{000},p_{001})$, we can either move $p_{000}$ and $p_{001}$ to the end of the list or move $p_{011}$ and $p_{111}$ to the beginning of the list. The first compatibility condition between $c'''$ and $c''''$ says that both of these ways are the same. That is, using $s(p_{000}+p_{001}+p_{011}+p_{111})=s(p_{011}+p_{111}+p_{000}+p_{001})=s(p_{011}+p_{000}+p_{001}+p_{111})=s(p_{000}+p_{011}+p_{111}+p_{001})$ with all possible groupings we find the coherence condition
\begin{gather}
\begin{aligned}\label{cocycle:5}
&c'''(p_{000}+p_{001},p_{011},p_{111})-c'''(p_{000},p_{011},p_{111})-c'''(p_{001},p_{011},p_{111})\\
-&c''''(p_{000},p_{001},p_{011}+p_{111})+c''''(p_{000},p_{001},p_{011})+c''''(p_{000},p_{001},p_{111})\\
+&c(p_{000},p_{001},p_{011},p_{111})+c(p_{000},p_{011},p_{111},p_{001})+c(p_{011},p_{111},p_{000},p_{001})\\
+&c(p_{011},p_{000},p_{001},p_{111})-c(p_{000},p_{011},p_{001},p_{111})-c(p_{011},p_{000},p_{111},p_{001})=0.
\end{aligned}
\end{gather}
Secondly, in an ordered list of three elements $(p_{00},p_{01},p_{11})$ in $P(V_2)$, moving $p_{00}$ first to the end of the list and then back to the beginning of the list should be compatible with not moving $p_{00}$ at all. Therefore the difference between various ways of moving $p_{00}$ to the end of the list (i.e, $c'''(p_{00},p_{01},p_{11})$) and the difference between various ways of moving $p_{00}$ to the beginning of the list (i.e. $c''''(p_{01},p_{11},p_{00})$) should add up to zero. This translates into the coherence condition
\begin{gather}
\begin{aligned}\label{cocycle:6}
c'(p_{00},p_{01}+p_{11})+c'''(p_{00},p_{01},p_{11})+c''''(p_{01},p_{11},p_{00})-c'(p_{00},p_{01})-c'(p_{00},p_{11})=0.
\end{aligned}
\end{gather}
Moreover, the compatibility between moving $p_{11}$ first to the beginning of the list (i.e. $c''''(p_{00},p_{01},p_{11})$) and then to the end of the list (i.e. $c'''(p_{11},p_{00},p_{01})$) and not moving $p_{11}$ at all, translates into the coherence relation 
\begin{gather}
\begin{aligned}\label{cocycle:7}
c'(p_{00}+p_{01},p_{11})+c''''(p_{00},p_{01},p_{11})+c'''(p_{11},p_{00},p_{01})-c'(p_{00},p_{11})-c'(p_{01},p_{11})=0.
\end{aligned}
\end{gather}
The final coherence condition between $c'''$ and $c''''$ is given by the relation 
\begin{gather}
\begin{aligned}\label{cocycle:8}
c'''(p_{00},p_{01},p_{11})-c'''(p_{00},p_{11},p_{01})+c''''(p_{00},p_{01},p_{11})-c''''(p_{01},p_{00},p_{11})=0,
\end{aligned}
\end{gather}
and it describes how to interchange $p_{00}$ and $p_{11}$ in an ordered list of three elements $(p_{00},p_{01},p_{11})$.

There is one last coherence condition enjoyed by all the obstructions found so far. From the observation that for any $p$, $q$ in $P(V_0)$, $2(c''(p+q)-c''(p)-c''(q))$ is equal to $2(c(p,q,p,q)-c''''(p,q,p+q)-c'''(p,p,q)-c'''(q,p,q)+c'(q,p))$, we find the relation
\begin{equation}\label{cocycle:9}
-c(p,q,p,q)+c''''(p,q,p+q)+c'''(p,p,q)+c'''(q,p,q)-c'(q,p)+c''(p+q)-c''(p)-c''(q)=0.
\end{equation}

We can summarize the above calculations as follows: The collection of maps $(c,c',c'',c''',c'''')$ is in $\hhom(L^4[P],G)$. Since the maps $(c,c',c'',c''',c'''')$ satisfy the relations (\ref{cocycle:1}), (\ref{cocycle:2}), (\ref{cocycle:2_bis}), (\ref{cocycle:3}), (\ref{cocycle:4}), (\ref{cocycle:5}), (\ref{cocycle:6}), (\ref{cocycle:7}), (\ref{cocycle:8}), and  (\ref{cocycle:9}) they are in the kernel of $(D_5)^*:\hhom(L^4[P],G) \ra \hhom(L^5[P],G)$ induced by the differential $D_5:L^5[P]\ra L^4[P]$ of the complex $L.(P)$. This is why the collection $(c,c',c'',c''',c'''')$ can be called as a 4-cocycle of $P$ with values in $G$. Upon choosing another set-theoretic cross section $t:P \ra C$, we find another collection of maps $(d,d',d'',d''',d'''')$ satisfying the same relations as the collection $(c,c',c'',c''',c'''')$. We leave it to the reader to show that these two collections are cohomologous. In summary, these calculations show that we can use the complex $L.(P)$ to compute $\HH^4(\mathrm{R}\hhom(P,G))$, that is $\HH^4(\mathrm{R}\hhom(P,G))=\HH^4(\hhom(L.(P),G))$.

In the calculations, we find a map $c'':P \ra G$ satisfying the relation (\ref{cocycle:2_bis}) that does not appear in the complex $A(P)$. Therefore we add the differentials $D_4[p]=-[p|_2p]$ and $D_5[p]=2[p]-[p|_3p]$ to the fourth and the fifth entries of $A(P)$. The calculations also show that the collection $(c,c',c'',c''',c'''')$ shall satisfy the relation (\ref{cocycle:9}). This corresponds to adding another generator $[p|^4q]$ to the fifth entry of $A(P)$ to kill the class $-[p |_1 q |_1 p |_1 q]+[p |_1 q |_2 p+q] + [p |_2 p |_1 q]+[q |_2 p |_1 q] - [q |_3 p] + [p + q]-[p]-[q]$, that is we shall add another differential $D_5[p|^4q]=-[p |_1 q |_1 p |_1 q]+[p |_1 q |_2 p+q] + [p |_2 p |_1 q]+[q |_2 p |_1 q] - [q |_3 p] + [p + q]-[p]-[q]$. The addition of these differentials turns the complex $A(P)$ in to the complex $L.(P)$. Hence, we can use $L.(P)$ as a free resolution of $P$ to compute $\eext^3(P,G)$.

We finish this calculation section by pointing out that the above discussion has another half which is not mentioned here. It is the reconstruction of the extension (\ref{eqn:1}) from the cocyclic description of the $G[2]$-torsor over $P$ which would have required a descent type argument over the complex $L.(P)$. The details of such a reconstruction will be the subject of a forthcoming paper where the $\twoG_{\twoP}$-torsors over $\twoP$ and the extensions of $\twoP$ by $\twoG$ will be studied in terms of cocycles.

\subsection{An algebraic point of view concerning the strict Picard condition}

Algebraically, adding the differentials $D_4[p]$, $D_5[p]$, and $D_5[p|^4q]$ arising from the strict Picard condition to the complex $A(P)$ can be described as follows: consider $\alpha:\mathbb{Z}[P] \oplus \mathbb{Z}[P^2] \ra \mathbb{Z}[P]$, defined by $\alpha[p]=-2[p]$ and $\alpha[p|^4 q]=[p]+[q]-[p+q]$, as the chain complex whose entries at degrees 3 and 4 are $\mathbb{Z}[P]$ and $\mathbb{Z}[P] \oplus \mathbb{Z}[P^2]$, respectively and all other entries are 0. We define the morphism $f$
\begin{equation}\label{chain_complex}
\begin{tabular}{c}
\xymatrix{\hdots\ar[r]& 0 \ar[r] \ar[d]& \mathbb{Z}[P] \oplus \mathbb{Z}[P^2] \ar[r]^(0.6){\alpha} \ar[d]^{f_4} & \mathbb{Z}[P] \ar[r] \ar[d]^{f_3} & 0 \ar[d] \ar[r] & \hdots\\
\hdots \ar[r] &A(P)_5 \ar[r]_{\partial_5} & A(P)_4 \ar[r]_{\partial_4} & A(P)_3 \ar[r]_(0.6){\partial_3} &A(P)_2\ar[r] & \hdots}
\end{tabular}
\end{equation}
where $f_i=0$ for $i \neq 3,4$ and $f_3[p]=[p|_2p]$, $f_4[p]=[p|_3p]$, and $f_4[p|^4q]=[p |_1 q |_1 p |_1 q]-[p |_1 q |_2 p+q] - [p |_2 p |_1 q] - [q |_2 p |_1 q] + [q |_3 p]$. It is straightforward to observe that the cone of (\ref{chain_complex}) is the complex $L.(P)$.

%--------------------------------------------------------------------------------------------
\bibliographystyle{plain}

\begin{thebibliography}{10}


\bibitem{MR2995663}
C. Bertolin.
\newblock Biextensions of {P}icard stacks and their homological interpretation.
\newblock {\em Adv. Math.} 233 p. 1--39, 2013.


\bibitem{bega}
C. Bertolin and F. Galluzzi.
\newblock The Brauer group of 1-motives.
\newblock arXiv:1705.01382, 2017.


\bibitem{beta}
C. Bertolin and A.~E. Tatar.
\newblock Extensions of {P}icard 2-stacks and the cohomology groups
  {$\mathrm{Ext}^i$} of length 3 complexes.
\newblock {\em Ann. Mat. Pura Appl. (4)} 193(1) p. 291--315, 2014.

\bibitem{Breen}
L. Breen.
\newblock On some extensions of abelian sheaves in dimensions two and three.
\newblock Unpublished Notes.

\bibitem{MR1086889}
L. Breen.
\newblock Bitorseurs et cohomologie non ab\'elienne.
\newblock In {\em The {G}rothendieck {F}estschrift, {V}ol.\ {I}}, volume~86 of
  {\em Progr. Math.} p. 401--476. Birkh\"auser Boston, Boston, MA, 1990.

\bibitem{MR1191733}
L. Breen.
\newblock Th\'eorie de {S}chreier sup\'erieure.
\newblock {\em Ann. Sci. \'Ecole Norm. Sup. (4)} 25(5) p. 465--514, 1992.

\bibitem{MR1301844}
L. Breen.
\newblock On the classification of {$2$}-gerbes and {$2$}-stacks.
\newblock {\em Ast\'erisque} 225 p. 160, 1994.

\bibitem{Breen69}
L. Breen.
\newblock On a nontrivial higher extension of representable abelian sheaves.
\newblock {\em Bull. Amer. Math. Soc.} 75 p. 1249--1253, 1969.


\bibitem{Brochard09}
S. Brochard.
\newblock Foncteur de {P}icard d'un champ alg\'ebrique.
\newblock {\em Math. Ann.} 343(3) p. 541--602, 2009.


\bibitem{EilenbergMaclane}
S. Eilenberg and S. MacLane.
\newblock Cohomology theory of {A}belian groups and homotopy theory. {II}.
\newblock {\em Proc. Nat. Acad. Sci. U. S. A.} 36 p. 657--663, 1950.

\bibitem{MR676809}
Eric~M. Friedlander.
\newblock {\em \'Etale homotopy of simplicial schemes}.
\newblock {\em Annals of Mathematics Studies}, vol. 104, Princeton University Press, Princeton, N.J.; University of Tokyo
  Press, Tokyo, 1982.

\bibitem{SGA7}
A.~Grothendieck et~al.
\newblock {\em Groupes de Monodromie en G\'eom\'etrie Alg\'ebrique}, SGA 7 I.
\newblock Lecture Notes in Mathematics, Vol. 288. Springer-Verlag, Berlin-New York, 1972.


\bibitem{SGA4}
A.~Grothendieck and J.~L. Verdier.
\newblock Th\'eorie des topos et cohomologie \'etale des sch\'emas, SGA 4 I.
\newblock  Lecture Notes in Mathematics, Vol. 269. Springer-Verlag, Berlin, 1972.


\bibitem{Huebs1980}
J. Huebschmann.
\newblock Crossed n-fold extensions of groups and cohomology.
\newblock {\em Commentarii Mathematici Helvetici} 55(1) p. 302--313, 1980.


\bibitem{Illusie}
L. Illusie
\newblock Complexe cotangent et d\'eformations. II.
\newblock Lecture Notes in Mathematics, Vol. 283. Springer-Verlag, Berlin-New York, 1972. 


\bibitem{MR2342833}
A. Joyal and J. Kock.
\newblock Coherence for weak units.
\newblock {\em Doc. Math.} 18 p.71--110, 2013

\bibitem{MR1278735}
M.~M. Kapranov and V.~A. Voevodsky.
\newblock {$2$}-categories and {Z}amolodchikov tetrahedra equations.
\newblock In {\em Algebraic groups and their generalizations: quantum and
  infinite-dimensional methods ({U}niversity {P}ark, {PA}, 1991)}, volume~56 of
  {\em Proc. Sympos. Pure Math.}, p. 177--259. Amer. Math. Soc., Providence,
  RI, 1994.

\bibitem{Mumford65}
D. Mumford.
\newblock Picard groups of moduli problems.
\newblock In {\em Arithmetical {A}lgebraic {G}eometry ({P}roc. {C}onf. {P}urdue
  {U}niv., 1963)}, p. 33--81. Harper \& Row, New York, 1965.


\bibitem{Rousseau}
A. Rousseau.
\newblock Bicat\'egories mono\"{\i}dales et extensions de gr-cat\'egories.
\newblock {\em Homology, Homotopy and Applications} 5(1) p. 437--547, 2003.


\bibitem{MR2735751}
A.~E. Tatar.
\newblock Length 3 complexes of abelian sheaves and {P}icard 2-stacks.
\newblock {\em Adv. Math.} 226(1) p. 62--110, 2011.

\end{thebibliography}

\end{document}